\documentclass[10pt]{amsart}
\usepackage{latexsym, amsmath, amssymb,amsthm,amsopn,amsfonts,amscd}
\usepackage{version}
\usepackage{epsfig,graphics,color,graphicx,graphpap}

\newcommand{\CC}{{\mathbb C}}

\newcommand{\RR}{{\mathbb R}}

\newcommand{\HH}{{\mathcal H}}

\newcommand{\cN}{{\mathcal N}}

\renewcommand{\Re}{\mathop{\rm Re}\nolimits}
\renewcommand{\Im}{\mathop{\rm Im}\nolimits}

\newcommand{\nn}{\nonumber}
\newcommand{\dd}{\tau}

\theoremstyle{plain}

\newtheorem{thm}{Theorem}
\newtheorem{prop}{Proposition}[section]

\newtheorem{lem}[prop]{Lemma}

\theoremstyle{definition}

\newtheorem{rem}{Remark}[section]

\numberwithin{equation}{section}

\def\squarebox#1{\hbox to #1{\hfill\vbox to #1{\vfill}}}

\newcommand{\la}{\langle}
\newcommand{\ra}{\rangle}

\newcommand{\p}{\partial}
\usepackage{amsxtra}

\usepackage{fancyhdr}
\ifx\pdfoutput\undefined
  \DeclareGraphicsExtensions{.eps}
\else
  \ifx\pdfoutput\relax
    \DeclareGraphicsExtensions{.eps}
  \else
    \ifnum\pdfoutput>0
      \DeclareGraphicsExtensions{.pdf}
    \else
      \DeclareGraphicsExtensions{.eps}
    \fi
  \fi
\fi

\title
[NLS / GP with double well potentials]
{Long time dynamics near the symmetry breaking bifurcation for Nonlinear Schr\"odinger/Gross-Pitaevskii Equations}

\author[J.L. Marzuola]{Jeremy L. Marzuola}

\author[M.I. Weinstein]{Michael I. Weinstein}

\address{Department  of Applied Physics and Applied Mathematics, Columbia University \\
200 S. W. Mudd, 500 W. 120th St., New York City, NY 10027, USA}
\email{jm3058@columbia.edu, miw2103@columbia.edu}

\begin{document}    

\begin{abstract} We consider a class  nonlinear Schr\"odinger / Gross-Pitaevskii equations (NLS/GP) with a focusing (attractive) nonlinear potential and symmetric double well linear potential. NLS/GP  plays a central role in the modeling of nonlinear optical and mean-field quantum many-body phenomena. It is known that there is a  critical $L^2$ norm (optical power / particle number) at which there is a symmetry breaking bifurcation of the ground state. 
We study the rich dynamical behavior near the symmetry breaking point.
The source of this behavior in the full Hamiltonian PDE is related  to the dynamics of a finite-dimensional Hamiltonian reduction.
We derive this reduction, analyze a part of its phase space and  prove a {\it shadowing theorem} on the persistence of solutions, with oscillating mass-transport between wells,  on very long, but finite, time scales within the full NLS/GP. The infinite time dynamics for NLS/GP are expected to depart, from the finite dimensional reduction, due to resonant coupling of discrete and  continuum / radiation modes.
 \end{abstract}

\maketitle

\section{Introduction and Outline}
\label{sec:Intro-and-Outline}

The  cubic nonlinear Schr\"odinger / Gross-Pitaevskii  (NLS / GP) equation
\begin{eqnarray}
\label{eqn:nlsdwp}
i \partial_t u =  (- \Delta + V(x)) u\ - \left| u\right|^2 u
\end{eqnarray}
plays a central role in the mathematical description of nonlinear optical and quantum many-body phenomena. 
  In the context of nonlinear optics, $u$ denotes the slowly varying envelope of a nearly monochromatic  electromagnetic field propagating in a wave-guide, $t$ the distance along the wave-guide and $x\in \RR^2$ the dimensions transverse to the wave guide \cite{Moloney-Newell,Boyd}.  At low intensity, light is guided to a higher refractive index region, corresponding to a potential well $V(x)$. The {\it Kerr} nonlinear effect gives rise to an increase in the refractive index in regions of higher intensity, $|u|^2$, and therefore a ``deeper'' effective potential well $V(x)-|u|^2$. In  the context of quantum many-body physics, NLS/GP emerges 
in the mean-field limit of many weakly interacting identical quantum particles obeying Bose statistics, as the number of particles tends to infinity 
\cite{PS:03,ESY,RS} and other recent works. The potential $V(x)$ governs the confining trap for the quantum particles.\footnote{NLS / GP falls into a larger class of models:
\begin{eqnarray}
\label{eqn:nlsdwp-gK}
i \partial_t u =  (- \Delta + V(x)) u\ +\ gK\left[\left|u\right|^2\right] u,
\end{eqnarray}
allowing for more general nonlinear terms, {\it e.g.} more general
functions of $|u|^2$ or $K$ nonlocal; see, for example,
\cite{KKSW}. For $g < 0$ (focusing case)  our analysis yields very
similar results to those obtained above. For $g>0$ (defocusing) our
methods  can be adapted to prove a shadowing result for trajectories
of the finite dimensional reduction which arises.}

In this paper we focus on a class of symmetric 
{\it double-well} potentials.  A model to keep in mind is the two parameter family of symmetric double well potentials:
\begin{eqnarray}
\label{dw-gauss}
V(x) = V(x;\sigma,L)= - \left[ \frac{1}{\sqrt{4 \pi \sigma^2}} e^{-\frac{(x-L)^2}{4 \sigma^2}} + \frac{1}{\sqrt{4 \pi \sigma^2}} e^{-\frac{(x+L)^2}{4 \sigma^2}} \right], 
\end{eqnarray}
which converges as $\sigma\to0^+$ to a double-delta  well at $\pm L$.  In figure  \ref{fig:pot} the case $L=3,\ \sigma =1$ is shown.

\begin{figure}
\includegraphics[scale=0.4]{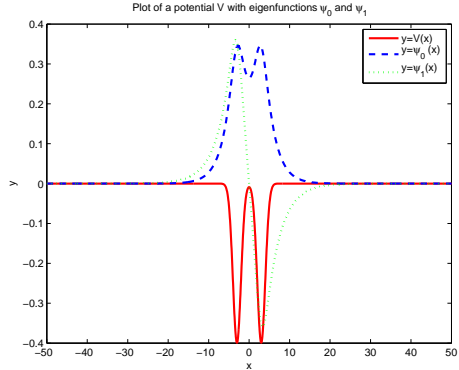} 
\caption{
Double-well potential, $V(x;3,1)$.  Superposed are the  two eigenstates, even ground ($\psi_0$) and odd excited  ($\psi_1$) states, of the Schr\"odinger operator, $-\partial_x^2+V(x;3,1)$}
\label{fig:pot}
\end{figure}

Double-well potentials are of particular interest in optics as models of coupled parallel wave guides channeling light, which interacts through through evanescent tails. In the quantum context, these are natural and simple systems in which to study quantum tunneling. As discussed below, the combined effects of a confining double-well potential
 with focusing cubic nonlinearity lead to the phenomenon of spontaneous {\it symmetry breaking of the ground state} at sufficiently high optical power 
  or particle number. See, for example,  \cite{YL:99,Oberthaler-etal:05,KCMFW} for experimental studies of symmetry breaking.  {\it   Our goal in this
    paper is to investigate the phase space dynamics  of NLS/GP   near
    the symmetry breaking point.}    

The equation NLS/GP, \eqref{eqn:nlsdwp-gK}, is a Hamiltonian system  and expressible in the form: 
\begin{equation}
i\partial_t u\ =\ \frac{ \delta{\HH}  }{\delta u^*}\ .
\label{Ham-u}
\end{equation}
 $\HH$ denotes the conserved Hamiltonian energy functional:
\begin{equation}
\HH[u]\ =\ \int \left( | \nabla u |^2 + V |u|^2 - \frac14 |u|^4 \right) dx.
\nonumber\end{equation}
The conserved squared $L^2$ norm (particle number / optical power)
 is denoted:
\begin{equation}
\cN[u]\ =\ \int |u|^2 dx.
\label{NNdef}
\end{equation}

We  are interested in the dynamics near special classes of {\it nonlinear bound states} of NLS/GP. Nonlinear bound states are solutions of the form
\begin{equation}
u(x,t)\ =\ e^{-i\Omega t}\ \Psi_\Omega(x) ,
\nn\end{equation}
where $\Psi_\Omega$ is spatially localized: 
\begin{equation}
 (- \Delta + V(x)) \Psi_\Omega\ -\left| \Psi_\Omega\right|^2 \Psi_\Omega\ =\ \Omega\ \Psi_\Omega,\ \ \ \Psi_\Omega\in H^1(\RR).\label{nl-elliptic} 
\end{equation}

Consider first the linear case of the linear eigenvalue problem:
\begin{equation}
(- \Delta + V(x)) \Psi_\Omega\ =\ \Omega\ \Psi_\Omega,\ \ \ \Psi_\Omega\in H^1(\RR).\label{evp} 
\end{equation}
 
In this case, there is a least energy {\it ground state}, $\psi_0$ with corresponding simple eigenvalue $\Omega_0$ \cite{RSv4}. If the separation between wells is sufficiently large, then the ground state eigenfunction is a {\it bimodal} positive symmetric state, which reflects the discrete symmetry of the potential \cite{EH,Simon,JW}; see figure \ref{fig:pot}.  In addition, there is an anti-symmetric (odd) state, $\psi_1$ with energy $\Omega_1$, such that $\Omega_0<\Omega_1<0$. 
 
In the attractive / focusing  {\it nonlinear} case, \eqref{eqn:nlsdwp},
the character of solutions, and the solution set varies with the solution norm. Indeed, if we  consider the set of solutions of \eqref{nl-elliptic}  on the level set 
\begin{equation}
\int |\Psi_\Omega|^2 = \cN,
\label{constraintN}
\end{equation}
we find that  for large enough well-separation, there is a {\it symmetry breaking threshold} $\cN_{cr}$ \cite{KKSW}; see figure\ref{fig:sols}:
 
\begin{enumerate}
\item If $\cN<\cN_{cr}$ there is a unique positive, symmetric  and bimodal state. 
\item For $\cN>\cN_{cr}$, (modulo phase) there are three positive localized states: a symmetric state (which exists for all $\cN>0$) and  two are {\it asymmetric states}, biased respectively to the right and left wells.   
\item As $\cN$ increases beyond $\cN_{cr}$ this symmetry broken state becomes increasingly concentrated in one of the wells \cite{AFGST,KKSW}. 
\item The  symmetric (bimodal) state is dynamically stable for $\cN< \cN_{cr}$ and unstable for  $\cN>\cN_{cr}$. For $\cN>\cN_{cr}$ the asymmetric states are stable.
\end{enumerate}
 That symmetry breaking occurs at {\it sufficiently large} values of $\cN$ was studied variationally in \cite{AFGST} for the nonlinear Hartree equation. Their  method can be adapted to a large class of equations, for which a ground state can be realized as a minimizer of a Hamiltonian, ${\mathcal H}$, subject to fixed ${\mathcal N}$.  The above detailed portrait of the symmetry breaking transition and the exchange of stability among branches was studied in detail in \cite{KKSW}. In particular, for large well-separation, $L$,  
 \begin{equation}
 \cN_{cr}(L)={\mathcal O}\left(\Omega_1(L)-\Omega_0(L)\right)={\mathcal O}(e^{-\kappa L}),\ \ \kappa>0.
 \label{Ncr-est}
 \end{equation}
\begin{figure}
\includegraphics[scale=0.4]{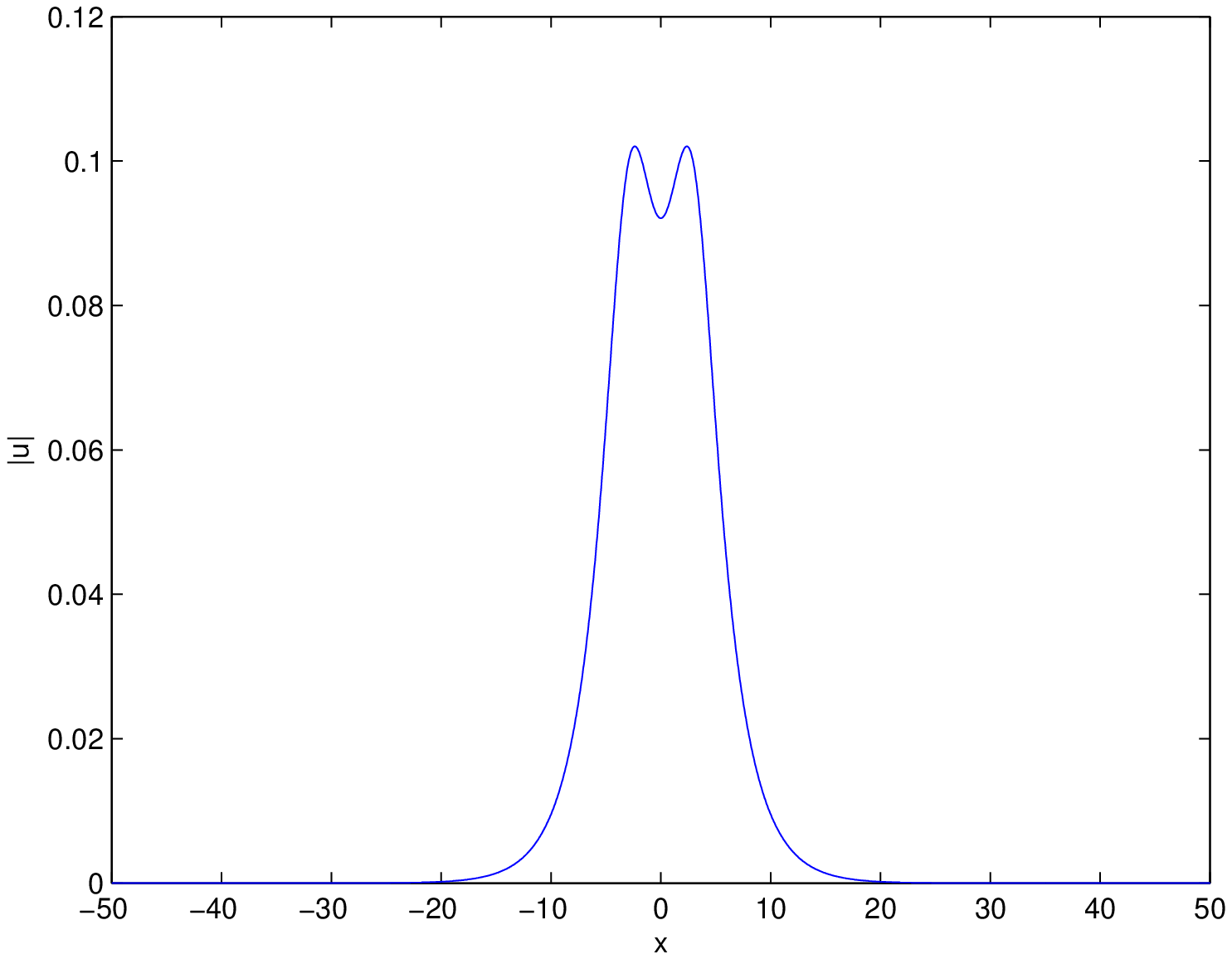} 
\includegraphics[scale=0.4]{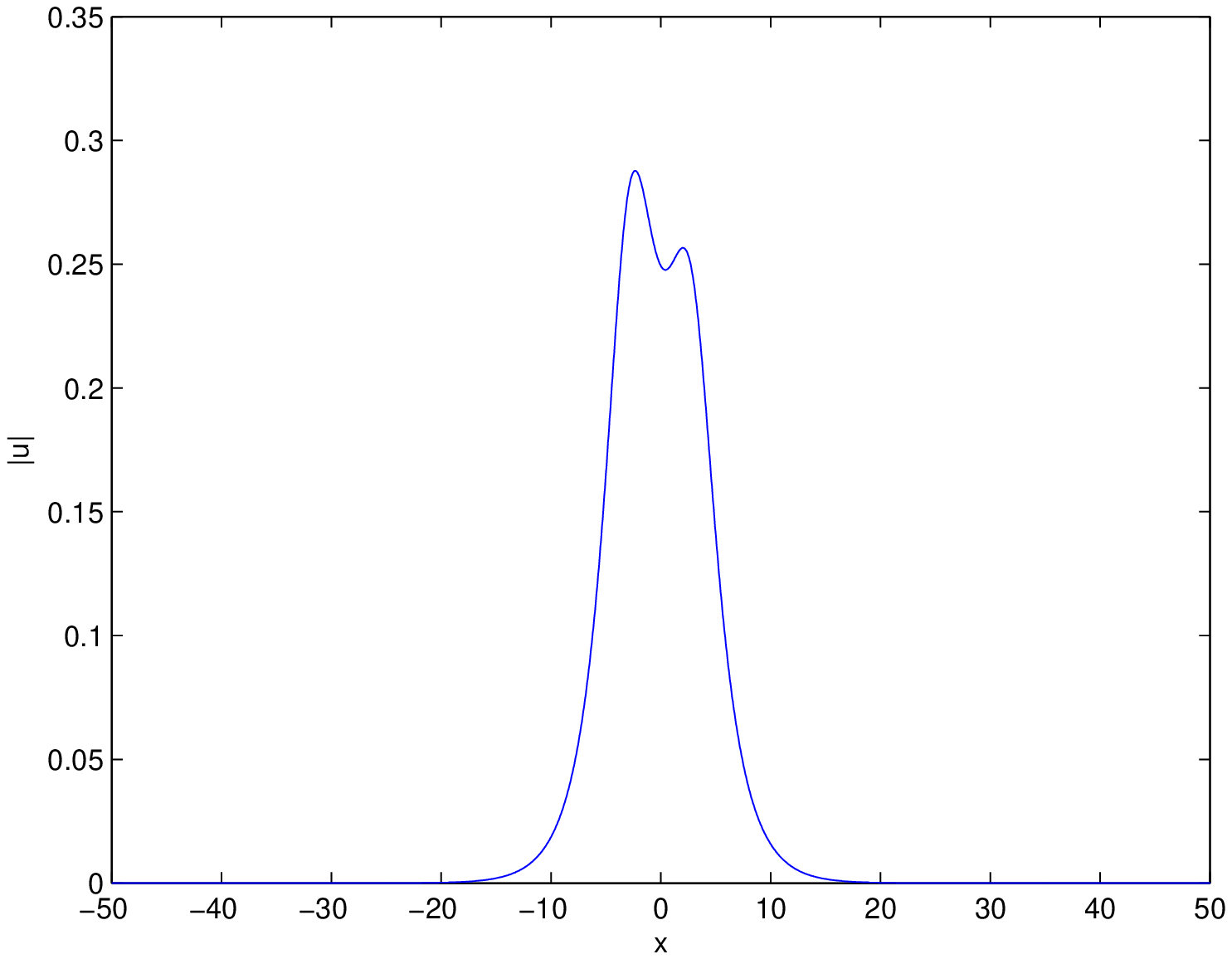} 
\caption{A symmetric (bimodal) soliton and an asymmetric soliton, which exists only for $\cN>\cN_{cr}$.}
\label{fig:sols}
\end{figure}
Figure \ref{fig8} displays a numerical computation of the symmetry-broken state, occurring for $\cN>\cN_{cr}$, where $V(x)$ is of the type displayed  (\ref{dw-gauss}).   
\begin{figure}
\begin{center}
\includegraphics[scale=0.55]{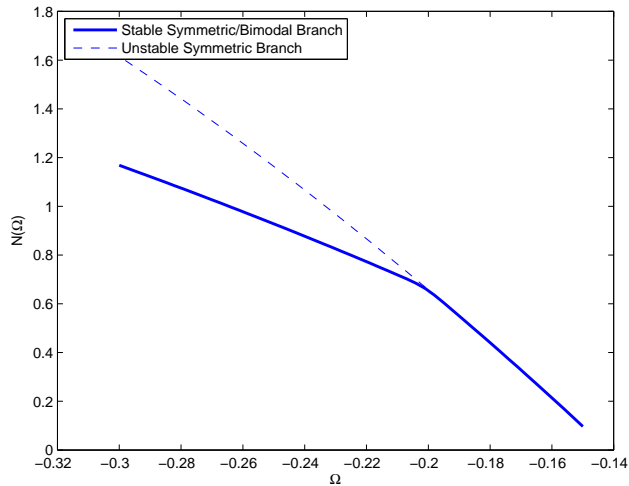} 
\end{center}
\caption{A numerical plot of the symmetry breaking in the soliton curve.  Specifically, we plot here $N$ versus $\Omega$ for a double- (Gaussian) well potential with separation parameter $L = 3$ and $\sigma = 1$.  Note the stable curve is invariant under reflection about the $y$ axis, hence the slightly thicker curve represents the resulting degeneracy for the asymmetric states.}
\label{fig8}
\end{figure}

Our goal is to explore the {\it detailed general dynamics} near the symmetry breaking transition. Toward a formulation of precise results, we first introduce a class of double-well potentials in  dimension one \footnote{The results of this paper can be extended to higher dimensions.}. Following \cite{EH}, start with a single rapidly decaying potential well centered at $0$, $V_0(x)$, for which the Schr\"odiner operator $H_0 = -\partial_x^2 + V_0$ has exactly one (simple) eigenvalue $\omega$.  Then, construct a double well potential
\begin{eqnarray}
V_L(x) = V_0 (x-L) + V_0 (x+L),\ \ L>0
\end{eqnarray}
and define the Schr\"odinger operator
\begin{equation}
H_L = -\partial_x^2 + V_L.
\label{HLdef}
\end{equation}
  There exists $L_0>0$ such that for $L>L_0$, $H_L$ has a pair of simple eigenvalues, $\Omega_0 = \Omega_0 (L)$ and $\Omega_1 = \Omega_1 (L)$ and corresponding eigenfunctions $\psi_0$ (even) and $\psi_1$ (odd):
  \begin{align}
&H \psi_j = \Omega_j \psi_j, \ j = 0,1;\ \ \ \psi_j\in L^2\nn\\
&\Omega_0 < \Omega_1 < 0. \nn
\end{align}
As noted above, the symmetry breaking threshold, $\cN_{cr}(L)$ 
 (equation \eqref{Ncr-est})  is exponentially small for large well-separation.
Therefore, to study the dynamics in a neighborhood of the symmetry
breaking point,  it is natural use coordinates associated with the
{\it linear} operator $H_0=-\partial_x^2+V$.  Throughout this result,
we will assume that $V$ is such that $\psi_0$
and $\psi_1$ are the only discrete eigenfunctions of $H_0$ and in
addition that $V$ is sufficiently smooth and decaying as defined in \cite{W}
and discussed in Appendix \ref{sec:est} as to guarantee
dispersive estimates to do perturbation theory.  We note that these
assumptions will be satisfied by double delta wells, for which we
have done our numerical computations.  These ideas will be explored further in a forthcoming note \cite{DMW}.

Thus, we expand the solution as follows:
\begin{align}
&u(x,t)\  =\ c_0(t) \psi_0(x) + c_1(t) \psi_1(x) + R (t,x) ,\nn\\
&\la \psi_j, R(\cdot,t)\ \ra\ =\ 0,\ \ \ j=0,1.\label{decomp}
\end{align}
By orthogonality, we have
\begin{equation}
\cN[u(\cdot,t)]\ =\ |c_0(t)|^2\ +\ |c_1(t)|^2\ +\ \int\ |R(x,t)|^2\
dx\ =\ \cN[u(\cdot,0)].
\label{totalenergy}
\end{equation}
  
Referring to this decomposition, we now give an overview of the paper:
 \begin{enumerate}
 \item In section \ref{sec:oscillator-field} we express the NLS / GP as an equivalent dynamical system governing $c_0(t)$, $c_1(t)$ and $R(x,t)$. This Hamiltonian system, equivalent to NLS / GP, has the form of two equations governing discrete nonlinear ``oscillators'', coupled to an equation for a wave field, $R(x,t)$.
 \item In section \ref{sec:finite-dim} we study the finite dimensional reduction governing $c_0$ and $c_1$, obtained by dropping all terms which couple the oscillator and field variables.   This reduction is a finite dimensional Hamiltonian system with conserved Hamiltonian:
\begin{equation}
H(c_0,\bar{c_0},c_1,\bar{c_1}),
\nn\end{equation}
given in \eqref{Hdef}, and $l^2$ invariant 
\begin{eqnarray*}
N[c_0,c_1]=|c_0|^2+|c_1|^2.  
\end{eqnarray*}

\begin{rem}  
Note also that the reduction above is what we get if we make the
Ansatz \eqref{decomp} into the Lagrangian of NLS/GP, that is restrict
${\mathcal L}_{NLS-GP}$ to the bound state manifold, and obtain the
equations of motion as in \cite{HZ}.
\end{rem}
 
Use of these two invariants facilitates an analysis of the finite dimensional phase space; trajectories with prescribed values of  $H$ and $N$ lie on a two-dimensional surface.  We then discuss some of the rich dynamics of this finite dimensional reduction and our goal is to  prove their persistence, for non-trivial time scales , for with the full PDE, NLS/GP. \bigskip
 
For the dynamics on level set $N\sim N_{cr}^{FD} \approx N_{cr}$, we establish the following behavior:\\ 
  \begin{enumerate}
 \item $N<N^{FD}_{cr}:$\ There is an elliptic fixed point, corresponding to the stable symmetric state of NLS/GP for $N<N_{cr}$.
 A neighorhood of this fixed point is foliated by stable time-periodic solutions. 
   \item  $N>N^{FD}_{cr}:$\ The fixed point for $N<N^{FD}_{cr}$ persists, but transitions from being a stable  elliptic point to an unstable saddle. This corresponds to the unstable symmetric-bimodal state for $\cN>\cN_{cr}$. At $N=N^{FD}_{cr}$, two new elliptic equilibria bifurcate and, a neighborhood of each is  foliated by stable time-periodic solutions. These stable time-periodic oscillations correspond to stable oscillations around the stable asymmetric standing waves for 
 $\cN>\cN_{cr}$.
 \item  $N>N^{FD}_{cr}:$\ There are periodic solutions, outside a separatrix, which encircle both new equilibria. In the physical configuration space, these correspond to soliton transport from one well to the other and back, continuing periodically. Furthermore, one can quantify the ``energy barrier'' that must be exceeded to dislodge a ``soliton'' from localization about one of the wells.  Energy thresholds, such as that described above play an important role in transport of energy in inhomogeneous and discrete systems. 
Natural directions to pursue beyond this work are transport in systems  with many wells and  the {\it Peireles-Nabbaro barrier} for motion of localized coherent structures discrete lattice systems.
 \end{enumerate}
\bigskip 

 \item In section \ref{sec:persist} we prove that {\it certain} phase space structures for the finite dimensional dynamical systems, persist in character for the full NLS/GP system, {\it on long, but finite time scales}. Stated nontechnically,\\ \\
 
 \noindent{\bf Theorem:}\ For any sufficiently small amplitude periodic solution about equilibrium an equilibrium state of the finite dimensional reduction ($N$ above or below the bifurcation threshold), there is a solution of the PDE (NLS/GP), whose projection into the finite dimensional phase space, shadows this finite dimension orbit on very long time scales.\\ 
 
 A precise statement is given in  Theorem \ref{thm:main-eq}; see also Figure \ref{fig:schem}. 
 The time scales on which these results hold enable us to see nearly periodic  oscillations for the PDE on long time scales, i.e. through many, many oscillations, but not on an infinite time scale. Indeed, we do not expect the persistence of such oscillations on infinite time scales due to  nonlinear coupling of bound to to radiation modes for the full system,  for example, \cite{SofWe:99,GW}.  We conjecture that in the infinite time limit, the soliton executes a (radiation) damped oscillation to some stable nonlinear bound state, corresponding to the damped oscillatory decay to a stable equilibrium of the finite dimensional reduction. Evidence is presented in section \ref{sec:num}, where
 numerical simulations are discussed.

 We remark that  our current theorem does not apply to perturbations of general ``large'' periodic orbits of the finite dimensional reduction.  More detailed information on the Floquet theory of the  linearized equations about general periodic orbits  is still needed. This is currently being investigated.\\
 
  \item In section \ref{sec:conclusion} we provide a summary and discussion of open problems.
 \end{enumerate}
\begin{figure}
\includegraphics[scale=0.25]{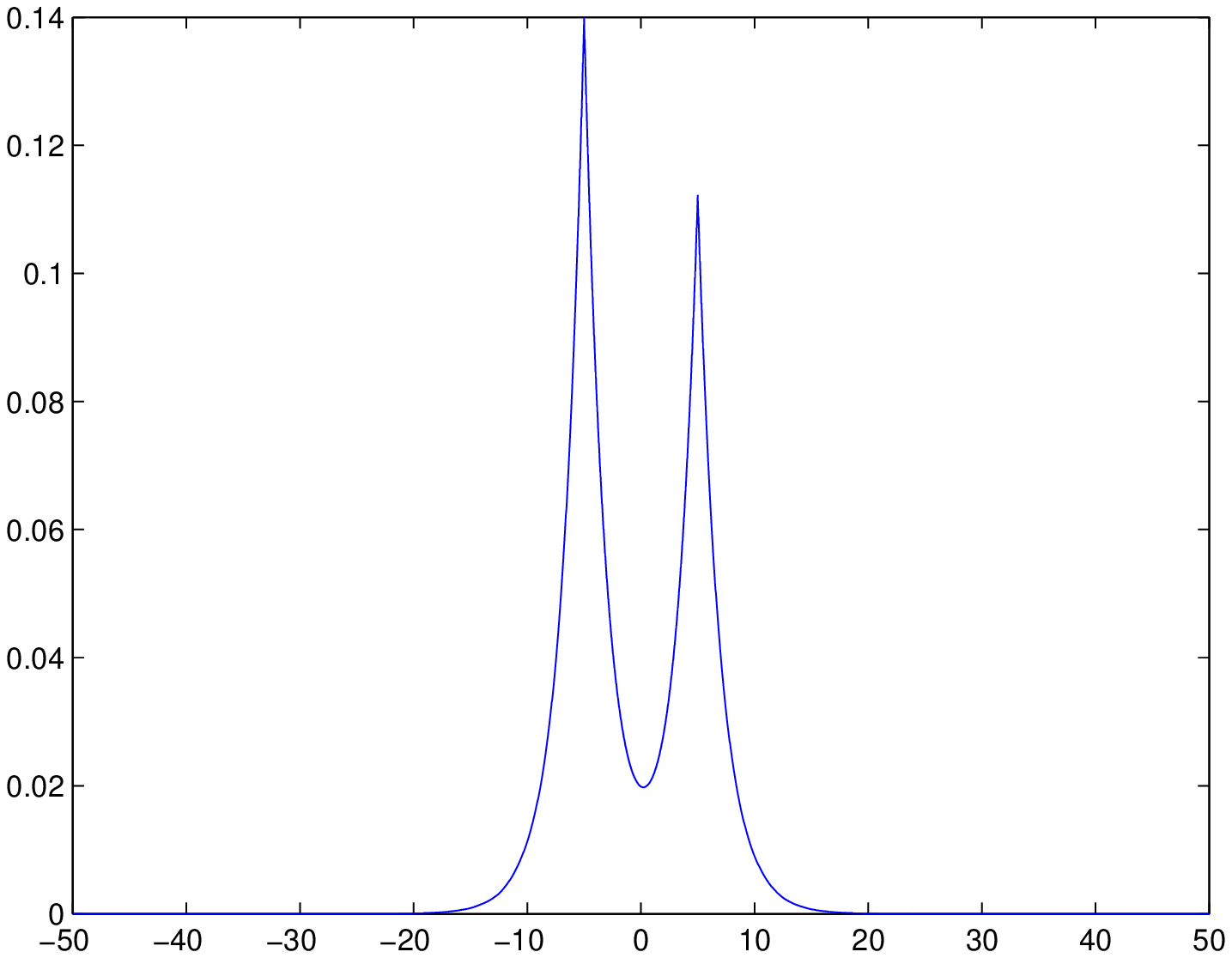} 
\includegraphics[scale=0.25]{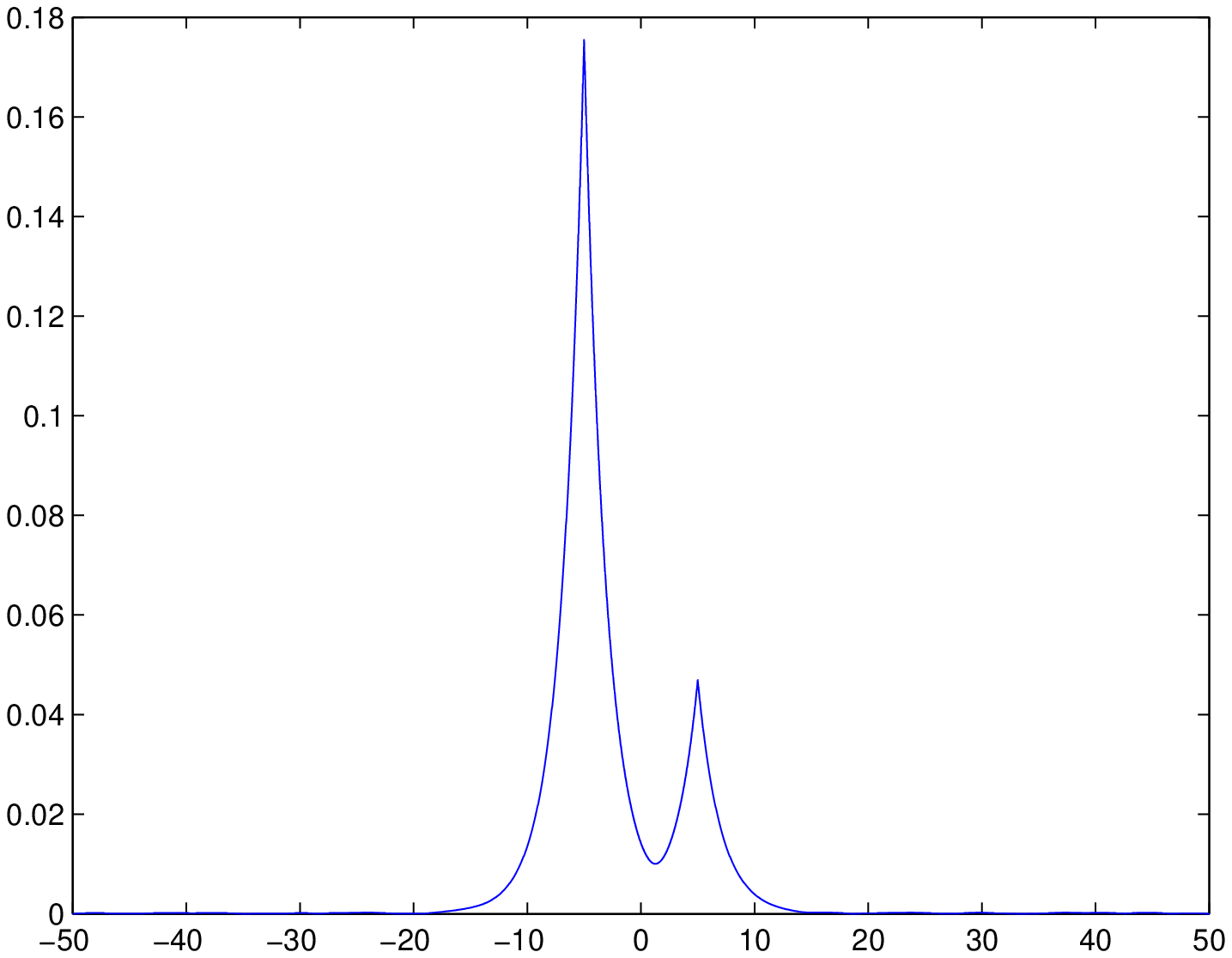}  \\
\includegraphics[scale=0.25]{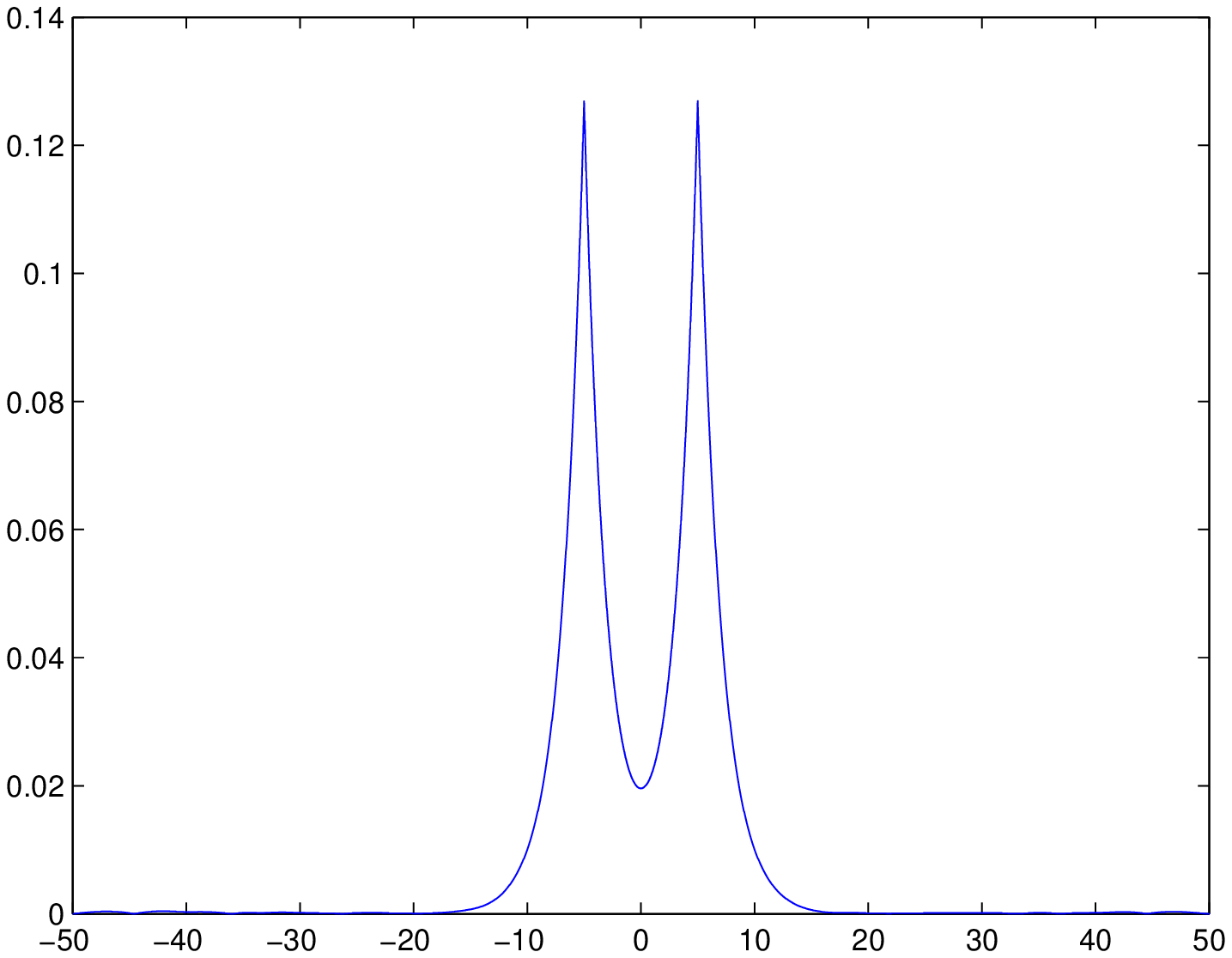} 
\includegraphics[scale=0.25]{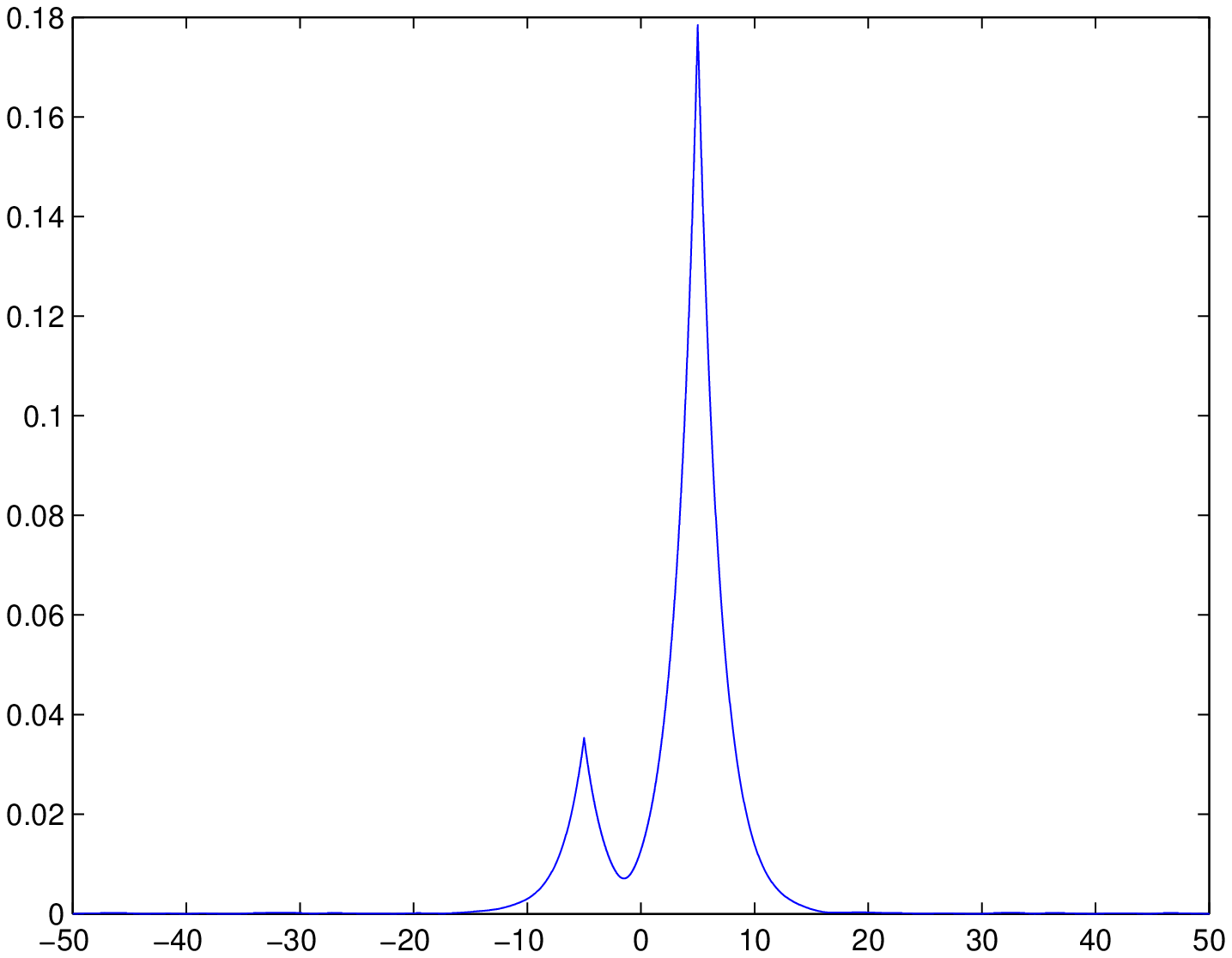}
\begin{center}
\includegraphics[scale=0.4]{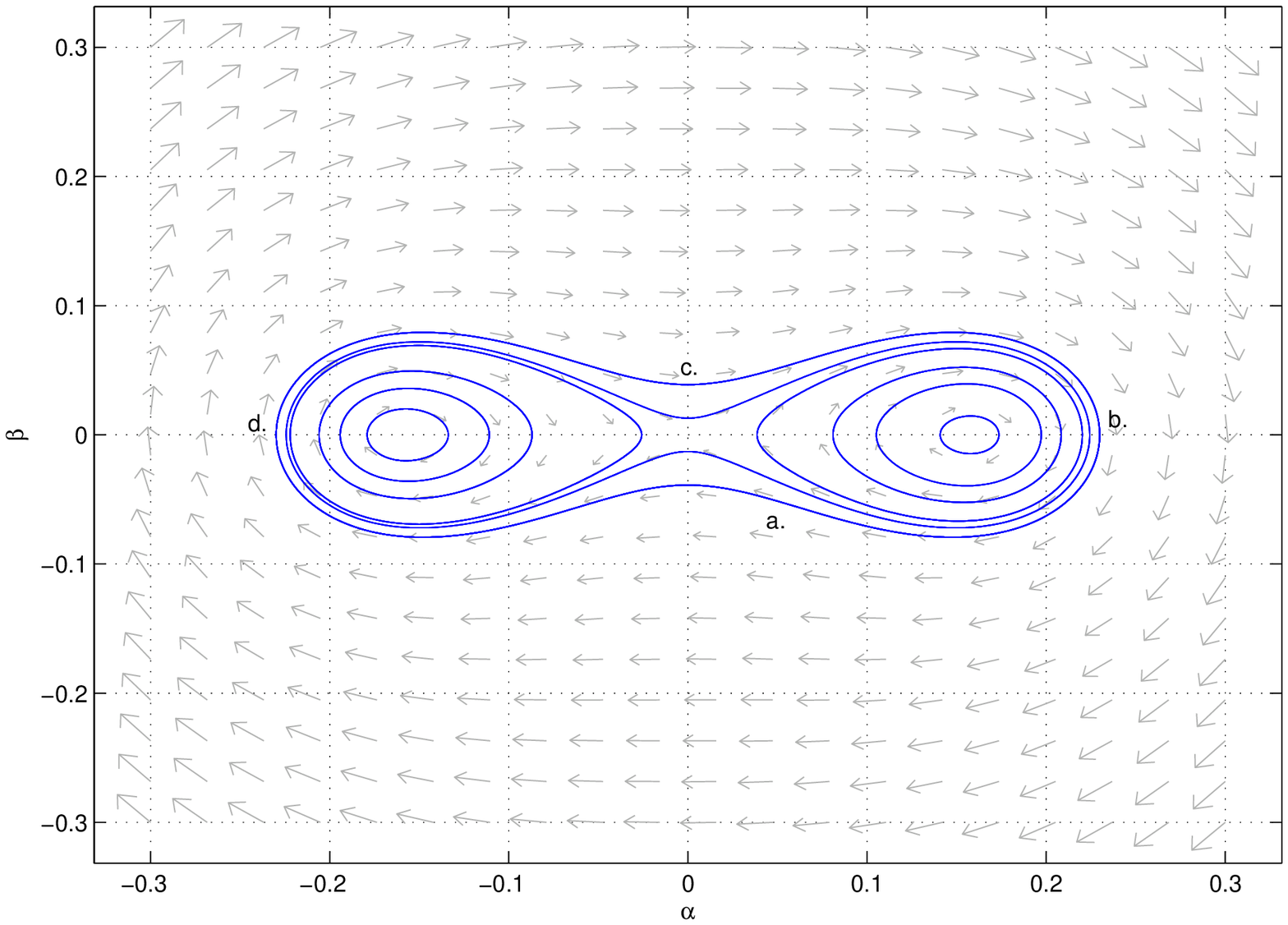}
\end{center}
\caption{A numerical plot of an oscillatory solution to \eqref{eqn:nlsdwp} at times a, b, c, and d respectively from the corresponding phase plane diagram of the finite dimensional Hamiltonian truncation.}
\label{fig:schem}
\end{figure}

{\sc Acknowledgments.}  
JLM was supported, in part, by a  U.S. National Science Foundation Postdoctoral Fellowship and a Hausdorff Center Postdoctoral Fellowship.  MIW was supported, in part, by  U.S. NSF Grants DMS-04-12305 and DMS-07-07850. MIW wishes to thank  Vered Rom-Kedar and Eli Shlizerman for stimulating discussions.  

\subsection{Notation}
 \begin{enumerate}
 \item The spaces $H^s(\RR^n)$, $L^p(\RR^n)$, $W^{k,p}$ are the
   standardly defined Sobolev integration spaces.
 \item We have the $L^2$ inner product:\ $\la f , g \ra\ =\ \int_{\RR^n} f\bar{g}$.
 \item Projection onto the bound states of $-\Delta+V$:
\begin{eqnarray*}
P_j f & = & \langle \psi_j , f \rangle\ \psi_j \ =\  (\pi_j f)\ \psi_j,\ \ j=0,1.
\end{eqnarray*}
\item Projection onto the continuous spectral part of $-\Delta+V$:
\begin{equation}
P_c f\ =\ (I - P_0 - P_1) f.
\nn\end{equation}
\end{enumerate}

\section{Formulation of NLS/GP as a coupled finite-infinite dimensional system}
\label{sec:oscillator-field}

In this section we derive an equivalent formulation of NLS/GP, appropriate for studying the exchange of energy between the bound and  radiative parts of the solution. We substitute the decomposition
 (\ref{decomp}) into NLS / GP and, to the resulting equation, apply the projection operators $P_0, P_1$ and $P_c$ to obtain the coupled system 

\begin{eqnarray}
\label{eqn:sys-id}
\left\{ \begin{array}{c}
i \dot{c}_0 - \Omega_0 c_0 + a_{0000} |c_0|^2 c_0 + a_{0011} (c_1^2 \bar{c}_0 +2 |c_1|^2 c_0)  =  F_0 (c_0,c_1,\bar{c}_0, \bar{c}_1; R, \bar{R}), \\
i \dot{c}_1 - \Omega_1 c_1 + a_{1111} |c_1|^2 c_1 + a_{0011} (c_0^2 \bar{c}_1 +2 |c_0|^2 c_1)  =  F_1 (c_0,c_1,\bar{c}_0, \bar{c}_1; R, \bar{R}), \\
i R_t - H R + P_c F_b (c_0,c_1,\bar{c}_0, \bar{c}_1)  =  P_c F_R (c_0,c_1,\bar{c}_0, \bar{c}_1; R, \bar{R}).
\end{array} \right. 
\end{eqnarray}
Here, 
\begin{eqnarray}
a_{ijkl} = \langle \psi_i \psi_j \psi_k, \psi_l \rangle,
\end{eqnarray}
\begin{eqnarray}
F_j = \pi_j F \label{Fj-def},
\end{eqnarray}
for $j = 0,1$ where
\begin{eqnarray}
F & =&  \left[ 2 |c_0|^2 \psi_0^2 + 2 |c_1|^2 \psi_1^2 + 2 (c_0 \bar{c}_1 + c_1 \bar{c}_0) \psi_0 \psi_1 \right] R  \nn\\
& + & \left[ c_0^2 \psi_0^2 + c_1^2 \psi_1^2 + 2 c_0 c_1 \psi_0 \psi_1 \right] \bar{R} \nn\\
& + & \left[\ \bar{c}_1 \psi_1 + \bar{c}_0 \psi_0 \right] R^2 + \left[ 2 c_0 \psi_0 + 2 c_1 \psi_1\ \right] |R|^2  +  |R|^2 R  , 
\label{eqn:F}
\end{eqnarray}
and
\begin{align}
F_b & =  P_c\ \left[\ |c_0|^2 c_0 \psi_0^3 + (c_0^2 \bar{c}_1+2 |c_0|^2 c_0) \psi_0^2 \psi_1 + (c_1^2 \bar{c_0} + 2 c_0 |c_1|^2) \psi_0 \psi_1^2 + |c_1|^2 c_1 \psi_1^3\ \right], \label{Fperp}\\
F_R & =  P_c \left(\ \left[\ 2 |c_0|^2 \psi_0^2 + 2 |c_1|^2 \psi_1^2 + 2 (c_0 \bar{c}_1 + c_1 \bar{c}_0) \psi_0 \psi_1\ \right] R \right.  \label{G-def}\\
&\ \ \ \ +  \left[\ c_0^2 \psi_0^2 + c_1^2 \psi_1^2 + 2 c_0 c_1 \psi_0 \psi_1 \ \right] \bar{R} \nn\\
&\ \ \ \ +  \left[\ \bar{c}_1 \psi_1 + \bar{c}_0 \psi_0 \right] R^2 + \left[ 2 c_0 \psi_0 + 2 c_1 \psi_1\ \right] |R|^2 \nn\\
& \ \ \ \  +  \left. |R|^2 R\ \right).  \nn
\end{align}
We have that $a_{ijkl} \ne 0\ \iff\ i+j+k+l = 0 \ \text{mod} 2$. For simplicity we take \begin{eqnarray*}
a_{ijkl} = 1 \ \text{if} \ i+j+k+l = 0 \ \text{mod} \ 2.
\end{eqnarray*}
Initial conditions for (\ref{eqn:sys-id}) are:
\begin{equation}
c_j(0)=\la \psi_j , u(\cdot,0) \ra,\ j=0,1;\ \ \ R(\cdot,0) = P_c R(\cdot,0) .
\label{sys-data}
\end{equation}
The system (\ref{eqn:sys-id}) may be viewed as an infinite dimensional Hamiltonian system, comprised of two coupled subsystems: one finite dimensional, governing bound state degrees of freedom described by $c_0$ and $c_1$, and a second, infinite dimensional, governing a dispersive wave field, $R(\cdot,t)=P_c R(\cdot,t)$. In the next section,
we focus on the finite dimensional truncation of (\ref{eqn:sys-id}) obtained by setting $R=0$. After obtaining a detailed description of the phase space of this truncation in section \ref{sec:finite-dim}, we then turn toward proving the long-time persistence of structures within the finite dimensional system, within the full infinite dimensional problem.

Before embarking on this path, we conclude this section with an alternative coordinate description of (\ref{eqn:sys-id}). These coordinates prove very useful in understanding the bifurcations within the finite dynamical subsystem, as $\cN$ is varied, and its participation within the infinite dimensional dynamics.

\subsection{Alternative coordinates}
\label{sec:hamco}

We shall require the following decomposition, proved in Appendix \ref{sec:errorestfd}.

Introduce the following change of coordinates: 
 \begin{equation}
 (\ c_0(t),c_1(t),\bar{c_0}(t),\bar{c_1}(t),R(\cdot,t)\ )\ \mapsto (\ A(t),\alpha(t),\beta(t),\theta(t),R(\cdot,t)\ )\nn\end{equation}
 defined by 
\begin{align}
c_0 (t) &= A(t) e^{i \theta (t)} ,
\label{eqn:rho0}\\
c_1 (t) &= (\alpha(t) + i \beta(t)) e^{i \theta (t)} , \label{eqn:rho1}\\
u(x,t) &= e^{i \theta (t)} (A(t)  \psi_0 + (\alpha(t) + i \beta(t)) \psi_1 + R) . \label{u-new}
\end{align}
Such coordinates have been used for finite dimensional systems in for instance the works \cite{KW}, \cite{SRK}.

Substitution of this ansatz into NLS/GP, we see
\begin{eqnarray*}
(i \dot{A}  - \dot{\theta} A - \Omega_0 A)  \psi_0 + (i \dot{\alpha} - \dot{\beta} - \dot{\theta} (\alpha + i \beta) - (\alpha + i \beta) \Omega_1 )  \psi_1 \\
 + i R_t - HR -\dot{\theta} R = F(A,\alpha,\beta;R,\bar{R}),
\end{eqnarray*}
where $F$ is determined as in \eqref{eqn:F}.
Hence,
\begin{eqnarray*}
\dot{A} & = & \Im (\pi_0 F ), \\
\dot{\alpha} & = & \dot{\theta} \beta + \Omega_1 \beta + \Im (\pi_1 (F)), \\
\dot{\beta} & = & -\dot{\theta} \alpha - \Omega_1 \alpha - \Re (\pi_1 (F)), \\
\dot{\theta} & = & - \Omega_0 - A^{-1} \Re( \pi_0 (F)), \\
i R_t & = & H R + \dot{\theta} R + P_c (F)
\end{eqnarray*}
or
\begin{eqnarray*}
\dot{A} & = & \Im (\pi_0 F), \\
\dot{\alpha} & = & (- \Omega_0 - A^{-1} \pi_0 (F)) \beta + \Omega_1 \beta + \Im (\pi_1 (F)), \\
\dot{\beta} & = & -(- \Omega_0 - A^{-1} \Re( \pi_0 (F))) \alpha - \Omega_1 \alpha - \Re (\pi_1 (F)), \\
\dot{\theta} & = & - \Omega_0 - A^{-1} \Re (\pi_0 (F)), \\
i R_t & = & (H - \Omega_0) R - A^{-1} \Re( \pi_0 (F)) R + P_c (F).
\end{eqnarray*}
Note, there will be linear terms in $R$ and $\bar{R}$ contained in $F$, which we must handle very carefully in our eventual iteration argument.

This leads to the system of the form
\begin{eqnarray*}
\dot{A} & = & - 2 \alpha \beta A + \text{Error} (R, \bar{R}; A, \alpha, \beta), \\
\dot{\alpha} & = & [\Omega_1 - (\alpha^2 + \beta^2) - A^2 + \dot{\theta}] \beta + \text{Error} (R, \bar{R}; A, \alpha, \beta) , \\
\dot{\beta} & = & - [ \Omega_1 - (\alpha^2 + \beta^2 + A^2) -2 A^2 + \dot{\theta}] \alpha + \text{Error} (R, \bar{R}; A, \alpha, \beta) , \\
\dot{\theta} & = & -\Omega_0 + A^2 + (3 \alpha^2 + \beta^2) + A^{-1}
\text{Error} (R, \bar{R}; A, \alpha, \beta) , \\
i R_t & = & (H - \Omega_0) R - A^{-1} \Re( \pi_0 (F))  R + P_c F_b (A, \alpha, \beta) + P_c F_R (A, \alpha, \beta; R, \bar{R}).
\end{eqnarray*}

Then, substituting $\dot{\theta}$, the system becomes
\begin{eqnarray}
\label{eqn:coupledfd1}
\dot{A} & = & - 2 \alpha \beta A + \text{Error}_A (R, \bar{R}; A, \alpha, \beta), \\
\label{eqn:coupledfd2}
\dot{\alpha} & = & [\Omega_1 -\Omega_0 + 2 \alpha^2 ] \beta + \text{Error}_\alpha (R, \bar{R}; A, \alpha, \beta) , \\
\label{eqn:coupledfd3}
\dot{\beta} & = & - [ \Omega_1 -\Omega_0 - 2 A^2 + 2 \alpha^2 ] \alpha + \text{Error}_\beta (R, \bar{R}; A, \alpha, \beta) , \\
\label{eqn:coupledfd4}
\dot{\theta} & = & -\Omega_0 + A^2 + 3 \alpha^2 + \beta^2 +
\text{Error}_\theta (R, \bar{R}; A, \alpha, \beta) , \\
\label{eqn:coupledid}
i R_t & = & (H - \Omega_0) R + (A^2 + 3 \alpha^2 + \beta^2) R + P_c F_b (A, \alpha, \beta) + P_c F_R (A, \alpha, \beta ; R, \bar{R}),
\end{eqnarray}
with $P_c R = R$.

For simplicity, set $\vec{\sigma} = (\alpha, \beta, A)$.
\begin{prop}
Assume $A,\dots, R$ defined as in \eqref{eqn:coupledfd1}-\eqref{eqn:coupledid}. Then, we have 
\begin{eqnarray}
\label{eqn:est-ERROR-A}
|\text{Error}_A| & \lesssim & |\vec{\sigma}|^2 \| R \|_{L^\infty} + |\vec{\sigma}| \| R \|_{L^\infty}^2 + \| R \|_{L^\infty}^3, \\ 
\label{eqn:est-ERROR-alpha}
|\text{Error}_\alpha| & \lesssim & |\vec{\sigma}|^2 \| R \|_{L^\infty} + |\vec{\sigma}| \| R \|_{L^\infty}^2 + \| R \|_{L^\infty}^3 , \\
\label{eqn:est-ERROR-beta}
|\text{Error}_\beta| & \lesssim & |\vec{\sigma}|^2 \| R \|_{L^\infty} + |\vec{\sigma}| \| R \|_{L^\infty}^2 + \| R \|_{L^\infty}^3, \\
\label{eqn:est-ERROR-theta}
|\text{Error}_\theta| & \lesssim & A^{-1} \left( |\vec{\sigma}|^2 \| R \|_{L^\infty} + |\vec{\sigma}| \| R \|_{L^\infty}^2 + \| R \|_{L^\infty}^3 \right), \\
\label{eqn:est-ERROR-fperp}
|P_c F_b| & = & \left| P_c \left[ A^3 \psi_0^3 + (\alpha^2 + \beta^2) (\alpha + i \beta) \psi_1^3 + A (\alpha + i \beta)^2 \psi_1^2 \psi_0 \right. \right.\\
& + & \left.\left.  2 A (\alpha^2 + \beta^2) \psi_1^2 \psi_0 + A^2 (\alpha - i \beta) \psi_0^2 \psi_1 + 2 A^2 (\alpha + i \beta) \psi_0^2 \psi_1 \right] \right|, \nn \\
\label{eqn:est-ERROR-gperp}
|P_c F_R| & = & | P_c \left( \left[  2 A^2 \psi_0^2 + 4 A(\alpha) \psi_0 \psi_1 + 2 (\alpha^2 + \beta^2) \psi_1^2  \right] R \right.  \\
& + & \left[ A^2 \psi_0^2 + (\alpha + i \beta)^2 \psi_1^2 + 2 A (\alpha + i \beta) \psi_0 \psi_1  \right] \bar{R}  \nn \\
 &+& \left. \left[ A \psi_0 + (\alpha - i \beta) \psi_1 \right] R^2 + \left[ 2 A \psi_0 + 2 (\alpha + i \beta) \psi_1 \right] |R|^2 + |R|^2 R \right)|. \nn
\end{eqnarray}
The error terms in $\alpha$ and $\beta$ gained from $\dot{\theta}$ are of the form $A^{-1} \alpha$, $A^{-1} \beta \ll 1$.  
\end{prop}

\begin{proof}
These follow directly from the computations in Appendix \ref{sec:errorestfd}.
\end{proof}

For convenience, we rewrite the system \eqref{eqn:coupledfd1}-\eqref{eqn:coupledid} in compact form as:
\begin{eqnarray}
\label{eqn:simpcoupledsystem}
\left\{ \begin{array}{c}
\dot{\vec{\sigma}} = \vec{F}_{FD} (\vec{\sigma}) + \vec{G}_{FD} (\vec{\sigma};R,\bar{R}), \\
\dot{\theta} = -\Omega_0 + A^2 + 3 \alpha^2 + \beta^2 + G_\theta (R,
\bar{R}; A, \alpha, \beta) , \\
(i \p_t -(H - \Omega_0)  - (A^2 + 3 \alpha^2 + \beta^2)) R = P_c F_{b}
(\vec{\sigma}) + P_cF_R (\vec{\sigma}; R, \bar{R}) .
\end{array} \right.
\end{eqnarray}

For simplicity, we take $F_b = P_c F_b$ and $F_R = P_c F_R$ in the sequel.

\section{Phase space of the finite dimensional Hamiltonian truncation of NLS/GP}
\label{sec:finite-dim}

In this section we study the finite-dimensional system obtained by setting the dispersive part of the solution, $R(x,t)$, equal to zero in (\ref{eqn:sys-id}). We denote the solution of the resulting system by $(\rho_0(t),\rho_1(t))\in \CC^2$:
\begin{eqnarray}
\label{eqn:sys-fd}
\left\{ \begin{array}{c}
i \dot{\rho_0} = \Omega_0 \rho_0 - ( \rho_0^2 \bar{\rho_0} + 2 \rho_1 \bar{\rho_1} \rho_0 + \rho_1^2 \bar{\rho_0}), \\
i \dot{\bar{\rho_0}} = -\Omega_0 \bar{\rho_0} + ( \bar{\rho_0}^2 \rho_0 + 2 \rho_1 \bar{\rho_1} \bar{\rho_0} + \bar{\rho_1}^2 \rho_0) , \\
i \dot{\rho_1} = \Omega_1 \rho_1 - ( \rho_1^2 \bar{\rho_1} +  2 \rho_0
\bar{\rho_0} \rho_1 +  \rho_0^2 \bar{\rho_1}) , \\
i \dot{\bar{\rho_1}} = -\Omega_1 \bar{\rho_1} + (\bar{\rho_1}^2 \rho_1 + 2 |\bar{\rho_0}|^2 \bar{\rho_1} +  \bar{\rho_0}^2 \rho_1) .
\end{array} \right. 
\end{eqnarray}
Symmetry breaking for a system of this type arising from a general
class of {\it defocusing}
  nonlinearities was considered recently in \cite{Sac}.

This is a two degree of freedom Hamiltonian system, with time-translation and phase invariances inherited from NLS/GP. The associated  time-conserved Hamiltonian and $L^2$ (optical power or particle number) functionals are:
\begin{align}
H & =  \Omega_0 |\rho_0|^2 + \Omega_1 |\rho_1|^2 - \frac{1}{2} |\rho_0|^4 - \frac{1}{2} |\rho_1|^4 - 2 |\rho_1|^2 |\rho_0|^2 - \frac{1}{2} (\rho_1^2 \bar{\rho_0}^2 + \bar{\rho_1}^2 \rho_0^2) , \label{Ndef}\\
N & =  |\rho_0|^2 + |\rho_1|^2 .\label{Hdef}
\end{align}
In terms of $H$, system (\ref{eqn:sys-fd}) can be expressed
 in Hamiltonian form
\begin{eqnarray}
i \partial_t \vec{\rho} = J \nabla_{\vec{\rho}} H,
\end{eqnarray}
where
\begin{eqnarray}
\vec{\rho} = \left[ \begin{array}{c}
\rho_0 \\
\bar{\rho_0} \\
\rho_1 \\
\bar{\rho_1}
\end{array} \right],\ \ \ \  J = \left[ \begin{array}{cccc} 
0 & 1 & 0 & 0 \\
-1 & 0 & 0 & 0 \\
0 & 0 & 0 & 1 \\
0 & 0 & -1 & 0 
\end{array} \right].
\end{eqnarray}

In terms of the alternative coordinates of section (\ref{sec:hamco}), 
\begin{equation}
\rho_0=Ae^{i\theta},\ \rho_1=(\alpha+i\beta)e^{i\theta},
\nn\end{equation}
hence the system (\ref{eqn:sys-fd}) takes the form:
\begin{eqnarray}
\label{eqn:ode-sys-alpha} 
\left\{ \begin{array}{c}
\dot{\alpha} = \left[\ \Omega_1+\dot{\theta}-(\alpha^2+\beta^2) -A^2\ \right]\ \beta , \\
\dot{\beta} = -\left[\ \Omega_1+\dot{\theta}-(\alpha^2+\beta^2)-3A^2\ \right] \alpha ,  \\
\dot{A} = -2 \alpha \beta A , \\
\dot{\theta} = - \Omega_0 + A^2+3\alpha^2+\beta^2 .
\end{array} \right.
\end{eqnarray}
Recall that, for simplicity, we have set  $a_{ijkl} = 1$ for all $i$, $j$, $k$, $l = 0$, $1$. Note that $\theta$ completely decouples from the $A,\alpha,\beta$ equations, meaning we have
\begin{eqnarray}
\label{eqn:alpha-reduced1}
\left\{ \begin{array}{c}
\dot{\alpha} = \left[\ \Omega_1- \Omega_0 +2 \alpha^2 \ \right]\ \beta , \\
\dot{\beta} = -\left[\ \Omega_1- \Omega_0 -2 A^2+2 \alpha^2\ \right] \alpha , \\
\dot{A} = -2 \alpha \beta A , 
\end{array} \right.
\end{eqnarray}
where now the $\theta$ equation is decoupled to give
\begin{eqnarray}
\label{eqn:theta-decoupled}
\dot{\theta} = - \Omega_0 + A^2+3\alpha^2+\beta^2 .
\end{eqnarray}

One can verify changing coordinates in (\ref{Ndef}-\ref{Hdef}), or directly from (\ref{eqn:ode-sys-alpha}), the time-conserved quantities: 
\begin{eqnarray}
N & = & A^2 + \alpha^2 + \beta^2, \\
H & = & \Omega_0 A^2 + \Omega_1 (\alpha^2 + \beta^2)
 - \frac{1}{2}A^4  - \frac{1}{2} (\alpha^2 + \beta^2)^2 - 2 A^2 (\alpha^2 + \beta^2) - A^2 (\alpha^2 - \beta^2).
\end{eqnarray}

We obtain a closed system for $(\alpha,\beta)$ using that $N = A^2 + \alpha^2 + \beta^2$ is conserved.  Then, we may reduce the system to
\begin{align}
\label{eqn:ode-sys-alpha-reduced}
\dot{\alpha} &= \left[\ \Omega_1- \Omega_0 + 2 \alpha^2 \right]\ \beta , \\
\dot{\beta} &= -\left[\ \Omega_1- \Omega_0 + 2( 2\alpha^2 +  \beta^2 - N) \ \right] \alpha .\nn 
\end{align}
In addition, the system has the conserved quantity
\begin{eqnarray*}
H & = & N (\Omega_0 - \frac{N}2) + \Omega_{10} (\alpha^2 + \beta^2) -2 N \alpha^2 + (\alpha^2 + \beta^2) + \alpha^4 - \beta^4.
\end{eqnarray*}

Now, let us define the matrices $B_{eq}$ and $\tilde{B}_{eq}$ to be those related to linearization about the equilibrium solution $(A_{eq},\alpha_{eq},\beta_{eq},\theta_{eq})$ and the decoupled reduced system  $(A_{eq},\alpha_{eq},\beta_{eq})$ respectively.  Similarly, let $M(t)$ and $\tilde{M}$ be the resulting monodromy matrices for nearby time dependent periodic orbits $(A(t),\alpha(t),\beta(t),\theta (t))$ and the decoupled reduced system  $(A (t),\alpha(t),\beta(t))$ respectively. 

In the following section, we actually discuss the relevant bounds on the operators $e^{B_{eq} t}$ and $e^{\tilde{B} t}$.  

 \subsection{Bifurcation of equilibria for \eqref{eqn:alpha-reduced1} and Symmetry Breaking in NLS/GP}
\label{sec:fd-phasespace}

 Recall that $N=N[A,\alpha,\beta]$  is a constant of the motion for the system (\ref{eqn:ode-sys-alpha}), corresponding to the physical quantities optical power or particle number. Thus it is natural to explore the nature of the phase space restricted to the level sets of $N$.
We are interested in time-periodic states of frequency $\Omega$, corresponding in the physical space to solutions of NLS/GP of the form 
 $u(x,t)=e^{-i\Omega t}U$. Thus, we transform the system to a rotating frame by setting
 \begin{equation}
 \theta(t) \ =\ \Theta(t) - \Omega t
 \label{Theta-def}
 \end{equation}
 and obtain 
 \begin{align}
\label{eqn:alpha-rotate}
\dot{\alpha} &= \left[\ \Omega_1-\Omega+\dot\Theta(t)-(\alpha^2+\beta^2) -A^2\ \right]\ \beta ,\\
\dot{\beta} &= -\left[\ \Omega_1-\Omega+\dot\Theta(t)-(\alpha^2+\beta^2)-3A^2\ \right] \alpha ,\\
\dot{A} &= -2 \alpha \beta A , \\
\dot{\Theta} &= \Omega- \Omega_0 + A^2+3\alpha^2+\beta^2.
\end{align}
 The states we seek are equilibria in this rotating frame. Thus, we have
 \begin{align}
\label{e-rotate.a}
& \left[\ \Omega_1-\Omega-(\alpha^2+\beta^2) -A^2\ \right]\ \beta = 0,\\
&\left[\ \Omega_1-\Omega-(\alpha^2+\beta^2)-3A^2\ \right] \alpha = 0,\label{e-rotate.b} \\
&\alpha \beta A = 0,\label{e-rotate.c}\\
&\Omega- \Omega_0 + A^2+3\alpha^2+\beta^2 = 0,
\label{e-rotate.d}
\end{align}
whose solutions we consider on the level set 
\begin{equation}
A^2+\alpha^2+\beta^2=N.
\label{e-rotate.e}\end{equation}
It is easy to observe an equilibrium corresponding to the\\
{\bf Symmetric states:}
 \begin{equation}
{\rm For\ }\ N\ge0:\ \ \ A_*=N^{\frac{1}{2}},\ \alpha_*=\beta_*=0,\ \Omega_*=\Omega_0-N.\label{sym-eq}\end{equation}
Via (\ref{u-new}) we identify this with the symmetric ground state of NLS/GP;
\begin{equation}
u(x,t)\ \sim\ N^{\frac{1}{2}}\ e^{i(-\Omega_0+N)t}\ \psi_0(x).
\nn\end{equation} 
{\it Due to its correspondence with the symmetric state, we refer
to this equilibrium of the finite dimensional reduction, the symmetric equilibrium.}

A second, bifurcating family can be found explicitly as follows. 
 Define
\begin{equation}
N^{FD}_{cr}\ =\ \frac{\Omega_{10}}{2}\ =\ \frac{\Omega_1-\Omega_0}{2}.
\label{Ncr-def}
\end{equation}
\begin{rem}
Had we set the nonlinearity coefficient $g=-1$ and the ``interaction weights'', $a_{ijkl}$ equal to one, we would have
\begin{eqnarray}
N_{cr}^{FD} = \frac{\Omega_{10}}{g(\int \psi_0^4 dx -3 \int \psi_0^2 \psi_1^2 dx)} > 0;
\end{eqnarray}
which is easily observed by comparison to a single well potential symmetric state for $L$ sufficiently large (see \cite{KKSW}).\footnote{Though $N_{cr}^{FD}$ is a good first order approximation to $N_{crit}$, the true symmetry breaking point in the nonlinear problem \eqref{eqn:nlsdwp}.  As in this note we will be studying existence of solutions close to those described by the finite dimensional dynamcis of \eqref{eqn:sys-fd}, we will work from here on using $N_{cr} = N_{cr}^{FD}$. }
\end{rem}

Using (\ref{e-rotate.d}) we first eliminate $\Omega$ from  (\ref{e-rotate.a}) and obtain $\left[\ \Omega_1-\Omega_0+2\alpha^2\ \right]\beta = 0$.\\ 
Since $\Omega_1-\Omega_0>0$ we conclude $\beta=0$. Thus, (\ref{e-rotate.c}) is satisfied. We now use (\ref{e-rotate.d}) again to eliminate $\Omega$ from (\ref{e-rotate.b}). Thus we have, since $\beta=0$
\begin{align}
A^2\ -\ \alpha^2\ &=\ N^{FD}_{cr},\nn\\
A^2\ +\ \alpha^2\ &=\ N\nn.
\end{align}
Solving for $A$ and $\alpha$ we have the following equilibria, corresponding to symmetry broken states, which bifurcate for at $N=N^{FD}_{cr}$: 

\noindent{\bf Symmetry broken states:}
\begin{align}
&{\rm For\ }\ N\ge N^{FD}_{cr}\equiv\frac{\Omega_{10}}{2}:\nn\\ 
&\ \ \ \  A_*=\left( \frac{N+N^{FD}_{cr}}{2} \right)^{\frac12},\ \ \alpha_*=\left( \frac{N-N^{FD}_{cr}}{2} \right)^{\frac12},\ \ \beta_*=0, \ \ \  
\Omega_*=-2\Omega_0+N^{FD}_{cr}-N.
\label{sym-brk}\end{align}
Via (\ref{u-new}) we identify this with the asymmetric ground state of NLS/GP;
\begin{equation}
u(x,t)\ \sim\ e^{i(-2\Omega_0+N^{FD}_{cr}-N)t}
\left(\ \left( \frac{N+N^{FD}_{cr}}{2} \right)^{\frac12}\ \psi_0(x)\ +\ 
 \left( \frac{N-N^{FD}_{cr}}{2} \right)^{\frac12}\ \psi_1(x)\ \right).
\nn\end{equation} 
{\it Due to its correspondence with the asymmetric (symmetry-broken) states of NLS/GP, we refer to this equilibria of the finite dimensional reduction as asymmetric equilibria.}
\subsection{Stability of equilibria;  finite dimensional analysis}
\label{subsec:stability-finite-dim}
 We consider the stability of the various solution branches obtained in the previous section.  We rewrite the system (\ref{eqn:alpha-rotate}),
 using the last equation to eliminate $-\Omega+\dot\Theta$ from the equations for $\alpha$ and $\beta$. Thus we have
 \begin{align}
\label{syst-rot}
\left\{ \begin{array}{c}
\dot{\alpha} = \left[\ \Omega_{10}+2\alpha^2\ \right]\ \beta , \\
\dot{\beta} = -\left[\ \Omega_{10}+2\alpha^2-2A^2\ \right] \alpha , \\
\dot{A} = -2 \alpha \beta A , \\
\dot{\Theta} = \Omega- \Omega_0 + A^2+3\alpha^2+\beta^2 .
\end{array} \right.
\end{align}
  Note that in these coordinates the equations for $\alpha$, $\beta$ and $A$ decouple from the equation for $\Theta$.

The finite dimensional system has a phase portrait, equivalent to \eqref{eqn:sys-fd}, see Figure \ref{fig7}. In particular, we observe elliptic and hyperbolic equilibria, and periodic orbits. We now embark on detailed linear stability analysis of these states.

Linearization about an arbitrary solution 
\begin{eqnarray}
(\alpha_*(t), \beta_*(t), A_*(t), \theta_*(t))
\end{eqnarray}
 gives the linearized perturbation equation 
\begin{eqnarray}
\label{eqn:sys-lin-alpha}
\partial_t \left[ \begin{array}{c}
\delta \alpha \\
\delta \beta \\
\delta A \\
\delta \theta
\end{array} \right] & = & \left[ \begin{array}{cccc} 
4 \alpha_* \beta_* & (\Omega_{10} + 2 \alpha_*^2) & 0 & 0 \\
-(\Omega_{10} +6\alpha_*^2-2A_*^2) & 0 & -4 \alpha_*A_* & 0 \\
- 2 A_* \beta_* & -2 \alpha_* A_* & 2 \alpha_* \beta_* & 0 \\
6\alpha_* & 2\beta_* & 2A_* & 0
\end{array} \right]  \left[ \begin{array}{c}
\delta \alpha \\
\delta \beta \\
\delta A \\
\delta \theta
\end{array} \right] \\
& = & B(t) \left[ \begin{array}{c}
\delta \alpha \\
\delta \beta \\
\delta A \\
\delta \theta
\end{array} \right] .
\end{eqnarray}
Since the evolution of $\alpha$, $\beta$ and $A$, decouple from that for $\Theta$, we consider the behavior of the reduced system
\begin{equation}
\partial_t \left[ \begin{array}{c}
\delta \alpha \\
\delta \beta \\
\delta A 
\end{array} \right] =  \tilde{B} (t) \left[ \begin{array}{c}
\delta \alpha \\
\delta \beta \\
\delta A
\end{array} \right] =  
\left[ \begin{array}{ccc} 
4 \alpha_* \beta_* & 2(N_{cr}^{FD} + \alpha_*^2) & 0  \\
-2(N_{cr}^{FD} +3\alpha_*^2-A_*^2) & 0 & -4 \alpha_*A_*  \\
- 2 A_* \beta_* & -2 \alpha_* A_* & -2 \alpha_* \beta_* \\
\end{array} \right]  \left[ \begin{array}{c}
\delta \alpha \\
\delta \beta \\
\delta A
\end{array} \right].\label{red-pert-eqn}\end{equation}
Then, $\tilde{B}$ is the $3$ by $3$ block in matrix of equation (\ref{eqn:sys-lin-alpha}). In obtaining (\ref{red-pert-eqn}), we used that $N_{cr}=\Omega_{10}/2$. 

{\bf Linearized dynamics about the symmetric equilibrium state:}

For the symmetric equilibrium, as displayed in (\ref{sym-eq}) we have 
\begin{eqnarray}
\label{eqn:eqsym}
(\alpha^{eq}_{-}, \beta^{eq}_{-}, A^{eq}_{-}, \theta^{eq}_{-} (t)) = (0,0,N^{\frac{1}{2}},(-\Omega_0 + N)t).
\end{eqnarray}
Hence,
\begin{eqnarray*}
B = B_{-} = 2 \left[ \begin{array}{cccc} 
0 & N^{FD}_{cr}  & 0 & 0 \\
N-N_{cr}^{FD} & 0 & 0 & 0 \\
0 & 0 & 0 & 0 \\
0 & 0 & N^{\frac{1}{2}} & 0 
\end{array} \right] .
\end{eqnarray*}

For the reduced system, we have
\begin{equation}\label{tB-symm}
\tilde{B}_{-}  =  2\left[ \begin{array}{ccc} 
0 & N^{FD}_{cr}  & 0 \\
N-N_{cr}^{FD} & 0 & 0 \\
0 & 0 & 0 
\end{array} \right], \\
\end{equation}
whose eigenvalues and implied linear stability character is as follows:
\begin{align}
&N<N_{cr}^{FD}:\ \ 
 \lambda_0=0,\ \ \lambda_\pm(N)=\pm i 2
 \left[\ (N_{cr}^{FD} -N)\ N_{cr}^{FD}\ \right]^{\frac12},\ \ \ 
 N<N_{cr}^{FD}\ \ \ ({\rm stable\ elliptic\ point}), \label{tB-symm-eig} \\
& N>N_{cr}^{FD}:\ \  
 \lambda_0=0,\ \ \lambda_\pm(N)=\pm 2 \left[\ (N-N_{cr}^{FD})\ N_{cr}^{FD}\ \right]^{\frac12},\ \ \ N>N_{cr}^{FD},\ \ \ ({\rm unstable\ saddle}).\label{tB-symm-eig-unstable}
\end{align}
see Fig. \ref{fig7}.  {\it Thus, the symmetric state transitions from stable to unstable as $N$ increases beyond $N_{cr}^{FD}$.}

\begin{figure}
\includegraphics[scale=0.25]{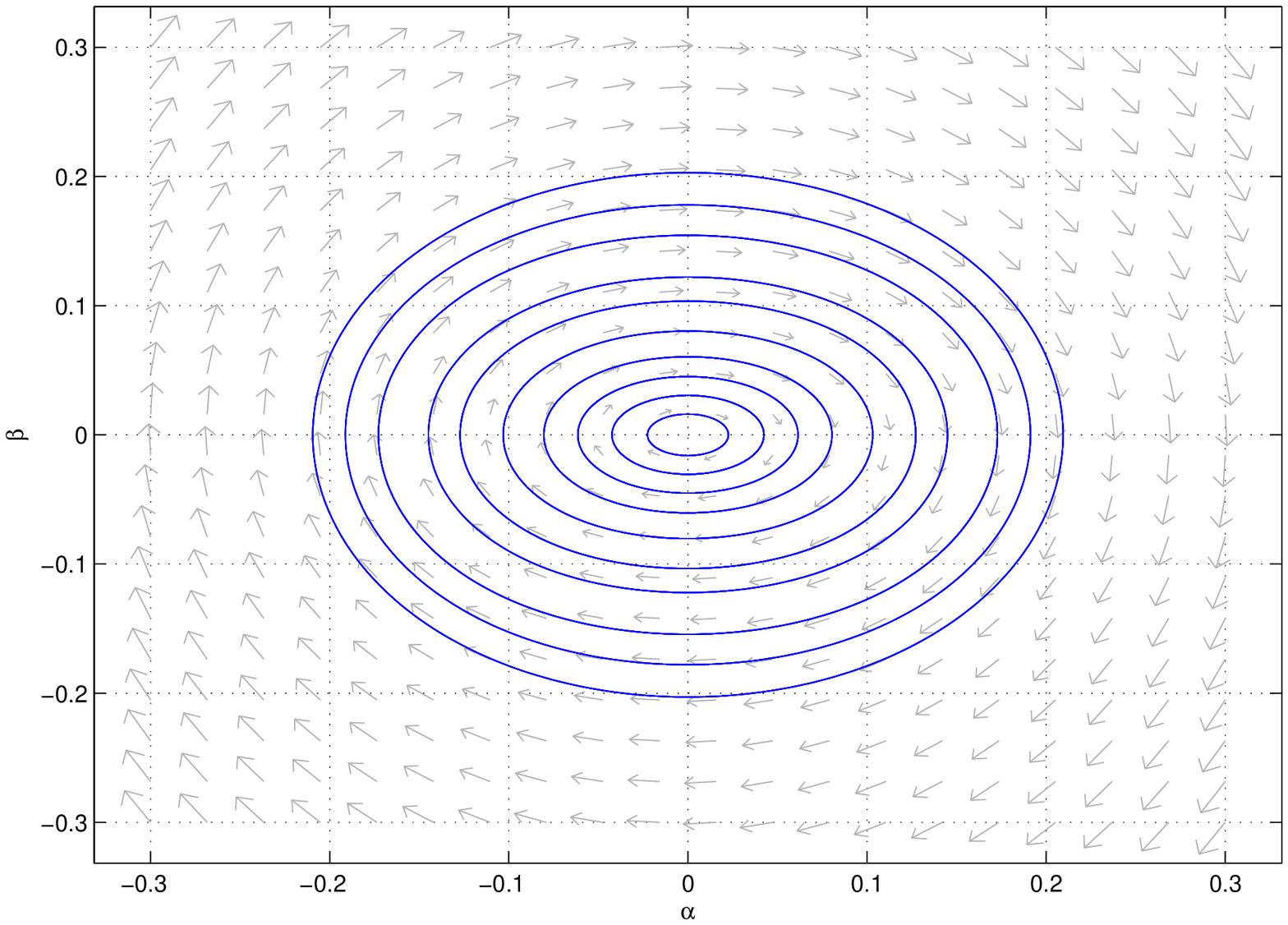}
\includegraphics[scale=0.25]{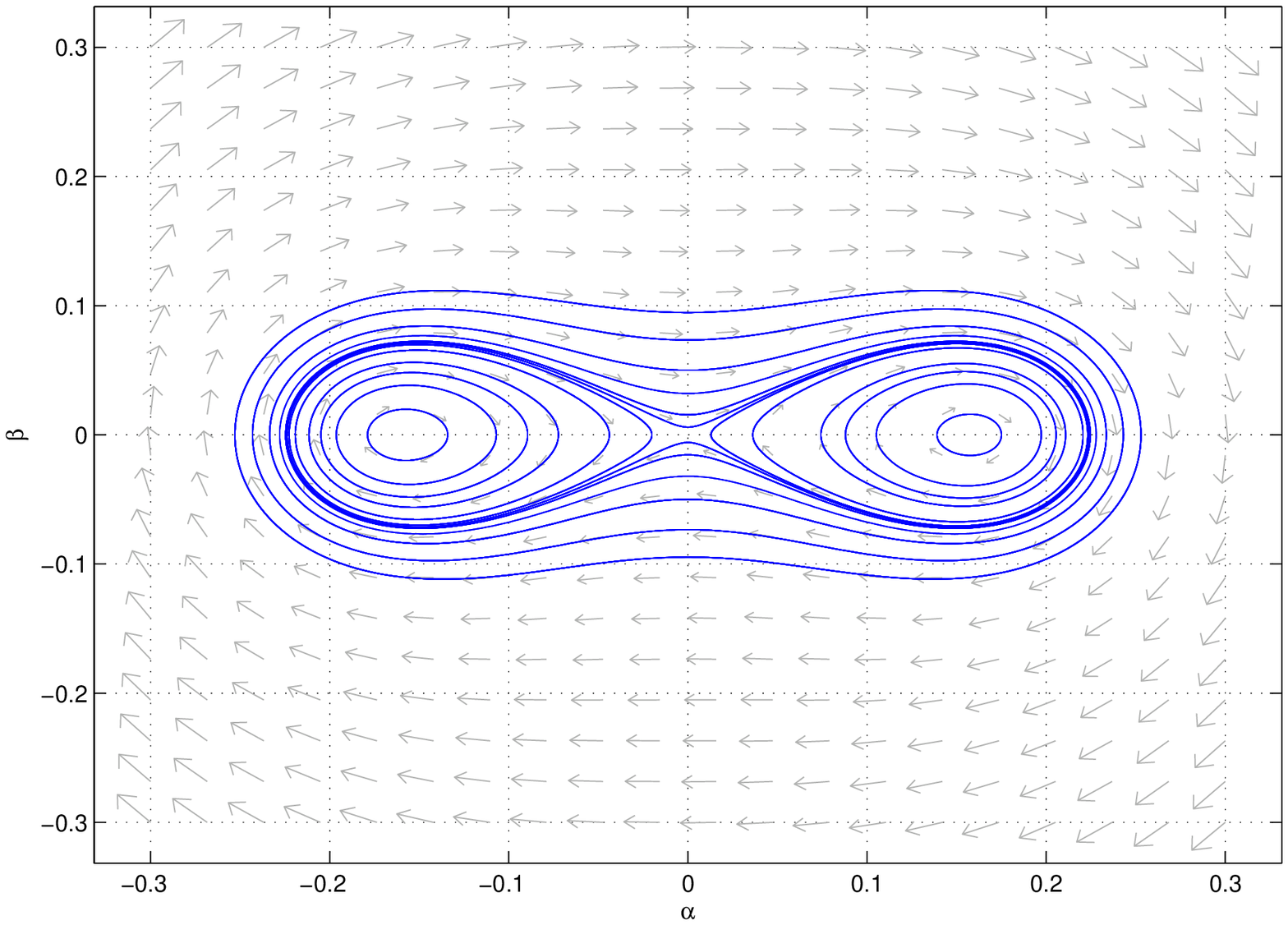}
\caption{Symmetric state for $N-N_{cr}^{FD}<0$ corresponds to an equilibrium elliptic point $(\alpha,\beta)=(0,0)$.   Plotted phase portrait corresponds to  parameter values  $N_{cr} = .1$ and  $N-N_{cr}^{FD} = -.05$ (left) and $N_{cr} = .1$ and  $N-N_{cr}^{FD} = .05$ (right).  It is clear for $N-N_{cr}^{FD} < 0$ the equilibrium sotion is stable and for $N-N_{cr}^{FD} > 0$ the equilibrium sotion is unstable.}
\label{fig7}
\end{figure} 

Furthermore, $\tilde{B}_{-}$ can be diagonalized 
\begin{align}
\tilde{B}_{-} \ &=\ \left[ \begin{array}{ccc} 
2N^{FD}_{cr} & 2N^{FD}_{cr}  & 0 \\
\lambda_+     & \lambda_-           & 0 \\
0 & 0 & 1 
\end{array} \right]\ \left[ \begin{array}{ccc} 
\lambda_+ & 0  & 0 \\
0    & \lambda_-           & 0 \\
0 & 0 & 0
\end{array} \right]\ \left[ \begin{array}{ccc} 
2N^{FD}_{cr} & 2N^{FD}_{cr}  & 0 \\
\lambda_+     & \lambda_-           & 0 \\
0 & 0 & 1 
\end{array} \right]^{-1}\nn
\end{align}
and the linear evolution is given by
\begin{equation}
e^{\tilde{B}_{-} t}\ =\ 
 \left[ \begin{array}{ccc} 
2N^{FD}_{cr} & 2N^{FD}_{cr}  & 0 \\
\lambda_+     & \lambda_-           & 0 \\
0 & 0 & 1 
\end{array} \right]\ \left[ \begin{array}{ccc} 
e^{\lambda_+t} & 0  & 0 \\
0    & e^{\lambda_-  t}         & 0 \\
0 & 0 & 1
\end{array} \right]\ \left[ \begin{array}{ccc} 
2N^{FD}_{cr} & 2N^{FD}_{cr}  & 0 \\
\lambda_+     & \lambda_-           & 0 \\
0 & 0 & 1 
\end{array} \right]^{-1} .
\end{equation}
In other words, we have the bound
\begin{eqnarray*}
e^{\tilde{B}_{-} t} \vec{\delta \alpha} = \left[ \begin{array}{c}
\cos( |\lambda_+| t) \delta \alpha + \left( \frac{N^{FD}_{cr}}{|N-N_{cr}^{FD}|} \right)^{\frac{1}{2}} \sin (|\lambda_+| t) \delta \beta \\
\cos( |\lambda_+| t) \delta \beta + \left( \frac{|N-N_{cr}^{FD}|}{N^{FD}_{cr}} \right)^{\frac{1}{2}} \sin (|\lambda_+| t) \delta \alpha\\
\delta A
\end{array} \right] .
\end{eqnarray*}

The full linearized dynamics are governed by the matrix
\begin{equation}
B_{-} \ =\ \left[ \begin{array}{cccc} 
* & * & *  & 0 \\
* & \tilde{B}_{3\times3} & *&0\\
* & * & *&0\\
0 & 0 &2{N}^{\frac12}          & 0 \\ 
\end{array} \right],
\end{equation}
where $B$ has the same eigenvalues as $\tilde{B}$, with $\lambda=0$ now a generalized eigenvalue of multiplicity two, implying linear growth of $\delta\theta(t)$.  Explicitly, we have
\begin{eqnarray*}
B_{-}^2 = 4 \left[ \begin{array}{cccc} 
N_{cr}^{FD} (N-N_{cr}^{FD}) & 0 & 0 & 0 \\
0 & N_{cr}^{FD} (N-N_{cr}^{FD}) & 0 &0\\
0 & 0 & 0 & 0\\
0 & 0 & 0  & 0 \\ 
\end{array} \right] .
\end{eqnarray*}
Hence, it is clear
\begin{eqnarray*}
e^{B_{-} t} = M_0 + 2 N^{\frac{1}{2}} t M_1 +  \cos( 2 \sqrt{ N_{cr}^{FD} |N-N_{cr}^{FD}|} t) M_2 +  \sin(2 \sqrt{ N_{cr}^{FD} |N-N_{cr}^{FD}|} t) M_3 ,
\end{eqnarray*}
where
\begin{eqnarray*}
M_0 & = &  \left[ \begin{array}{cccc} 
0 & 0 & 0 & 0 \\
0 & 0 & 0 &0\\
0 & 0 & 1 & 0\\
0 & 0 & 0  & 1 \\ 
\end{array} \right], \\
M_1 & = & \left[ \begin{array}{cccc} 
0 & 0 & 0 & 0 \\
0 & 0 & 0 &0\\
0 & 0 & 0 & 0\\
0 & 0 & 1  & 0 \\ 
\end{array} \right], \\
M_2 & = &  \left[ \begin{array}{cccc} 
1 & 0 & 0 & 0 \\
0 & 1 & 0 &0\\
0 & 0 & 0 & 0\\
0 & 0 & 0  & 0 \\ 
\end{array} \right], \\
M_3 & = & \left[ \begin{array}{cccc} 
 0 & \frac{\sqrt{N_{cr}^{FD}}}{\sqrt{|N-N_{cr}^{FD}|}} & 0 & 0 \\
\frac{\sqrt{|N-N_{cr}^{FD}|}}{\sqrt{N_{cr}^{FD}}} & 0 & 0 &0\\
0 & 0 & 0 & 0\\
0 & 0 & 0  & 0 \\ 
\end{array} \right].
\end{eqnarray*}
Note here the leading order behavior is then a system which oscillates at the correct period and grows linearly only in the phase term.  On a component basis, we may say
\begin{eqnarray*}
e^{B_{-} t} \vec{\eta} \approx \vec{\eta} +  \sin(2 \sqrt{ N_{cr}^{FD} |N-N_{cr}^{FD}|} t) \frac{\sqrt{N_{cr}^{FD}}}{\sqrt{N-N_{cr}^{FD}}} \eta_2 \vec{e}_4 + 2 N^{\frac{1}{2}} t \eta_3 \vec{e}_4.
\end{eqnarray*}

Alternatively, we can use the spectrum of $B_{-}$ to define the matrix 
\begin{eqnarray*}
P_{-} = \left[ \begin{array}{cccc}
1 & 1 & 0 & 0 \\
- i \left( \frac{N_{cr}^{FD} - N}{N_{cr}^{FD}} \right)^{\frac{1}{2}} &  i \left( \frac{N_{cr}^{FD} - N}{N_{cr}^{FD}} \right)^{\frac{1}{2}} & 0 & 0 \\
0 & 0 & 1 & 0 \\
0 & 0 & 0 & 1
\end{array} \right].
\end{eqnarray*}
Then, we have
\begin{eqnarray*}
P_{-}^{-1} B_{-} P_{-} = \left[ \begin{array}{cccc}
-2i \left( (N_{cr}^{FD} - N)N_{cr}^{FD} \right)^{\frac{1}{2}} & 0 & 0 & 0 \\
0& 2 i \left( (N_{cr}^{FD} - N)N_{cr}^{FD} \right)^{\frac{1}{2}} & 0 & 0 \\
0 & 0 & 0 & 0 \\
0 & 0 & 2 N^{\frac{1}{2}} & 0
\end{array} \right].
\end{eqnarray*}
Hence, we see
\begin{eqnarray*}
e^{B_{-} t} = P_{-}  \left[ \begin{array}{cccc}
e^{-2i \left( (N_{cr}^{FD} - N)N_{cr}^{FD} \right)^{\frac{1}{2}} t} & 0 & 0 & 0 \\
0& e^{2 i \left( (N_{cr}^{FD} - N)N_{cr}^{FD} \right)^{\frac{1}{2}} t} & 0 & 0 \\
0 & 0 & 1 & 0 \\
0 & 0 & 0 & 1
\end{array} \right] P_{-}^{-1} + 2 N^{\frac{1}{2}} t \left[ \begin{array}{cccc}
0 & 0 & 0 & 0 \\
0&0 & 0 & 0 \\
0 & 0 & 0 & 0 \\
0 & 0 & 1 & 0
\end{array} \right] .
\end{eqnarray*}

{\bf Linearized dynamics about asymmetric equilibrium states, $N>N_{cr}^{FD}$:}

For either asymmetric equilibrium, displayed in (\ref{sym-eq}) we have 
\begin{eqnarray}
\label{eqn:eqanti}
(\alpha^{eq}_{+}, \beta^{eq}_{+}, A^{eq}_{+}, \theta^{eq}_{+} (t)) = \left( \left( \frac{N - N_{cr}^{FD}}{2} \right)^{\frac{1}{2}} ,0,\left( \frac{N + N_{cr}^{FD}}{2} \right)^{\frac{1}{2}},(-2 \Omega_0 + N_{cr}^{FD} - N)t \right).
\end{eqnarray}
Hence,
\begin{equation}
B_{+} \ =\ \left[ \begin{array}{cccc} 
0 & 2(\frac{(N+N_{cr}^{FD})}{2}) & 0  & 0 \\
-2 (N-N_{cr}^{FD})& 0 & \frac{4}{\sqrt{2}} (N-N_{cr}^{FD})^{\frac{1}{2}} (\frac{(N+N_{cr}^{FD})}{2})^{\frac{1}{2}} &0\\
0 & -\frac{2}{\sqrt{2}} (N-N_{cr}^{FD})^{\frac{1}{2}} (\frac{(N+N_{cr}^{FD})}{2})^{\frac{1}{2}} & 0 &0\\
\frac{6}{\sqrt{2}} (N-N_{cr}^{FD})^{\frac{1}{2}} & 0 &2(\frac{(N+N_{cr}^{FD})}{2})^{\frac{1}{2}}   & 0 \\ 
\end{array} \right] .
\end{equation}

In this case, we substitute (\ref{sym-brk}) into the expression for $\tilde{B}$ in (\ref{red-pert-eqn}) and obtain:
\begin{equation}\label{tB-asymm}
\tilde{B}  =  \left[ \begin{array}{ccc} 
0 & N+N^{FD}_{cr}  & 0 \\
-2(N-N_{cr}^{FD}) & 0 & 2\left(N^2-(N_{cr}^{FD})^2\right)^{\frac12} \\
0 &  -\left(N^2-(N_{cr}^{FD})^2\right)^{\frac12} & 0 
\end{array} \right],\ \ \ N>N_{cr}^{FD} =\frac{\Omega_{10}}{2} .
\end{equation}
The eigenvalues of $\tilde{B}$ are:
\begin{equation}
\lambda_0=0,\ \ \lambda_\pm(N)\ =\ \pm\ 2i\ \left(N^2-(N_{cr}^{FD})^2\right)^{\frac12},\ \ \ N>N_{cr}^{FD}.
\label{tB-asymm-eig}
\end{equation}
Therefore, the bifurcating asymmetric states are stable elliptic points.
Also, $\tilde{B}_+$ is diagonalizable, resulting in
\begin{eqnarray*}
e^{\tilde{B}_{+} t} \vec{\delta \alpha} = \left[ \begin{array}{c}
\frac12(1 +\cos( |\lambda_+| t)) \delta \alpha +\left( \frac{N+N^{FD}_{cr}}{N-N_{cr}^{FD}} \right)^{\frac{1}{2}} \sin (|\lambda_+| t) \delta \beta + \frac12 \left( \frac{N+N^{FD}_{cr}}{N-N_{cr}^{FD}} \right)^{\frac{1}{2}} (1- \cos (|\lambda_+| t)) \delta A \\
\cos( |\lambda_+| t) \delta \beta + \left( \frac{|N-N_{cr}^{FD}|}{N+N^{FD}_{cr}} \right)^{\frac{1}{2}} \sin (|\lambda_+| t) \delta \alpha + \sin (|\lambda_+| t) \delta A\\
\frac12 \left( \frac{|N-N_{cr}^{FD}|}{N+N^{FD}_{cr}} \right)^{\frac{1}{2}} (1-\cos ( |\lambda_+| t) ) \delta \alpha + \frac12 \sin ( |\lambda_+| t) \delta \beta +\frac (1 +\cos( |\lambda_+| t)) \delta A 
\end{array} \right] .
\end{eqnarray*}

In a manner analogous to the case of symmetric bound states, we have
\begin{eqnarray*}
B_{+} \left[ \begin{array}{c}
0 \\
0 \\ 
0 \\
1
\end{array} \right] = 0, \ B_{+} \left[ \begin{array}{c}
\left( \frac{N+N_{cr}^{FD}}{N-N_{cr}^{FD}}  \right)^{\frac{1}{2}} \\
0 \\ 
1 \\
0
\end{array} \right] = \left[ \begin{array}{c}
0 \\
0 \\ 
0 \\
1
\end{array} \right] .
\end{eqnarray*}
Hence, it follows
\begin{eqnarray*}
e^{B_{+} t} \vec{\eta} & \approx & \vec{\eta} + \sin (2 \left( N^2 - N^2_{cr} \right)^{\frac{1}{2}} t) \left( N^2 - N^2_{cr}\right)^{\frac{1}{2}} \eta_2 \vec{e}_1 \\
& + & (1- \cos (2 \left( N^2 - N^2_{cr} \right)^{\frac{1}{2}} t)) \left( \frac{N + N_{cr}^{FD}}{N-N_{cr}^{FD}} \right)^{\frac{1}{2}} \eta_3 \vec{e}_1 + ( N+N_{cr}^{FD})^{\frac{1}{2}} (1+ t)  \eta_3 \vec{e}_4 .
\end{eqnarray*}
Once again, we see the dominant behavior be oscillations in $\alpha, \beta, A$ with a linear growth accompanied by a factor of order $N^{\frac{1}{2}} $ in $\theta$.

As before, we can use the spectrum of $B_{+}$ to define the matrix 
\begin{eqnarray*}
P_{+} = \left[ \begin{array}{cccc}
1 & 1 & 1 & 0 \\
-2 i \left( \frac{N - N_{cr}^{FD}}{N+N_{cr}^{FD}} \right)^{\frac{1}{2}} & 2 i \left( \frac{N - N_{cr}^{FD}}{N+N_{cr}^{FD}} \right)^{\frac{1}{2}} & 0 & 0 \\
-\left( \frac{N - N_{cr}^{FD}}{N+N_{cr}^{FD}} \right)^{\frac{1}{2}} & - \left( \frac{N - N_{cr}^{FD}}{N+N_{cr}^{FD}} \right)^{\frac{1}{2}} & \left( \frac{N - N_{cr}^{FD}}{N+N_{cr}^{FD}} \right)^{\frac{1}{2}} & 0 \\
0 & 0 & 0 & 1
\end{array} \right]
\end{eqnarray*}
and
\begin{eqnarray*}
P_{+}^{-1} = \left[ \begin{array}{cccc}
\frac14 & -\frac14 \left( \frac{N + N_{cr}^{FD}}{N-N_{cr}^{FD}} \right)^{\frac{1}{2}} &  -\frac14 \left( \frac{N + N_{cr}^{FD}}{N-N_{cr}^{FD}} \right)^{\frac{1}{2}} & 0 \\
\frac14 &  \frac{1}{4i} \left( \frac{N + N_{cr}^{FD}}{N-N_{cr}^{FD}} \right)^{\frac{1}{2}} &  -\frac14 \left( \frac{N + N_{cr}^{FD}}{N-N_{cr}^{FD}} \right)^{\frac{1}{2}} & 0 \\
\frac12 & 0 & \frac12 \left( \frac{N + N_{cr}^{FD}}{N-N_{cr}^{FD}} \right)^{\frac{1}{2}} & 0 \\
0 & 0 & 0 & 1
\end{array} \right].
\end{eqnarray*}
Then, we have
\begin{eqnarray*}
P_{+}^{-1} B_{+} P_{+} = \left[ \begin{array}{cccc}
-2i \left( (N^2 - (N_{cr}^{FD})^2) \right)^{\frac{1}{2}} & 0 & 0 & 0 \\
0& 2 i \left( (N^2 - (N_{cr}^{FD})^2)  \right)^{\frac{1}{2}} & 0 & 0 \\
0 & 0 & 0 & 0 \\
\frac{4}{\sqrt{2}} (N-N_{cr}^{FD})^{\frac{1}{2}} & \frac{4}{\sqrt{2}} (N-N_{cr}^{FD})^{\frac{1}{2}} & \frac{8}{\sqrt{2}} (N-N_{cr}^{FD})^{\frac{1}{2}} & 0
\end{array} \right].
\end{eqnarray*}
Hence, we see
\begin{eqnarray*}
e^{B_{+} t} & = & P_{+}  \left[ \begin{array}{cccc}
e^{2i \left( (N_{cr}^{FD} + N)(N-N_{cr}^{FD}) \right)^{\frac{1}{2}} t} & 0 & 0 & 0 \\
0& e^{-2 i \left( (N_{cr}^{FD} - N)(N-N_{cr}^{FD}) \right)^{\frac{1}{2}} t} & 0 & 0 \\
0 & 0 & 1 & 0 \\
0 & 0 & 0 & 1 
\end{array} \right] P_{+}^{-1} \\
&& + 4 \sqrt{2}  (N-N_{cr}^{FD})^{\frac{1}{2}} (e^{2i \left( N^2-N^2_{cr} \right)^{\frac{1}{2}} t} -1) P_{+} \left[ \begin{array}{cccc}
0 & 0 & 0 & 0 \\
0&0 & 0 & 0 \\
0 & 0 & 0 & 0 \\
0 & 1  & 0 & 0
\end{array} \right] P_{+}^{-1} \\
&& +  4 \sqrt{2}  (N-N_{cr}^{FD})^{\frac{1}{2}} (e^{-2i \left( N^2-N^2_{cr} \right)^{\frac{1}{2}} t} -1) P_{+} \left[ \begin{array}{cccc}
0 & 0 & 0 & 0 \\
0&0 & 0 & 0 \\
0 & 0 & 0 & 0 \\
1 & 0 & 0 & 0
\end{array} \right] P_{+}^{-1} \\
&& +  8 \sqrt{2}  (N-N_{cr}^{FD})^{\frac{1}{2}} t  \left[ \begin{array}{cccc}
0 & 0 & 0 & 0 \\
0 &0 & 0 & 0 \\
0 & 0 & 0 & 0 \\
\frac12 & 0 & \frac12 \frac{\sqrt{N_{cr}^{FD} + N}}{\sqrt{N-N_{cr}^{FD}}} & 0
\end{array} \right] .
\end{eqnarray*}

\begin{rem}
The analysis of $e^{B_{\pm} t}$ provides partial understanding of the inherent instability in the phase of the finite dimensional periodic orbits.  Specifically, an orbit $(A_1,\alpha_1,\beta_1)$ will oscillate with period $T_1$ while a different orbit $(A_2,\alpha_2,\beta_2)$ will oscillate with period $T_2$.  As a result, if two orbits begin quite close in phase, they will naturally oscillate out of phase with one another, which from the analysis above gives the linear shift in the $\delta \theta$ component.  However, the oscillations will be purely described by the $(A,\alpha,\beta)$ system, meaning we must decouple the phase from the equation to get stability.
\end{rem}

We have can the following
\begin{prop}
\label{prop:perorb}
The system \eqref{eqn:ode-sys-alpha-reduced} has periodic orbit solutions.
\end{prop}

\begin{proof}
The proof follows from standard level set techniques.  Namely, we have the convex geometric curve $H$, which we slice at a particular value of $N$.  Since the $\alpha$, $\beta$ coordinates are constrained to live on that level set, the only possible orbits are those that orbit with a period determined by $N$.  Note, the equilibrium solution is the minimum of $H$ with respect to $N$.

First of all, we have 
\begin{eqnarray*}
H_\alpha & = & 4 \alpha [2\alpha^2 - (N-\frac{\Omega_{10}}{2}) + \beta^2], \\
H_\beta & = & 2 \beta [\Omega_{10} +2 \alpha^2], \\
H_{\alpha, \alpha} & = & 4 [(6 \alpha^2 - (N-\frac{\Omega_{10}}{2}) +
\beta^2] , \\
H_{\beta,\beta} & = & 2[ \Omega_{10} + 2 \alpha^2] , \\
H_{\alpha, \beta} & = & 8 \alpha \beta.
\end{eqnarray*}
Hence, for $N < N_{cr}^{FD}$, we have an equilibrium point only at $(\alpha, \beta) = (0,0)$.  For $N > N_{cr}^{FD}$, we have equilibrium points at $(\alpha, \beta) = (0,0), (\pm \sqrt{ \frac{N - \frac{\Omega_{10}}{2}}{2}}, 0)$.  Using the Hessian, we see that $(0,0)$ is a local minimum for $H$ when $N<N_{cr}^{FD}$, and $(0,0)$ is a local maxima while $(\pm \sqrt{ \frac{N - \frac{\Omega_{10}}{2}}{2}}, 0)$ are local minima when $N<N_{cr}^{FD}$.  Hence, for a fixed $N$, we select a bounded, closed curve for $(\alpha, \beta)$.  By the Poincare-Bendixson Theorem, the possible curves are asymptotic to a limit cycle.  However, the orbit is a closed curve, so it is necessarily periodic.
\end{proof}

\begin{rem}
We note that the dichotomy in the types of equilibria found for $N>N_{cr}^{FD}$ and $N<N_{cr}^{FD}$ exactly describe the phenomenon of symmetry breaking for the profile of the stable bound state for \eqref{eqn:nlsdwp} at small amplitudes as seen in \cite{KKSW}.
\end{rem}

\subsection{Periodic Orbits Near the Equilibrium Point}
\label{sec:per}

In this section we estimate the periodic orbits and describe them near the equilibrium point.

Linearizing about the equilibrium solutions of \eqref{eqn:ode-sys-alpha} for $N-N_{cr}^{FD}>0$, one may define 
\begin{eqnarray}
\alpha & = & \sqrt{\frac{|N-N_{cr}^{FD}|}{2}} + h_1, \\
\beta & = & h_2,
\end{eqnarray} 
where $h_1$, $h_2$ small perturbations.
Plugging this ansatz into \eqref{eqn:ode-sys-alpha} gives
\begin{eqnarray}
\ddot{h}_1 = - 4 (N-N_{cr}^{FD}) \left( N_{cr}^{FD} + \frac{(N-N_{cr}^{FD})}{2} \right) h_1 + \mathcal{O} (|\vec{h}|^2) .
\end{eqnarray}
Hence, the period of oscillation near the equilibrium point is of the form
\begin{eqnarray}
T = \frac{\sqrt{2} \pi}{|N^2-N^2_{cr}|^{\frac{1}{2}}} + \mathcal{O} (|N-N_{cr}^{FD}|^2) = T_{+} (N-N_{cr}^{FD}) + \mathcal{O} (|N-N_{cr}^{FD}|^2).
\end{eqnarray}

Similarly, for $N-N_{cr}^{FD}<0$ we have 
\begin{eqnarray}
\alpha =  h_1, \\
\beta = h_2,
\end{eqnarray} 
which when plugged into \eqref{eqn:sys-fd-pol} gives
\begin{eqnarray}
\ddot{h}_1 = 4 |N-N_{cr}^{FD}| N_{cr}^{FD} h_1 + \mathcal{O} (|\vec{h}|^2).
\end{eqnarray}
Hence, the period of oscillation near the equilibrium point is of the form
\begin{eqnarray}
T = \frac{ \pi}{ (|N-N_{cr}^{FD}| N_{cr}^{FD})^{\frac{1}{2}}} + \mathcal{O} (|N-N_{cr}^{FD}|^2) = T_{-} (N-N_{cr}^{FD}) + \mathcal{O} (|N-N_{cr}^{FD}|^2).
\end{eqnarray}
It is on a time scale of multiple oscillations we hope to control the difference between the observed finite dimensional periodic solutions and the full solution to the PDE.

This leads us to the following:

\begin{prop}\label{prop:perorb-period}
Fix $(N-N_{cr}^{FD}) > 0$ or $(N-N_{cr}^{FD}) < 0$ such that $|N-N_{cr}^{FD}| \ll 1$.  Let us assume $N_{cr}^{FD} \gg |N-N_{cr}^{FD}|$.  Take $(\rho_0^{eq}, \rho_1^{eq})$ be the corresponding equilibrium solution for \eqref{eqn:sys-fd}.  For any periodic solution $\rho_0 (t), \rho_1 (t)$ of \eqref{eqn:sys-fd} such that 
\begin{eqnarray}
| (\rho_0 (t), \rho_1 (t) ) - (\rho_0^{eq}, \rho_1^{eq}) | \ll |N-N_{cr}^{FD}|,
\end{eqnarray}
we have 
\begin{eqnarray}
(\rho_0 (t+T), \rho_1 (t+T) ) = (\rho_0 (t), \rho_1 (t) ) ,
\end{eqnarray}
where 
\begin{eqnarray}
|T - T_{\pm} (N-N_{cr}^{FD}) | \ll |N-N_{cr}^{FD}|^2.
\end{eqnarray}
\end{prop}

\begin{rem}
For small perturbations of the equilibrium point (for either $(N-N_{cr}^{FD})<0$ or $(N-N_{cr}^{FD})>0$, it is for precisely the period $T_\pm (N-N_{cr}^{FD})$ on which we must control the coupling to the continuous spectrum for the full solution to \eqref{eqn:nlsdwp} in order to prove these finite dimensional structures are observable over many oscillations, see Section \ref{sec:persist}.  In order to generalize our result to any periodic solution predicted by the finite dimensional dynamics, we must understand fully the period of each full oscillation.  This will be discussed further in Section \ref{sec:conclusion}.
\end{rem}
Finally, we can state the following

\begin{prop}
Fix $\epsilon > 0$.  There exists a $ \delta > 0$ such that if a given periodic orbit solution of \eqref{syst-rot}, $(A(t),\alpha(t),\beta(t), \theta(t))$, with period $T$ with
\begin{eqnarray}
| (A(0),\alpha (0), \beta (0), 0 ) - (A_{eq}, \alpha_{eq}, \beta_{eq}, \theta_{eq}) | < \delta,
\end{eqnarray}
then we have
\begin{eqnarray}
\| M(t) \|_{L^ \infty \to L^\infty} < (1 + \epsilon) \| e^{B_{eq} t} \|_{L^ \infty \to L^\infty},  \\
\| \tilde{M} (t) \|_{L^ \infty \to L^\infty} < (1 + \epsilon) \| e^{\tilde{B}_{eq} t} \|_{L^ \infty \to L^\infty}.
\end{eqnarray}
\end{prop}

\begin{proof}
The proof follows from continuity of the Floquet multipliers.  We discuss the case for $B_{eq}$ and $M(t)$ here.  The analysis for the case for $\tilde{B}_{eq}$ and $\tilde{M}$ will follow similarly.

It is clear that $(\dot A(t),\dot \alpha(t),\dot \beta(t),\dot \theta (t))$ is a solution to \eqref{eqn:sys-lin-alpha}, giving at least one Floquet multiplier $\lambda_1 = 1$. Similarly, diffentiation with respect to the period gives a similar result, meaning we in fact have $\lambda_1 = \lambda_2 = 1$.

From the analysis of the phase diagram for periodic orbits in $\alpha,\beta$,
we have that $\alpha (t) = \alpha (-t)$ and $\beta (t) = -\beta (-t)$.  Hence
\begin{eqnarray}
\int_0^T \alpha (s) \beta(s) ds = 0.
\end{eqnarray}
So, given $\text{tr} (B(t))$, we know from standard Floquet theory
\begin{eqnarray}
\prod_{j=1}^4 \lambda_j = 1.
\end{eqnarray}

Hence, either $\lambda_3 = \bar{\lambda}_4$ with $|\lambda_3|=|\lambda_4| =1$ or $\lambda_3,\lambda_4 \in \RR$ with $\lambda_3 = \lambda_4^{-1}$.  By continuity of the Floquet multipliers, near $\vec{\alpha}_{eq}$, we have the only the degeneracy at $\lambda = 1$ resulting in a similar growth behavior to that of $e^{B_{eq} t}$ as seen in Section \ref{subsec:stability-finite-dim}.
\end{proof}

\section{An Ansatz for the Coupled System}
\label{sec:ansatz}

Starting with the  finite dimensional system given in \eqref{eqn:alpha-reduced1}, we consider a periodic orbit, $(\tilde{\alpha}(t), \tilde{\beta}(t))$, near equilibrium point  and construct a periodic orbit of the extended system \eqref{eqn:ode-sys-alpha}:
\begin{eqnarray*}
{\sigma}_{*}(t) = \left(\tilde{\alpha}(t) , \tilde{\beta}(t), \tilde{A}(t)\right).
\end{eqnarray*}
Below we shall specify how near the equilibrium $\sigma_*$ need be.  

 We now write the system \eqref{eqn:simpcoupledsystem} by centering around the orbit ${\sigma}_*$ of the finite dimensional truncation:
 \begin{align}
 \sigma(t) &= \sigma_*(t) + \eta,\nn\\
  &\equiv \left(\tilde{A}(t)+ \eta_A(t), \tilde{\alpha}(t)+ \eta_\alpha(t), 
  \tilde{\beta}(t)+\eta_\beta(t) \right).
 \label{eta-def}
\end{align}
This corresponds to a solution of the form:
\begin{eqnarray*}
u(x,t) = e^{i \theta(t)} \left( (\tilde{A}(t) + \eta_A(t) )\ \psi_0 + [(\tilde{\alpha} (t) + \eta_\alpha(t)) + i (\tilde{\beta} (t) + \eta_\beta(t))] \psi_1 + R(x,t)\right) 
\end{eqnarray*}
with initial conditions
\begin{eqnarray*}
u_0(x) = e^{i \theta(0) } \left(\ \tilde{A}(0)\ \psi_0(x) + [\tilde{\alpha}(0)+ i \tilde{\beta}(0)]\ \psi_1(x)\ \right) .
\end{eqnarray*}

Centered about ${\sigma}_*$, the system \eqref{eqn:simpcoupledsystem} becomes:
\begin{eqnarray*}
\dot{\vec{\eta}} & = &D_{{\sigma}} F_{FD} \left({\sigma}_*(t)\right) \vec{\eta} + [\vec{F}_{FD} (\vec{\sigma}_{*} + \vec{\eta}) - \vec{F}_{FD} (\vec{\sigma}_*)- D_{\vec{\sigma}} \vec{F}_{FD} (\vec{\sigma}_*) \vec{\eta} ] + \vec{G}_{FD} (\vec{\sigma}_{*},\vec{\eta};R,\bar{R}), \\
\dot{\theta} & = & -\Omega_0 + A^2 + 3 \alpha^2 + \beta^2 + G_\theta (R, \bar{R}; A, \alpha, \beta), \\
i R_t & = & (H - \Omega_0) R + (\tilde{A}^2 + 3 \tilde{\alpha}^2 + \tilde{\beta}^2) R + F_b \left({\sigma}_{*}+{\eta}\right) + F_R \left( {\sigma}_{*} + {\eta}; R, \bar{R}\right)\nn\\
 &&\ \ \ \ \ \ \  + \left(\ A^2-\tilde{A}^2\ +\ 3(\alpha^2-\tilde{\alpha}^2)\ +\ \beta^2-\tilde{\beta}^2\ \right) R\ .
\end{eqnarray*}

We have the estimates:
\begin{eqnarray*}
\left|\ \vec{F}_{FD} (\vec{\sigma}_{*} + \vec{\eta}) - \vec{F}_{FD} (\vec{\sigma}_{*}) -  D_{\vec{\sigma}} F_{FD} (\vec{\sigma}_*)\vec{\eta}\ \right| = 
\mathcal{O} \left(\ \tilde{A} |\vec{\eta}|^2 + |\vec{\eta}|^3\ \right) .
\end{eqnarray*}
See Appendix \ref{sec:errorestfd} for the explicit expressions 
in the system.\\ 

We denote by  $\tilde{R}$ the leading order part of $R$, driven
by the periodic solution $\sigma_*(t)$:
 \begin{eqnarray}
\label{eqn:tildeR}
i \tilde{R}_t = (H - \Omega_0) \tilde{R} + (\tilde{A}^2 + 3 \tilde{\alpha}^2 + \tilde{\beta}^2) \tilde{R} + P_c F_b(\vec{\sigma}_{*}) .
\end{eqnarray}
The correction to $\tilde{R}$ is given by $w$, which satisfies:
\begin{equation}
R = \tilde{R} + w.
\label{tR-def}
\end{equation}
Introduce, $\tilde{M}(t)$, a fundamental solution matrix for the system of ODEs with time-periodic coefficients:
\begin{equation}
\partial_t\eta = D_\sigma F(\sigma_*(t))\ \eta.
\nn\end{equation}. 

We shall study  the following system of integral equations
 for $\eta(t),\ \theta(t), w(x,t)$:
\begin{align}
\vec{\eta} (t) &= \int_0^t \tilde{M}(t) \tilde{M}^{-1} (s) \left[ (\vec{F}_{FD} (\vec{\sigma}_{*}+{\eta}) -  \vec{F}_{FD} (\vec{\sigma}_{*}) - D_\sigma \vec{F}_{FD} (\vec{\sigma}_{*}) \vec{\eta})+ \vec{G}_{FD} (\vec{\sigma}_{*},\vec{\eta};R,\bar{R}) \right] ds ,\nn \\
\label{eqn:intsys}
\theta &=  \theta_0 + \int_0^t \left[-\Omega_0 + A^2 + 3 \alpha^2 + \beta^2 + G_\theta (R, \bar{R}; \vec{\sigma}_{*},\vec{\eta}) \right] ds, \\
w(t) &= \int_0^t e^{iH(t-s)-i \Omega_0 (t-s) +i \int_s^t (\tilde{A}^2 + 3 \tilde{\alpha}^2 + \tilde{\beta}^2) (s') ds'} P_c \left[ (F_b (\vec{\sigma}_{*} + \vec{\eta}) - F_b (\vec{\sigma}_{*})) + 
F_R(\vec{\sigma}_{*},\vec{\eta};R,\bar{R}) \right] . \nn
\end{align}

 We view a solution, $(\vec{\eta},R)$, of this system of integral equations  as fixed point of a mapping, $\mathcal{M}$:
 \begin{equation}
 (\vec{\eta},R) = \mathcal{M} (\vec{\eta},R).
 \label{fixedpoint}
 \end{equation}
 
 In the next section we formulate and solve this fixed point problem in a function space, which yields the existence of solutions to NLS/GP which ``shadow'' the periodic orbit, $\sigma_*(t)$, on the time scale of many periods.

\section{Main Results}
\label{sec:persist}

Recall that the period of the orbits described in  Proposition \ref{prop:perorb} 
satisfies
\begin{equation}
T_{period} \sim (|N_{cr}^{FD} - N| N_{cr}^{FD})^{-\frac12}.
\nn\end{equation}
 In this section we shall construct solutions to the full PDEs, which ``shadow'' the finite dimensional orbits for many periods.
 
Let $\dd>0$ denote a number to be chosen sufficiently small. And define the region of parameter space, about the symmetry breaking point,  in which we conduct our study by
\begin{align}
\left| N_{cr}^{FD}-N \right|\ &=\ \dd, \nn\\
N_{cr}^{FD}\ &=\ \dd^\gamma.\label{periodN} 
\end{align}
For now $0<\gamma<1$, but  we shall place
further constraints will be placed on $\gamma$. 

\begin{rem}\label{remark:dd-small}
Recall that $\dd=N_{cr}^{FD}\sim \mathcal{N}_{cr}(L)$, where $L$ is the well-spacing parameter for the double-well. Since, for $L$ sufficiently large,  $ \mathcal{N}_{cr}(L)\ \sim\ \Omega_1(L)-\Omega_0(L)\ \lesssim e^{-\kappa L}$, the eigenvalues splitting ,
 and $\dd$ can be made small by choosing $L$ sufficiently large.
 \end{rem}

In terms of $\dd$, we now describe the periodic solutions discussed in Propositions \ref{prop:perorb} and \ref{prop:perorb-period}:
\begin{prop}\label{prop:delta-perorb} 
There exists $\dd_0>0$, such that for $0<\dd<\dd_0$,  the system \eqref{eqn:alpha-reduced1} has periodic solutions, which are small perturbations of the equilibria:
\begin{align}
|\tilde{A}(t)|^2 &\sim \dd^\gamma,\ \  |\tilde{\alpha}(t)|^2 +
|\tilde{\beta}(t)|^2\ \sim \dd,\ \ \ N>N_{cr}^{FD}, \\
|\tilde{A}(t)|^2 &\sim \dd^\gamma,\ \  |\tilde{\alpha}(t)|^2 +
|\tilde{\beta}(t)|^2\ \sim \dd^{1+\delta},\ \ \ N<N_{cr}^{FD} ,
\label{orbit-size-tau}
\end{align}
where $\delta>0$ is to be chosen below.

The period of oscillations of  the periodic solutions of Proposition \ref{prop:perorb} is
\begin{equation}
T_{period}(\dd) \sim \dd^{-\frac{1+\gamma}{2}} .
\label{Tdelta}\end{equation}
\end{prop}

We shall seek to contruct solutions on a time interval of the form
\begin{equation}
0\le t\le T_*(\dd),\ \ \ \ T_*(\dd)\sim \dd^{-\epsilon}\times T_{period}(\dd)\ =\ \dd^{-\frac{1+\gamma}{2} -\epsilon} .
\label{Texist}
\end{equation}
\subsection{Notation}
It should be noted, notationally when we refer to
\begin{eqnarray}
A \lesssim B,
\end{eqnarray}
we mean 
\begin{eqnarray}
A \leq C B
\end{eqnarray}
for some $C$ of order $1$.  Also, for
\begin{eqnarray}
A \ll  B,
\end{eqnarray}
we mean
\begin{eqnarray}
A \leq c B
\end{eqnarray} 
for some $c>0$ much less than $1$.


\subsection{Statement of Theorem}
\label{sec:thm}

\begin{thm}
\label{thm:main-eq}
Assume $0<\dd<\dd_0$ and $\frac{7}{9} < \gamma < 1$.\\
  Denote by 
\begin{eqnarray*}
{\sigma}_{*}(t)\ =\ \left(\ \tilde{A}(t), \tilde{\alpha}(t) , \tilde{\beta}(t)\ \right), 
\end{eqnarray*}
a periodic solution of  \eqref{eqn:ode-sys-alpha}, for which:
\begin{enumerate}
\item  whose period $T_{period}(\dd) $ satisfies
\begin{eqnarray*}
T_{period}(\dd) \lesssim (|N_{cr}^{FD} - N| N_{cr}^{FD})^{-\frac12} 
 = \dd^{-\frac{1+\gamma}{2}}
 \end{eqnarray*}
 \item whose fundamental matrix, $\tilde{M}(t)$ of the linearized dynamics about $\sigma_*(t)$ satisfies the norm bound:
\begin{eqnarray*}
0 < s,t < T_{period}(\dd)\ \implies \| \tilde{M}(t) \tilde{M}^{-1} (s) \| \leq C\ \left( \frac{N+N_{cr}^{FD}}{|N-N_{cr}^{FD}|}\right)^{\frac12}\ =\ C\ \dd^{\frac{\gamma-1}{2}}, 
\end{eqnarray*}
\item and which is a small orbit about an equilibrium $\left(\ A^{eq},\alpha^{eq},\beta^{eq}\ \right)$ as in Proposition \ref{prop:delta-perorb}. 
See figures \ref{fig:proof-ab} and \ref{fig:proof-bb}.
\end{enumerate}
Take initial data for NLS/GP of the form
\begin{eqnarray}
u_0(x) = e^{i \theta (0) } \left(\  \tilde{A} (0) \psi_0(x)\ +\  [\tilde{\alpha} (0) + i \tilde{\beta}(0)] \psi_1(x)\  \right) ,
\end{eqnarray}
where $\theta(0)\in\RR$ is chosen arbitrarily.

Then there exists a solution $u(x,t)$ of \eqref{eqn:nlsdwp} 
 of the form 
 \begin{eqnarray}
u(x,t) = e^{i \theta (t)} \left( (\tilde{A} (t) + \eta_A(t)) \psi_0(x) + [(\tilde{\alpha} (t) + \eta_\alpha(t)) + i ( \tilde{\beta} (t) + \eta_\beta)] \psi_1(x)  + \tilde{R} (x,t) + w(x,t) \right),
\nn
\end{eqnarray}
where $\tilde{R}$ as in \eqref{eqn:tildeR}.

\noindent Furthermore,   ${\eta}(t)\equiv\left(\ \eta_A(t),\eta_\alpha(t),\eta_\beta(t)\ \right),\ \theta(t) \in C^1([0,T_*(\dd)])$,  $w \in L^\infty_t H^1_x \cap L^4_t L^\infty_x$  and $w$ satisfies the bounds
\begin{eqnarray}
\| {\eta} \|_{L^\infty_t[0,T_*(\dd)]} + \| w \|_{ L^\infty_t( [0,T_*(\dd)]; H^1_x)} + \| w \|_{ L^4_t( [0,T_*(\dd)]; L^\infty_x)} \lesssim \tau^{\frac{1}{2}+\delta_1} ,
\end{eqnarray}
for all $t \in I =[0,T_*(\dd)]=[0,T_{period}(\dd)\ \dd^{-\epsilon}]$ where $\epsilon>0$ is specified in the proof.
\end{thm}

\begin{figure}
\scalebox{.75}{\input{proof-ab.tex}}
\caption{Let $N > N_{cr}^{FD}$.  By Theorem \ref{thm:main-eq},  for data $(\tilde{\alpha}_0,\tilde{\beta}_0)$ on the   periodic orbit 
$(\ \tilde{\alpha}(t),\tilde{\beta}(t)\ )$ and such that 
 $|\tilde{\alpha}_0-\alpha_{eq}|^2 + |\tilde{\beta}_0|^2 \sim \dd^{1+\delta}$, the solution $u(x,t)$  evolves within the  dashed annulus of width $\dd^{1+\delta_1},\ \ \delta_1>\delta$. } 
\label{fig:proof-ab}
\end{figure}

\begin{figure}
\scalebox{.75}{\input{proof-bb.tex}}
\caption{Let $N < N_{cr}^{FD}$.  By Theorem \ref{thm:main-eq},  for data $(\tilde{\alpha}_0,\tilde{\beta}_0)$ on the   periodic orbit 
$(\ \tilde{\alpha}(t),\tilde{\beta}(t)\ )$ and such that 
 $|\tilde{\alpha}_0|^2 + |\tilde{\beta}_0|^2 \sim \dd^{1+\delta}$, the solution $u(x,t)$  evolves within the  dashed annulus of width $\dd^{1+\delta_1}, \ \ \delta_1>\delta.$ } 
\label{fig:proof-bb}
\end{figure}


\subsection{Proof of theorem}
\label{sec:proof}

In the sequel, we will often use the contracted notation
\begin{eqnarray*}
L^p_t W^{k,q}_x = L^p ([0,T^*] ; W^{k.q} (\mathbb{R})),
\end{eqnarray*}
where $T^*$ is given by \eqref{Texist}, \eqref{periodN}.

We now show that the map 
\begin{eqnarray*}
(\vec{\eta}, w ) = \mathcal{M} (\vec{\eta}, w)
\end{eqnarray*}
defined in section \ref{sec:ansatz} is a contraction.

Define the  space 
\begin{eqnarray*}
X(I) = X([0,T_*(\dd)]) = \left\{\ (\eta,w): \eta\in L^\infty_t([0,T_*(\dd)]),\ \ w\in L^\infty_t([0,T_*(\dd)]; H^1_x) \cap L^4_t([0,T_*(\dd)]; L^\infty_x)
\ \right\}
\end{eqnarray*}
equipped with the natural norm
\begin{eqnarray*}
\| (\vec{\eta},w) \|_{X(I)} = \| {\eta} \|_{L^\infty_t(I)} + \| w \|_{L^4_t(I; L^\infty_x)} + \| w \|_{L^\infty_t(I; H^1_x)},
\end{eqnarray*}
 \begin{equation}
{\rm where}\ \ \ \  I=[0,T_*(\dd)].
 \nn\end{equation}
  We define $B_\dd (I) \subset X(I)$ such that $(\vec{\eta},R) \in B_\dd (I)$ if and only if
\begin{eqnarray}
\label{eqn:contmapball}
\| (\eta,w) \|_{X(I)}  \le \dd^{\frac{1}{2}+\delta_1} ,
\end{eqnarray}
 where $\delta_1> 0$ will be chosen later. 
 
We must prove the following:
\begin{prop} 
  The mapping $\mathcal{M} : X(I) \to X(I)$, defined in \eqref{fixedpoint},  
has the properties
\begin{enumerate}
\item $\mathcal{M}: B_\dd(I) \to B_\dd(I)$.
\item There exists $\kappa<1$ such that given
$(\vec{\eta}_j,w_j) \in  B_\dd(I)$ for $j = 1,2$, we have
\begin{eqnarray*}
d(\mathcal{M}(\vec{\eta}_1,w_1),\mathcal{M}(\vec{\eta}_2,w_2)) \leq\ \kappa\  d((\vec{\eta}_1,w_1) , (\vec{\eta}_2,w_2)).
\end{eqnarray*}
\end{enumerate} 
Thus, there exists a unique solution $(\vec{\eta},w)$ in 
$B_\dd(I)$.
\end{prop}

\begin{proof}
We begin by proving necessary bounds on $\tilde{R}$. 

\begin{prop}
\label{lem:tildeR} 
Given a solution $\vec{\sigma}_{*} = (\tilde{A}(t), \tilde{\alpha}(t), \tilde{\beta}(t))$ to \eqref{eqn:alpha-reduced1}, as in Proposition \ref{prop:delta-perorb}, 
there exists $\dd_0>0, \delta_0>\delta$, such that for all $\dd<\dd_0$ we have the following:\\
 for all $t \in I$,  and defining
\begin{eqnarray}
\tilde{R} = \int_0^t e^{-iH(t-s)-i \Omega_0 (t-s) +i \int_s^t (\tilde{A}^2 + 3 \tilde{\alpha}^2 + \tilde{\beta}^2) (s') ds'}\ P_b F(\vec{\sigma}_{*}) ds
\end{eqnarray}
for $F_b$ defined in \eqref{eqn:sys-id}, we have
\begin{eqnarray*}
\| \tilde{R} \|_{L^\infty_t(I;L^\infty_x)} \lesssim \dd^{1+\delta_0}.
\end{eqnarray*}
\end{prop}

\begin{proof}
Note that it is only  the magnitude of the components of ${\sigma}_{*}$, via $F(\sigma_*(t))$, which  factor in to the bounds of $\tilde{R}$.

From Proposition \ref{prop:delta-perorb} and the finite dimensional conservation laws

\begin{eqnarray}
\tilde{A}(t) = \dd^{\frac{\gamma}{2}}\ +\  \epsilon_0 (t),
\end{eqnarray}
where $|\epsilon_0| \sim \dd^{\frac{1+\delta}{2}}$.
Based on the expansion in the ansatz, we have
\begin{eqnarray*}
F_b ({\sigma}_*) & = & P_c [\tilde{A}^3 \psi_0^3 + (\tilde{A}^2 (\tilde{\alpha}-i \tilde{\beta})+2 \tilde{A}^2 (\tilde{\alpha}+i \tilde{\beta}) ) \psi_0^2 \psi_1 + ( (\tilde{\alpha} + i \tilde{\beta})^2 \tilde{A} + 2 \tilde{A} (\tilde{\alpha}^2 + \tilde{\beta}^2) ) \psi_0 \psi_1^2 \\
& + & (\alpha^2 + \beta^2)(\alpha + i \beta) \psi_1^3].
\end{eqnarray*}
The term of largest order in this expansion is 
\begin{eqnarray}
\tilde{A}^3 \psi_0^3.
\end{eqnarray}
The bounds on the remaining terms will follow similarly, so we look at
\begin{eqnarray}
\int_0^t e^{-i(H-\Omega_0)(t-s) - i \int_s^t (\tilde{A}^2 + 3 \tilde{\alpha}^2 + \tilde{\beta}^2) (s') ds'} P_c ( \dd^{\frac{\gamma}{2}} + \epsilon_0)^3 \psi_0^3 ds\nn\\
\sim\ \dd^{\frac{3\gamma}{2}} \int_0^t e^{-i(H-\Omega_0)(t-s) - i \int_s^t (\tilde{A}^2 + 3 \tilde{\alpha}^2 + \tilde{\beta}^2) (s') ds'} P_c  \psi_0^3 ds .\nn
\end{eqnarray}
In particular, we will show there exists $\delta_0>\delta$, such that 
\begin{eqnarray}
\label{eqn:n}
\| \dd^{\frac{3 \gamma}{2}} \int_0^t e^{-i(H-\Omega_0)(t-s) - i \int_s^t (\tilde{A}^2 + 3 \tilde{\alpha}^2 + \tilde{\beta}^2) (s') ds'} P_c  \psi_0^3 ds \|_{L^\infty_{t,x}} \lesssim \dd^{1+\delta_0}.
\end{eqnarray}

From \eqref{eqn:ode-sys-alpha}, we know
\begin{eqnarray}
\dot{\tilde{\theta}} = -\Omega_0 + \tilde{A}^2 + 3 \tilde{\alpha}^2 (t) + \tilde{\beta}^2 .
\end{eqnarray}
Hence, the leading order constant terms from the derivative are $N_{cr}^{FD} - \Omega_0$.  We write
\begin{eqnarray*}
\dd^{\frac{3 \gamma}{2}}\ \int_0^t e^{-i(H-\Omega_0)(t-s)-i \int_s^t (\tilde{A}^2 + 3 \tilde{\alpha}^2 + \tilde{\beta}^2) (s') ds'} P_c  \psi_0^3 ds = \\
\dd^{\frac{3 \gamma}{2}}\ e^{-iHt-i \tilde \theta (t)} \int_0^t e^{i \tilde \theta (s) + i\Omega_0 s - i N_{cr}^{FD} s}  e^{i Hs - i \Omega_0 s+ i N_{cr}^{FD} s}  P_c  \psi_0^3 ds = \\
\dd^{\frac{3 \gamma}{2}}\ e^{-iHt-i \tilde \theta (t)} \int_0^t e^{i \tilde\theta (s) + i\Omega_0 s - i N_{cr}^{FD} s} \frac{d}{ds} \frac{1}{i} (H-\Omega_0+N_{cr}^{FD})^{-1}  e^{i Hs - i \Omega_0 s+ i N_{cr}^{FD} s}  P_c  \psi_0^3 ds,
\end{eqnarray*}
where the resolvent $(H-\Omega_0+N_{cr}^{FD})^{-1}$ is a well-defined operator in $H^1$ on $P_c  \psi_0^3$, see Appendix \ref{sec:est}.
Using the estimates from Appendix \ref{sec:est} and integration by parts, we have
\begin{eqnarray*}
\dd^{\frac{3 \gamma}{2}} && \!\!\!\!\!\!\!\!\!\!\!\!\! \left|  e^{-iHt-i \theta (t)} \int_0^t  e^{i \tilde \theta(s) + i\Omega_0 s - i N_{cr}^{FD} s} e^{i Hs - i \Omega_0 s + i N_{cr}^{FD} s}  P_c  \psi_0^3 ds \right| \\
&\lesssim&  \dd^{\frac{3 \gamma}{2}} \left| (H-\Omega_0+N_{cr}^{FD})^{-1} P_c  \psi_0^3 \right|  
+ \dd^{\frac{3 \gamma}{2}} \left| e^{-iHt} e^{-i \tilde \theta (t)} (H-\Omega_0+N_{cr}^{FD})^{-1} P_c  \psi_0^3 \right| \\
&&+ \dd^{\frac{3 \gamma}{2}} \mathcal{O} (\dd) \ \int_0^t \| e^{iHs} (H-\Omega_0+N_{cr}^{FD})^{-1} P_c  \psi_0^3 \|_{L^\infty} ds \\
&\lesssim&  \dd^{\frac{3 \gamma}{2}} \| (H-\Omega_0+N_{cr}^{FD})^{-1} P_c  \psi_0^3 \|_{H^1}  
+ \dd^{\frac{3 \gamma}{2}} \| e^{-iHt} e^{-i \tilde \theta (t)} (H-\Omega_0+N_{cr}^{FD})^{-1} P_c  \psi_0^3 \|_{H^1} \\
&&+ \dd^{\frac{3 \gamma}{2}} \mathcal{O} (\dd) \ \int_0^t \| e^{iHs} (H-\Omega_0+N_{cr}^{FD})^{-1} P_c  \psi_0^3 \|_{H^1} ds \\
&\lesssim&  \dd^{\frac{3 \gamma}{2}}  + \dd^{\frac{3 \gamma}{2}} 
+ \dd^{\frac{3 \gamma}{2}} \mathcal{O} (\dd)\ t.
\end{eqnarray*}
By selecting $\gamma > \frac{2}{3}$, we ensure that all terms resulting from integration by parts are bounded by $\dd^{1+\delta_1}$ for all $t \in I = [0,\dd^{-\frac{1+\gamma}{2} - \epsilon}]$.  
\end{proof}

Using dispersive estimates, we have additional bounds for $\tilde R$ given by the following
\begin{lem}
\label{lem:tildeRstr}
Given $\tilde{R}$ defined as in \eqref{eqn:tildeR}, we have
\begin{eqnarray*}
\| \tilde{R} \|_{H^1} \lesssim |\tilde{A}|^3 |I|^{\frac34}
\end{eqnarray*}
and
\begin{eqnarray*}
\| \tilde{R} \|_{L^p (I; L^q_x)} \lesssim |\tilde{A}|^3 |I|^{\frac34}
\end{eqnarray*}
 for any Strichartz pair $(p,q)$.
\end{lem}

\begin{proof}
To begin, we first note that by the $L^2$ boundedness of wave operators discussed in Appendix \ref{sec:est}, we need only consider 
\begin{eqnarray*}
\| \langle H \rangle \tilde{R} \|_{L^2_x}.
\end{eqnarray*}

As in Lemma \ref{lem:tildeR}, we look at the the worst term, namely
\begin{eqnarray*}
\| \langle H^s \rangle \tilde{R} \|_{L^\infty L^2} & \leq & \| \int_0^t e^{i H(t-s)} \tilde{A}^3 (s) \langle H^s \rangle  P_c \psi_0^3 ds \|_{L^\infty L^2} \\
& \leq & \dd^{\frac{3 \gamma}{2}} \| \psi_0^3 \|_{L^1} |I|^{\frac34} \\
& \lesssim & \dd^{\frac{3 \gamma}{2}} |I|^{\frac34} ,
\end{eqnarray*}
where we have used the Strichartz estimate \eqref{eqn:strich2} with dual Strichartz norm $L^\frac43 (I; L^1)$. 

The Strichartz estimate follows similarly.
\end{proof}

\begin{rem}
Though such estimates do not arise in this work, let us also point out the following simple estimate
\begin{eqnarray}
| \langle \tilde{R}, \chi_1 \rangle | & \lesssim & \dd^{\frac{3 \gamma}{2}} \int_0^t \langle \chi_1, e^{-i H (t-s)} \chi_2 \rangle ds \\
& \lesssim & \dd^{\frac{3 \gamma}{2}} \int_0^t \langle t-s \rangle^{-\frac{3}{2}} ds \lesssim \dd^{\frac{3 \gamma}{2}} t^{-\frac{1}{2}} ,
\end{eqnarray}
which is proved in \cite{Sch} and may be applicable when trying to prove long or infinite time results on similar problems to the one studied here.
\end{rem}
  
Note, by standard Sobolev embeddings and the contraction assumption, we have
\begin{eqnarray}
\| R \|_{ L^\infty_{t,x} } \lesssim \dd^{\frac{1}{2} + \delta_1}.
\end{eqnarray}

Since we have proper bounds on $\tilde{R}$, we must now bound
\begin{eqnarray*}
\| \vec{\eta} \|_{L^\infty} & = & \| \int_0^t M (t) M^{-1}(s) \big[ \vec{F}_{FD} (\vec{\sigma}_{*} + \vec{\eta}) -  \vec{F}_{FD} (\vec{\sigma}_{*}) - D_{\sigma^*} \vec{F}_{FD} (\vec{\sigma}_{*}) \vec{\eta}  \\
&&  + \vec{G}_{FD} (\vec{\sigma}_{*} + \vec{\eta};R,\bar{R}) \big]  ds \|_{L^\infty}
\end{eqnarray*}
and
\begin{eqnarray*}
\| w \|_{L^\infty H^1 \cap L^4 L^\infty} & = & \| \int_0^t e^{iH(t-s)}  e^{i \Omega_0 (t-s)} e^{i \int_s^{t} (\tilde{A}^2 + 3 \tilde{\alpha}^2 + \tilde{\beta}^2) (s') ds'} \\
&& \times P_c \left[ (F_b (\vec{\sigma}_{*}+\vec{\eta}) - F_b (\vec{\sigma}_{*})) + F_R (\vec{\sigma}_{*},\vec{\eta};R,\bar{R}) \right] ds \|_{L^\infty H^1 \cap L^4 L^\infty} .
\end{eqnarray*}

From Appendix \ref{sec:est}, we have for any Strichartz pair $(p,q)$ that
\begin{eqnarray}
\| \int_0^t e^{iH(t-s)} P_c f \|_{L^p W^{1,q}} \lesssim \| f(x,t) \|_{L^{\tilde{p}}_t W^{1,\tilde{q}}_x},
\end{eqnarray}
where $(\tilde{p},\tilde{q})$ is a dual Strichartz pair.  In one dimension, it is useful to note we may take $\tilde{p} = \frac{4}{3}+\mu_1$ and $\tilde{q} = 1+\mu_2$ for $\mu_j > 0$ small, $j = 1,2$.  In other words, we can be as close to the endpoint estimate of $\tilde{p} = \frac{4}{3}$ and $\tilde{q} = 1$ as we need to be.   

By reducing the system and including the lowest order terms from the expansions above, we can reduce to controlling a model problem of the form
\begin{eqnarray*}
\eta &=& \int_0^t \left( \frac{N+N_{cr}^{FD}}{|N-N_{cr}^{FD}|}  \right)^{\frac12} [ \tilde{A} \eta^2 + \eta^3 \\
&+& \tilde{A}^2 \langle \tilde{R}, \chi \rangle + \tilde{A}^2 \langle w, \chi \rangle + 2 \tilde{A} \eta \langle \chi, \tilde{R} \rangle + \eta^2 \langle \tilde{R}, \chi \rangle \\
&+& \tilde{A} \eta \langle w, \chi \rangle + \eta^2 \langle w, \chi \rangle  ] ds \\
& = &  {Term}_1^\eta + {Term}_2^\eta + \dots + {Term}_8^\eta] ,
\end{eqnarray*}
where $\chi \in \mathcal{S}$, and
\begin{eqnarray*}
w & = & \int_0^t e^{iH(t-s)}  e^{i \Omega_0 (t-s)} e^{i \int_s^{t} (\tilde{A}^2 + 3 \tilde{\alpha}^2 + \tilde{\beta}^2) (s') ds'} \\
& \times & P_c [(3 \tilde{A}^2 \eta \psi_0^3 + 2 \tilde{A} \tilde{R} \eta + \tilde{A} \psi_0^2 \tilde{R} \eta) \\ 
&+& (\tilde{A} ^2 \psi_0^2 w + w^3  + 2 \tilde{A} \eta w + \tilde{A} \langle \chi, w \rangle \tilde{R} + \tilde{A} \langle \chi, \tilde{R} \rangle w ) \\
& + & (\tilde{A}^2 \psi_0^2 \tilde{R} + \tilde{A} \langle \chi, \tilde{R} \rangle \tilde{R} ) ] ds \\
& = &  {Term}_1^w + {Term}_2^w + \dots + {Term}_{10}^w.
\end{eqnarray*}

For our model problem, we now take $(\eta,w) \in X$, $\| (\eta,w) \|_{X} \leq \dd^{\frac{1}{2}+\delta_1}$.  In addition, take $(\eta_j,w_j) \in X$, $\| (\eta_j,w_j) \|_{X} \leq \dd^{\frac{1}{2}+\delta_1}$ for $j=1,2$.

It follows
\begin{eqnarray*}
\| {Term}_1^\eta \|_{L^\infty} & = & \left( \frac{N+N_{cr}^{FD}}{|N-N_{cr}^{FD}|}  \right)^{\frac12} \| \int_0^t \tilde{A} (s) \eta^2 (s) ds \|_{L^\infty} \\ 
& \leq & \dd^{\frac{\gamma}{2} - \frac{1}{2} + \frac{\gamma}{2} - \frac{1+\gamma}{2}-\epsilon} \| \eta \|^2_{L^\infty}\ \\
& \leq & \dd^{\frac{\gamma }{2} + 2\delta_1 -\epsilon}\ \le\ \tau^{\frac{1}{2}+\delta_1}\end{eqnarray*}
provided we choose
 \begin{equation}
 \delta_1 > \frac{1-\gamma}{2}+\epsilon.
 \label{constraint-1}
 \end{equation}
  In addition, it is clear from the analysis there that
\begin{eqnarray*}
\| {Term}_1^{\eta_1} - {Term}_1^{\eta_2} \|_{L^\infty} & \leq &
\dd^{\frac{\gamma - 1}{2} + \delta_1 -\epsilon} \| \eta_1 - \eta_2
\|_{L^\infty} .
\end{eqnarray*} 

Next, for ${Term}_2^\eta$ we get
\begin{eqnarray*}
\| {Term}_2^\eta \|_{L^\infty} & = & \left( \frac{N+N_{cr}^{FD}}{|N-N_{cr}^{FD}|} \right)^{\frac12} \| \int_0^t \eta^3 (s) ds \|_{L^\infty} \\ 
& \leq & \dd^{\frac{\gamma}{2} - \frac{1}{2} - \frac{1+\gamma}{2} - \epsilon} \| \eta \|^3_{L^\infty}  \\
& \leq &  \dd^{\frac{1}{2}+3\delta_1-\epsilon}
\end{eqnarray*}
provided 
\begin{equation}
2\delta_1>\epsilon.
\label{constraint-2}
\end{equation}
Similarly, we have
\begin{eqnarray*}
\| {Term}_2^{\eta_1} - {Term}_2^{\eta_2} \|_{L^\infty} & \leq & \dd^{2\delta_1-\epsilon}\ \| \eta_1 - \eta_2 \|_{L^\infty}.
\end{eqnarray*}

For ${Term}_3^\eta$, we have
\begin{eqnarray*}
\| {Term}_3^\eta \|_{L^\infty} & = & \left( \frac{N+N_{cr}^{FD}}{|N-N_{cr}^{FD}|} \right)^{\frac12} \| \int_0^t \tilde{A}^2 \langle \tilde{R}, \chi \rangle ds \|_{L^\infty} \\ 
& \leq & \dd^{\frac{\gamma}{2} - \frac{1}{2} + \gamma + 1-\frac12 - \frac{\gamma}{2} - \epsilon} 
  \sim \tau^{\gamma-\epsilon} \le  \dd^{\frac{1}{2}+\delta_1}
\end{eqnarray*}
provided 
\begin{equation}
\gamma -\epsilon > \frac{1}{2}+\delta_1.
\label{constraint-3}
\end{equation}

  Since ${Term}_3^\eta$ is independent of $(\eta,w)$, it follows easily that
\begin{eqnarray*}
\| {Term}_3^{\eta_1} - {Term}_3^{\eta_2} \|_{L^\infty} = 0,
\end{eqnarray*}
meaning this term does not factor into the contraction.

The bound on ${Term}_4^\eta$ is of the form
\begin{eqnarray*}
\| {Term}_4^\eta \|_{L^\infty}  & = & \left( \frac{N+N_{cr}^{FD}}{|N-N_{cr}^{FD}|} \right)^{\frac12} \| \int_0^t \tilde{A}^2 \langle w, \chi \rangle ds \|_{L^\infty} \\ 
& \leq & \dd^{\frac{\gamma}{2} - \frac{1}{2} + \gamma } \| \chi \|_{L^{\frac43} L^1} \| w \|_{L_t^4 L_x^\infty}  \\
& \leq & \dd^{ \frac{\gamma}{2} - \frac{1}{2} + \gamma -\frac38 -\frac{3\gamma}{8} - \frac{3 \epsilon}{4} } \| w \|_{L^4 L^\infty}  \\
& \leq & \dd^{ \frac{9\gamma}{8} - \frac{7}{8} - \frac{3 \epsilon}{4} } \| w \|_{L^4 L^\infty}  \ \le\  \dd^{\frac{1}{2}+\delta_1},
\end{eqnarray*}
provided
\begin{equation}
\gamma > \frac79 + \frac23\epsilon.
\label{constraint-4}
\end{equation}

It follows immediately that
\begin{eqnarray*}
\| {Term}_4^{\eta_1} - {Term}_4^{\eta_2} \|_{L^\infty} & \leq & \dd^{\frac{\gamma}{2} - \frac{1}{2} + \gamma -\frac38 -\frac{3\gamma}{8} - \frac{3 \epsilon}{4} } \| w \|_{L^4 L^\infty}  \\
& \leq & \dd^{ \frac{9\gamma}{8} - \frac{7}{8} - \frac{3 \epsilon}{4} } \ \| w_1 - w_2 \|_{L^4 L^\infty} .
\end{eqnarray*}

For ${Term}_5^\eta$, we have
\begin{eqnarray*}
\| {Term}_5^\eta \|_{L^\infty}  & = & \left( \frac{N+N_{cr}^{FD}}{|N-N_{cr}^{FD}|} \right)^{\frac12} \| \int_0^t 2 \tilde{A} \eta \langle \chi, \tilde{R} \rangle ds \|_{L^\infty} \\ 
& \leq & \dd^{\frac{\gamma}{2} - \frac{1}{2} + \frac{\gamma}{2}- \frac12 - \frac{\gamma}{2} - \epsilon } \| \eta \|_{L^\infty} \| \tilde{R} \|_{L^\infty_{x,t}} \\
& \leq & \dd^{\frac{\gamma}{2} - \epsilon }  \| \eta \|_{L^\infty} \le \dd^{\frac{1}{2}+\delta_1},\ \ (\gamma>2\epsilon),
\end{eqnarray*}
which follows from previous constraints on $\gamma$.  Once again, we have as well
\begin{eqnarray*}
\| {Term}_5^{\eta_1} - {Term}_5^{\eta_2} \|_{L^\infty} & \leq & \dd^{\frac{\gamma}{2} - \epsilon }  \| \eta_1 - \eta_2 \|_{L^\infty}.
\end{eqnarray*}

For ${Term}_6^\eta$, we have using Proposition \ref{lem:tildeR}
\begin{eqnarray*}
\| {Term}_6^\eta \|_{L^\infty}  & = & \left( \frac{N+N_{cr}^{FD}}{|N-N_{cr}^{FD}|} \right)^{\frac12} \| \int_0^t \eta^2 \langle \tilde{R}, \chi \rangle ds \|_{L^\infty} \\ 
& \leq & \dd^{\frac{\gamma}{2} - \frac{1}{2} + 1 - \frac12 - \frac{\gamma}{2} - \epsilon } \| \eta \|_{L^\infty}^2  \\
& \leq & \dd^{\frac{1}{2}+\delta_1 - \epsilon } \| \eta \|_{L^\infty} \le \dd^{\frac{1}{2}+\delta_1},
\end{eqnarray*}
which follows if $\delta_1 > \epsilon$.  Furthermore, 
\begin{eqnarray*}
\| {Term}_6^{\eta_1} - {Term}_6^{\eta_2} \|_{L^\infty} & \leq &
\dd^{\frac{1}{2}+\delta_1 - \epsilon } \| \eta_1 - \eta_2
\|_{L^\infty}.
\end{eqnarray*}

For ${Term}_7^\eta$, we have
\begin{eqnarray*}
\| {Term}_7^\eta \|_{L^\infty}  & = & \left( \frac{N+N_{cr}^{FD}}{|N-N_{cr}^{FD}|} \right)^{\frac12}  \| \int_0^t \tilde{A} \eta \langle w, \chi \rangle ds \|_{L^\infty} \\ 
& \leq & \dd^{\frac{\gamma}{2} - \frac{1}{2} + \frac{\gamma}{2} - \frac38 - \frac{3\gamma}{8} - \frac{3 \epsilon}{4} } \| \eta \|_{L^\infty} \| w \|_{L^4 L^\infty}  \\
& \leq & \dd^{\frac{5\gamma}{8} - \frac38  - \frac{3 \epsilon}{4} + \delta_1} \| \eta \|_{L^\infty} \le  \dd^{\frac{1}{2}+\delta_1}
\end{eqnarray*}
provided  
\begin{equation}
\gamma > \frac35 + \frac65\epsilon-\frac85\delta_1.
\label{constraint-5}
\end{equation}
 Furthermore, 
\begin{eqnarray*}
\| {Term}_7^{\eta_1} - {Term}_7^{\eta_2} \|_{L^\infty} & \le& \dd^{\frac{5\gamma}{8} - \frac38  - \frac{3 \epsilon}{4} + \delta_1}\ \| (\eta_1-\eta_2,w_1-w_2) \|_X.
\end{eqnarray*}

For ${Term}_8^\eta$, we have
\begin{eqnarray*}
\| {Term}_8^\eta \|_{L^\infty}  & = & \left( \frac{N+N_{cr}^{FD}}{|N-N_{cr}^{FD}|} \right)^{\frac12} \| \int_0^t \eta^2 \langle w, \chi \rangle ds \|_{L^\infty} \\ 
& \leq & \dd^{\frac{\gamma}{2} - \frac{1}{2}  - \frac38 - \frac{3\gamma}{8} - \frac{3 \epsilon}{4} } \| \eta \|_{L^\infty}^2 \| w \|_{L^4 L^\infty}  \\
& \leq & \dd^{\frac{\gamma}{8} + \frac{1}{8} - \frac{3 \epsilon}{4}+ 2 \delta_1 }\| w \|_{L_t^4 L_x^\infty} \le \dd^{\frac{1}{2}+\delta_1}
\end{eqnarray*}
by previous constraints on $\gamma$.  Again, it follows  that
\begin{eqnarray*}
\| {Term}_8^{\eta_1} - {Term}_8^{\eta_2} \|_{L^\infty} & \leq & \dd^{\frac{\gamma}{8} + \frac{1}{8} - \frac{3 \epsilon}{4}+ 2 \delta_1 }\| w_1 - w_2 \|_{L_t^4 L_x^\infty} .
\end{eqnarray*}
\bigskip

We now study the map on the dispersive part, $w$. For ${Term}_1^w$,using the Strichartz estimates \eqref{eqn:strich2}, we have 
\begin{eqnarray*}
\| {Term}_1^w \|_{L^\infty H^1 \cap L^4 L^\infty} & = & \| \int_0^t e^{iH(t-s)}  e^{i \Omega_0 (t-s)} e^{i \int_s^{t} (\tilde{A}^2 + 3 \tilde{\alpha}^2 + \tilde{\beta}^2) (s') ds'} P_c 3 \tilde{A}^2 \eta \psi_0^3 ds \|_{L^\infty H^1 \cap L^4 L^\infty} \\ 
& \leq & \dd^{\gamma} \| \eta \|_{L^\infty} \| \psi_0^3 \|_{L^{\frac43} W^{1,1}}  \\
& \leq & \dd^{\gamma - \frac38 - \frac{3 \gamma}{8} - \frac{3 \epsilon}{4}} \| \eta \|_{L^\infty} \le \dd^{\frac{1}{2}+\delta_1},\ (\gamma>\frac35), 
\end{eqnarray*}
by previous constraints on $\gamma$.  Hence, 
\begin{eqnarray*}
\| {Term}_1^{w_1} - {Term}_1^{w_2} \|_{L^\infty H^1 \cap L^4 L^\infty}  & \leq & \dd^{\gamma - \frac38 - \frac{3 \gamma}{8} - \frac{3 \epsilon}{4}} \| \eta_1 - \eta_2 \|_{L^\infty} .
\end{eqnarray*}

For ${Term}_2^w$, we have using bounds on $\tilde{R}$ in $H^1$ 
 (Lemma \ref{lem:tildeRstr})
\begin{eqnarray*}
\| {Term}_2^w \|_{L_t^\infty H_x^1 \cap L_t^4 L_x^\infty}  & = & \| \int_0^t e^{iH(t-s)}  e^{i \Omega_0 (t-s)} e^{i \int_s^{t} (\tilde{A}^2 + 3 \tilde{\alpha}^2 + \tilde{\beta}^2) (s') ds'} P_c 2 \tilde{A} \tilde{R} \eta ds \|_{L^\infty H^1 \cap L^4 L^\infty} \\ 
& \leq & \dd^{\frac{\gamma}{2}} \| \tilde{R} \|_{L^1 H^1} \| \eta \|_{L^\infty} \\
& \leq & \dd^{2 \gamma - \frac{7}{8} - \frac{7\gamma}{8} - \frac{7\epsilon}{4}} \| \eta \|_{L^\infty} \le  \dd^{\frac{9 \gamma}{8} - \frac{7}{8} - \frac{7\epsilon}{4}} \| \eta \|_{L^\infty} \le  \dd^{\frac{1}{2}+\delta_1},
\end{eqnarray*}
which follows using $\gamma > \frac79+\frac{14}{9}\epsilon$.  
Hence, it follows that
\begin{eqnarray*}
\| {Term}_2^{w_1} - {Term}_2^{w_2} \|_{L^\infty H^1 \cap L^4 L^\infty} & \leq & \dd^{\frac{9 \gamma}{8} - \frac{7}{8} - \frac{7\epsilon}{8}} \| \eta_1 - \eta_2 \|_{L^\infty}.
\end{eqnarray*}

For ${Term}_3^w$, using Lemma \ref{lem:tildeRstr} once again we have 
\begin{eqnarray*}
\| {Term}_3^w \|_{L^\infty H^1 \cap L^4 L^\infty} & = & \| \int_0^t e^{iH(t-s)}  e^{i \Omega_0 (t-s)} e^{i \int_s^{t} (\tilde{A}^2 + 3 \tilde{\alpha}^2 + \tilde{\beta}^2) (s') ds'} P_c \tilde{A} \psi_0^2 \tilde{R} \eta ds \|_{L^\infty H^1 \cap L^4 L^\infty} \\ 
& \leq & \dd^{\frac{\gamma}{2} } \| \tilde{R} \psi_0^2 \|_{L^1 H^1} \| \eta \|_{L^\infty}  \\
& \leq & \dd^{2 \gamma - \frac{7}{8} - \frac{7\gamma}{8} - \frac{7\epsilon}{4}} \| \eta \|_{L^\infty} \le  \dd^{\frac{9 \gamma}{8} - \frac{7}{8} - \frac{7\epsilon}{4}} \| \eta \|_{L^\infty} \le  \dd^{\frac{1}{2}+\delta_1},
\end{eqnarray*}
which follows easily from previous constraints on $\gamma$.  
Hence, it follows directly that
\begin{eqnarray*}
\| {Term}_3^{w_1} - {Term}_3^{w_2} \|_{L^\infty H^1 \cap L^4 L^\infty}  & \leq & \dd^{\frac{9 \gamma}{8} - \frac{7}{8} - \frac{7\epsilon}{4}}  \| \eta_1 - \eta_2 \|_{L^\infty} .
\end{eqnarray*}

For ${Term}_4^w$, we have
\begin{eqnarray*}
\| {Term}_4^w \|_{L^\infty H^1 \cap L^4 L^\infty}  & = & \| \int_0^t e^{iH(t-s)}  e^{i \Omega_0 (t-s)} e^{i \int_s^{t} (\tilde{A}^2 + 3 \tilde{\alpha}^2 + \tilde{\beta}^2) (s') ds'} P_c  \tilde{A}^2 \psi_0^2 w ds \|_{L^\infty H^1 \cap L^4 L^\infty} \\
&\le & \dd^{\gamma}\ \left\| \psi_0^2 w\right\|_{L^\frac{4}{3}_tW^{1,1}} \leq  \dd^{ - \frac38 + \frac{5 \gamma}{8} - \frac{3 \epsilon}{4} }   \| w \|_{L^\infty H^1} \le 
 \dd^{\frac{1}{2}+\delta_1},\ \ \ (\gamma > \frac35),
\end{eqnarray*}
which follows from previous constraints on $\gamma$.
Furthermore, 
\begin{eqnarray*}
\| {Term}_4^{w_1} - {Term}_4^{w_2} \|_{L^\infty H^1 \cap L^4 L^\infty}  & \leq & \dd^{\frac58 + \frac{\gamma}{8} - \frac{3 \epsilon}{4} } \| w_1 - w_2 \|_{L^\infty H^1} .
\end{eqnarray*}

For ${Term}_5^w$, we have
\begin{eqnarray*}
\| {Term}_5^w \|_{L^\infty H^1 \cap L^4 L^\infty} & = & \| \int_0^t e^{iH(t-s)}  e^{i \Omega_0 (t-s)} e^{i \int_s^{t} (\tilde{A}^2 + 3 \tilde{\alpha}^2 + \tilde{\beta}^2) (s') ds'} P_c w^3 ds \|_{L^\infty H^1 \cap L^4 L^\infty} \\ 
& \leq & \dd^{-\frac12 - \frac{\gamma}{2} - \epsilon} \| w \|_{L^\infty_t H^1_x}^3 \\
& \leq & \dd^{\frac{1-\gamma}2 - \epsilon +2 \delta_1} \| w \|_{L^\infty H^1} \le \ \dd^{\frac{1}{2}+\delta_1},
\end{eqnarray*}
which follows easily by previous constraints on $\gamma$.  
Moreover,
\begin{eqnarray*}
\| {Term}_5^{w_1} - {Term}_5^{w_2} \|_{L^\infty H^1 \cap L^4 L^\infty}  & \leq &\dd^{\frac{1-\gamma}2 - \epsilon +2 \delta_1} \| w_1 - w_2 \|_{L^\infty H^1} .
\end{eqnarray*}

For ${Term}_6^w$, we have
\begin{eqnarray*}
\| {Term}_6^w \|_{L^\infty H^1 \cap L^4 L^\infty} & = & \| \int_0^t
e^{iH(t-s)}  e^{i \Omega_0 (t-s)} e^{i \int_s^{t} (\tilde{A}^2 + 3
  \tilde{\alpha}^2 + \tilde{\beta}^2) (s') ds'} 2 P_c \tilde{A} \eta w ds \|_{L^\infty H^1 \cap L^4 L^\infty} \\ 
& \leq & \dd^{\frac{\gamma}{2} - \frac{1}{2} - \frac{\gamma}{2} - \epsilon } \| \eta \|_{L^\infty} \| w \|_{L^\infty H^1} \le \dd^{\delta_1 - \epsilon} \| w \|_{L^\infty H^1},
\end{eqnarray*}
which follows easily from previous constraints on $\delta_1$.  
Hence, it follows directly that
\begin{eqnarray*}
\| {Term}_6^{w_1} - {Term}_6^{w_2} \|_{L^\infty H^1 \cap L^4 L^\infty}  & \leq &  \dd^{\delta_1 - \epsilon}   \| w_1 - w_2 \|_{L^\infty H^1} .
\end{eqnarray*}

For ${Term}_7^w$, we have from Lemma \ref{lem:tildeRstr}
\begin{eqnarray*}
\| {Term}_7^w \|_{L^\infty H^1 \cap L^4 L^\infty} & = & \| \int_0^t e^{iH(t-s)}  e^{i \Omega_0 (t-s)} e^{i \int_s^{t} (\tilde{A}^2 + 3 \tilde{\alpha}^2 + \tilde{\beta}^2) (s') ds'} P_c \tilde{A} \langle \chi, w \rangle \tilde{R} ds \|_{L^\infty H^1 \cap L^4 L^\infty} \\ 
& \leq & \dd^{\frac{\gamma}{2} -\frac{1}{2} - \frac{\gamma}{2} - \epsilon}  \| \tilde{R} \|_{L^\infty_t H^1_x} \| w \|_{L^\infty H^1} \\
& \leq & \dd^{\frac{9 \gamma}{8} - \frac78 - \frac74 \epsilon} \| w \|_{L^\infty H^1} \le  \dd^{\frac{1}{2}+\delta_1}, \ \ \ (\gamma > \frac79),
\end{eqnarray*}
which follows from previous constraints on $\gamma$.
It follows directly that
\begin{eqnarray*}
\| {Term}_7^{w_1} - {Term}_7^{w_2} \|_{L^\infty H^1 \cap L^4 L^\infty}  & \leq & \dd^{\frac{9 \gamma}{8} - \frac78 - \frac74 \epsilon}  \| w_1 - w_2 \|_{L^\infty H^1} .
\end{eqnarray*}
The bounds for ${Term}_8^w$ follows in a similar manner.

For ${Term}_9^w$ and ${Term}_{10}^w$, we have
\begin{eqnarray*}
\| {Term}_9^w \|_{L^\infty H^1 \cap L^4 L^\infty}  & = & \| \int_0^t e^{iH(t-s)}  e^{i \Omega_0 (t-s)} e^{i \int_s^{t} (\tilde{A}^2 + 3 \tilde{\alpha}^2 + \tilde{\beta}^2) (s') ds'} P_c \tilde{A}^2 \psi_0^2 \tilde{R} ds \|_{L^\infty H^1 \cap L^4 L^\infty} \\ 
& \leq & \dd^\gamma \| \psi_0^2 \tilde{R} \|_{L^1 H^1} \leq  \dd^{\frac{13}{8} \gamma - \frac78 -\frac{11}{8} \epsilon} \le \  \dd^{\frac{1}{2}+\delta_1}
\end{eqnarray*}
and
\begin{eqnarray*}
\| {Term}_{10}^w \|_{L^\infty H^1 \cap L^4 L^\infty}  & = & \| \int_0^t e^{iH(t-s)}  e^{i \Omega_0 (t-s)} e^{i \int_s^{t} (\tilde{A}^2 + 3 \tilde{\alpha}^2 + \tilde{\beta}^2) (s') ds'} P_c \tilde{A} \langle \chi, \tilde{R} \rangle \tilde{R}ds \|_{L^\infty H^1 \cap L^4 L^\infty} \\ 
& \leq &  \dd^{\gamma} \| \tilde{R} \|_{L^1_t H^1_x} \| \tilde{R} \|_{L^\infty_{x,t}}  \leq \dd^{2 \gamma + 1 - \frac38 \gamma - \frac38 - \frac34 \epsilon  } \le \dd^{\frac{1}{2}+\delta_1}.
\end{eqnarray*}
Both of these terms are independent of $(\eta,w)$, meaning the contraction mapping follows easily.  

Hence, choosing $\gamma$, $\delta_1$, and $\epsilon$ such that the constraints \eqref{constraint-1},\eqref{constraint-2},\eqref{constraint-3},\eqref{constraint-4},\eqref{constraint-5} are satisfied, the contraction argument follows and the result holds.

Once we have solved for $(\vec{\sigma},R)$, it is clear that the resulting function $\theta \in C (I)$ by construction.  
\end{proof}

\section{Numerical Simulations}
\label{sec:num}

\subsection{Polar Coordinates}
\label{sec:radco}

As it will simplify the process of building initial conditions for numerically solving \eqref{eqn:nlsdwp} that display the behaviors we study above, let us discuss here an alternative set of coordinates for \eqref{eqn:sys-fd}.  Namely, we set $\rho_0 = r_0 e^{i \theta_0}$ and $\rho_1 = r_1 e^{i \theta_1}$.  This leads to the following system of ODE's:
\begin{eqnarray}
\label{eqn:sys-fd-pol}
\left\{ \begin{array}{c}
\dot{r_0} = r_1^2 r_0 \sin( 2 \Delta \theta), \\
\dot{r_1} = - r_0^2 r_1 \sin (2 \Delta \theta), \\
\dot{ (\Delta \theta) } = \Omega_1 - \Omega_0 + (r_1^2 - r_0^2)(1+ \cos( 2 \Delta \theta)),
\end{array} \right.
\end{eqnarray}
where $\Delta \theta = \theta_1 - \theta_0$.  Given the system above, we can say that the bifurcation of stability occurs at 
\begin{eqnarray}
N_{cr}^{FD} = \frac{\Omega_1 - \Omega_0}{2}.
\end{eqnarray}
As we are interested in the behavior quite near the bifurcation point,
we define new parameters $\epsilon_0$, $\epsilon_1$ and $n$ such that
\begin{eqnarray}
r_0 & = & \sqrt{ N_{cr}^{FD} } + \epsilon_0, \\
r_1 & = & \epsilon_1, \\
N & = & N_{cr}^{FD} + n,
\end{eqnarray}
where
\begin{eqnarray}
n = \epsilon_0^2 + \epsilon_1^2 + 2 \sqrt{ N_{cr}^{FD} } \epsilon_0.
\end{eqnarray}

Then, we have
\begin{eqnarray}
\dot{\epsilon_0} & = & \epsilon_1^2 \left( \sqrt{ N_{cr}^{FD} } + \epsilon_0 \right) \sin( 2 \Delta \theta), \\ 
\dot{\epsilon_1} & = & - \epsilon_1 \left( \sqrt{ N_{cr}^{FD} } +
  \epsilon_0 \right)^2 \sin( 2 \Delta \theta) , \\
\dot{ (\Delta \theta) } & = & \Omega_1 - \Omega_0 + \left( \epsilon_1^2 -  \left( \sqrt{ N_{cr}^{FD} } + \epsilon_0 \right)^2 \right)(1+ \cos( 2 \Delta \theta)).
\end{eqnarray}

It should be noted, using the conservation laws we can write the system for $\epsilon_1$ and $\Delta \theta$ independently
\begin{eqnarray}
\label{eqn:radodesys}
\left\{ \begin{array}{c}
\dot{\epsilon_1} = - \epsilon_1 \left( N_{cr}^{FD} +n-\epsilon_1^2
\right) \sin( 2 \Delta \theta) , \\
\dot{ (\Delta \theta) } = \Omega_1 - \Omega_0 - \left(  \left( N_{cr}^{FD} +n- 2\epsilon_1^2  \right) \right)(1+ \cos( 2 \Delta \theta))
\end{array} \right.
\end{eqnarray}
and hence analyze phase plane diagrams, see Figures \ref{fig1} and
\ref{fig2}.  In particular, the behavior of $\Delta \theta$ in these
regions leads to interesting oscillatory behavior as shown in the
phase diagrams featured in Figures \ref{fig1} and \ref{fig2}.  For $n
> 0$, a simple calculation shows that the equilibrium solutions occur
for $\epsilon_1 = \sqrt{ \frac{n}{2} } $, $\Delta \theta = k \pi$ for
$k \in \mathbb{Z}$.  

\begin{figure}
\includegraphics[scale=0.25]{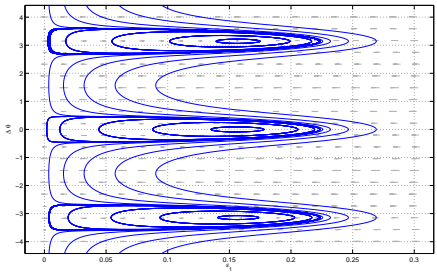}
\caption{A numerical plot of a phase plane with $\Delta \theta$
  plotted versus $\epsilon_1$ for $n =>0$.  Here, we have $N_{cr}^{FD}
  = .2$, $n = N - N_{cr}^{FD} = .05$.}
\label{fig1}
\end{figure} 

\begin{figure}
\includegraphics[scale=0.25]{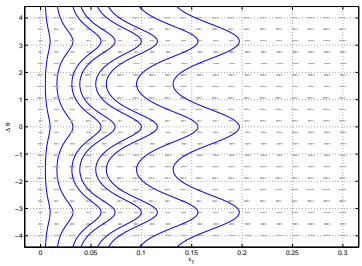}
\caption{A numerical plot of a phase plane with $\Delta \theta$
  plotted versus $\epsilon_1$ for $n < 0$.  Here, we have $N_{cr}^{FD}
  = .2$, $n = N - N_{cr}^{FD}= -.05$.}
\label{fig2}
\end{figure} 

For $n>0$, we have trapped orbits near the values $\Delta \theta = k \pi$ for $k \in \mathbb{Z}$.  This is a manifestation of the orbital stability of these mixed states.  However, oscillations between wells can be generated by a large enough phase shift to leave the region where such trapped orbits occur.  These oscillations are then large, in particular $\epsilon_1$ must reach some $\epsilon_{max}$ before decreasing.  

For $n<0$, we see the oscillations still attain a maximum at $\Delta \theta = k \pi$, however their amplitude approaches $0$ with the initial $\epsilon_1(0)$.  This is a manifestation of the stability of the symmetric state in this regime.  One may ask if such finite dimensional Hamiltonian dynamics would appear in the infinite dimension dynamics of the PDE.  For this, see Figures \ref{fig3}, \ref{fig4} and \ref{fig5} for numerical evidence of their existence for long times.

To begin, we locate the symmetry breaking point for a particular system.  To do this, we use spectral renormalization with an asymmetric initial function to find the symmetry breaking point, see Fig. \ref{fig8}. 

In the remainder of this section, we run simulations with wells of the form
\begin{eqnarray}
\label{eqn:ddw}
V_L = q (\delta(x-\frac{L}{2}) + \delta(x+\frac{L}{2})).
\end{eqnarray}
Similar results hold and the same numerical analysis tools are applicable for potentials with more regularity.  Knowing the bifurcation point as discussed in \cite{JW}, we can numerically integrate using a finite element method similar to that in \cite{HMZ1}, where scattering of soliton solutions across single delta function potentials was analyzed (see \cite{ADKM} for analysis of finite element methods for nonlinear Schr\"odinger equations without potential).  It is quite simple to adapt the method presented there to allow for a double-well potential (for delta functions or smoother potentials).  The initial data is generated by finding the lowest entries of the spectrum of the discretized representation of $H = -\Delta + V$ in the Galerkin approximation, which is an operation embedded in many numerical software programs.  For simplicity, we use the $eig$ function from {\it Matlab}.  Then, we may numerically solve the PDE system \eqref{eqn:nlsdwp} with initial data corresponding to that necessary for the three types of oscillation described in Section \ref{sec:finite-dim}.  Note, one could also use the solitons from the spectral renormalization code (see \cite{AM}) as initial data quite easily, however these represent true nonlinear structures and we wish to observe structures derived from the finite dimensional dynamics, which we only expect to persist on finite time scales due to the nonlinear structure.  The orbital stability of the nonlinear objects is an interesting question in its own right and was explored in \cite{KKSW}.  The equilibrium point of our dynamical system is in fact the finite dimensional part of a soliton solution, so there is no question the orbital stability of the soliton and the long time existence of oscillations near an equilibrium point are related. The phase plane diagrams for the finite dimensional dynamics are plotted using the MATLAB software program {\it pplane7} \cite{AP}.

\begin{figure}
\begin{center}
\includegraphics[scale=0.35]{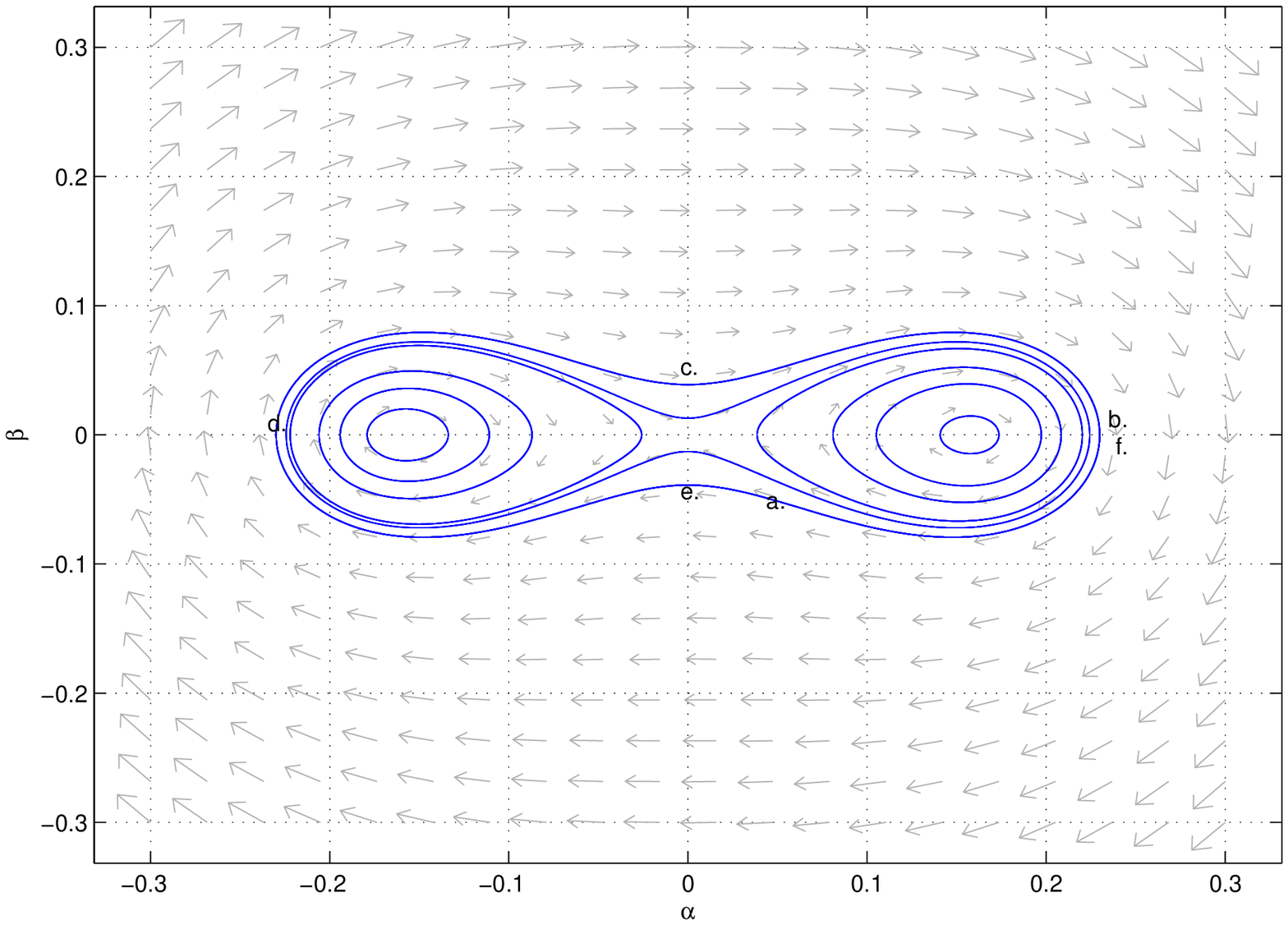}
\end{center}   
\includegraphics[scale=0.2]{dwpngeq01a} 
\includegraphics[scale=0.2]{dwpngeq01b} \\
\includegraphics[scale=0.2]{dwpngeq01c}  
\includegraphics[scale=0.2]{dwpngeq01d} \\
\includegraphics[scale=0.2]{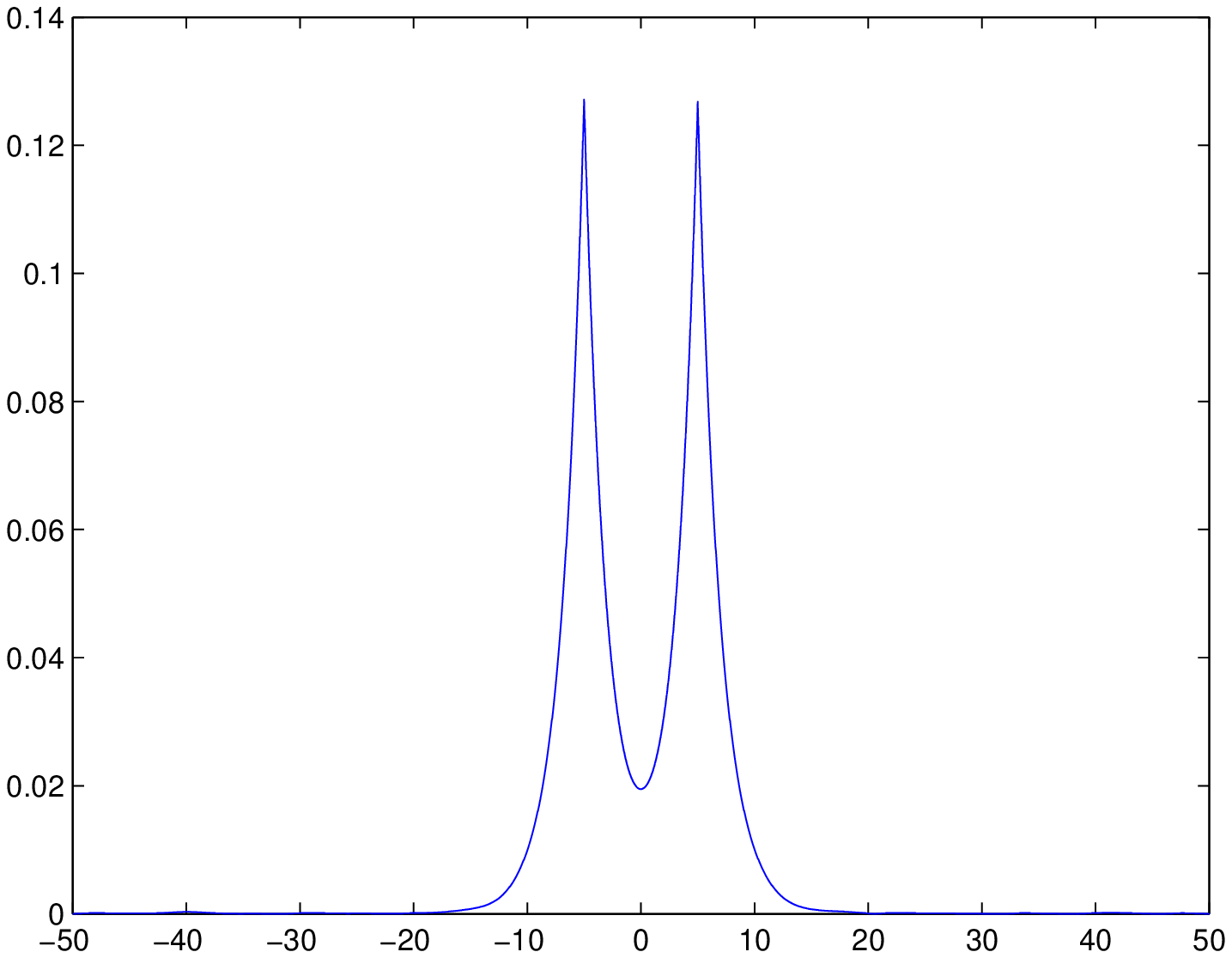}   
\includegraphics[scale=0.2]{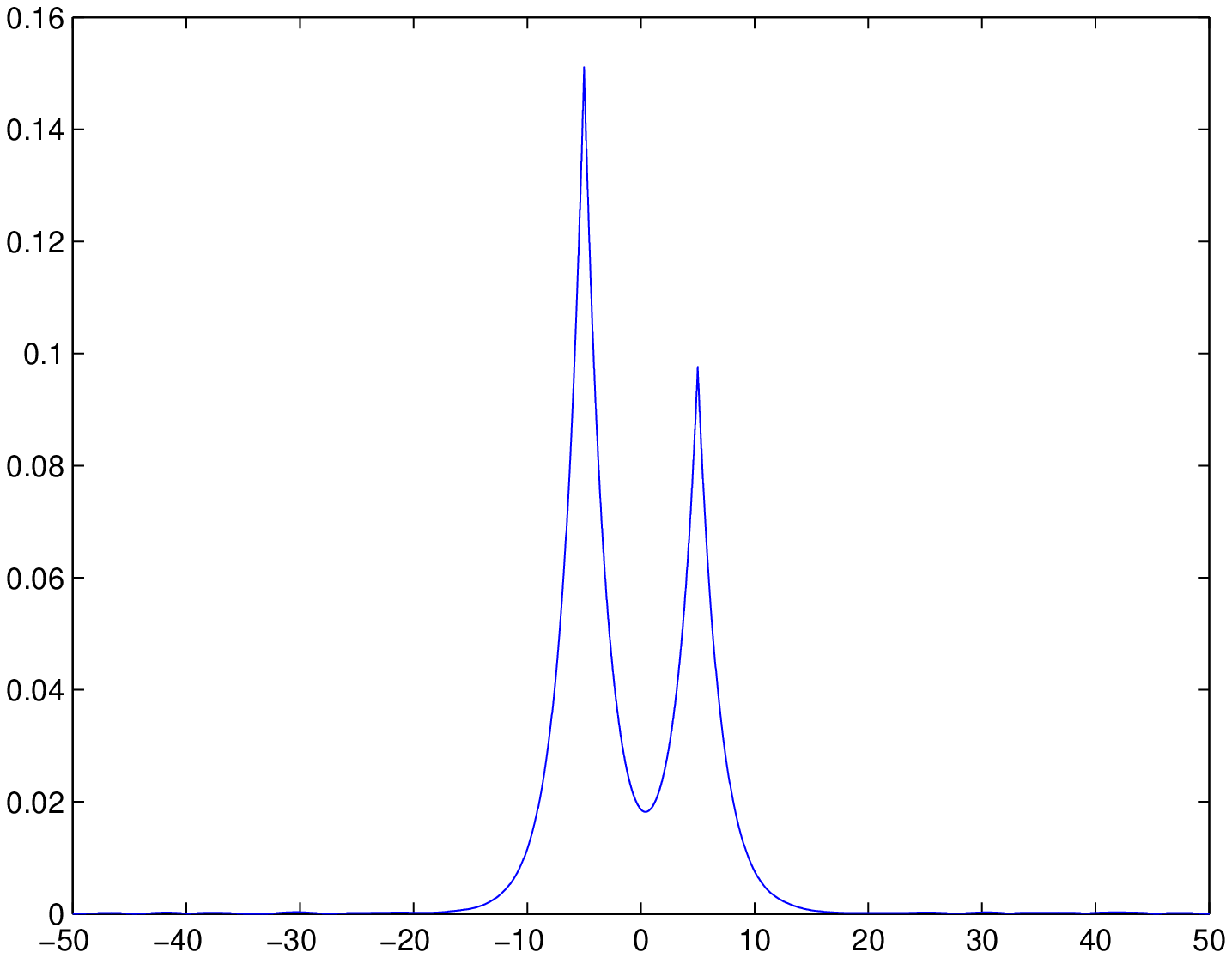}
\caption{At top, a numerical plot of the phase plane diagram for $n =
  N-N_{cr}^{FD}>0$ with specific points chosen along a closed orbit
  which shows full oscillation of mass from one well to another.  Below numerical plots of the absolute value of the solution to Equation \eqref{eqn:nlsdwp} at various times with initial data such that $n = N-N_{cr}^{FD} > 0$ and $\Delta \theta (0) = 1$.  The plots correspond to points a, b, c, d, e, and f respectively from the specified orbit. }
\label{fig3}
\end{figure}

\begin{figure}
\begin{center}
\includegraphics[scale=0.35]{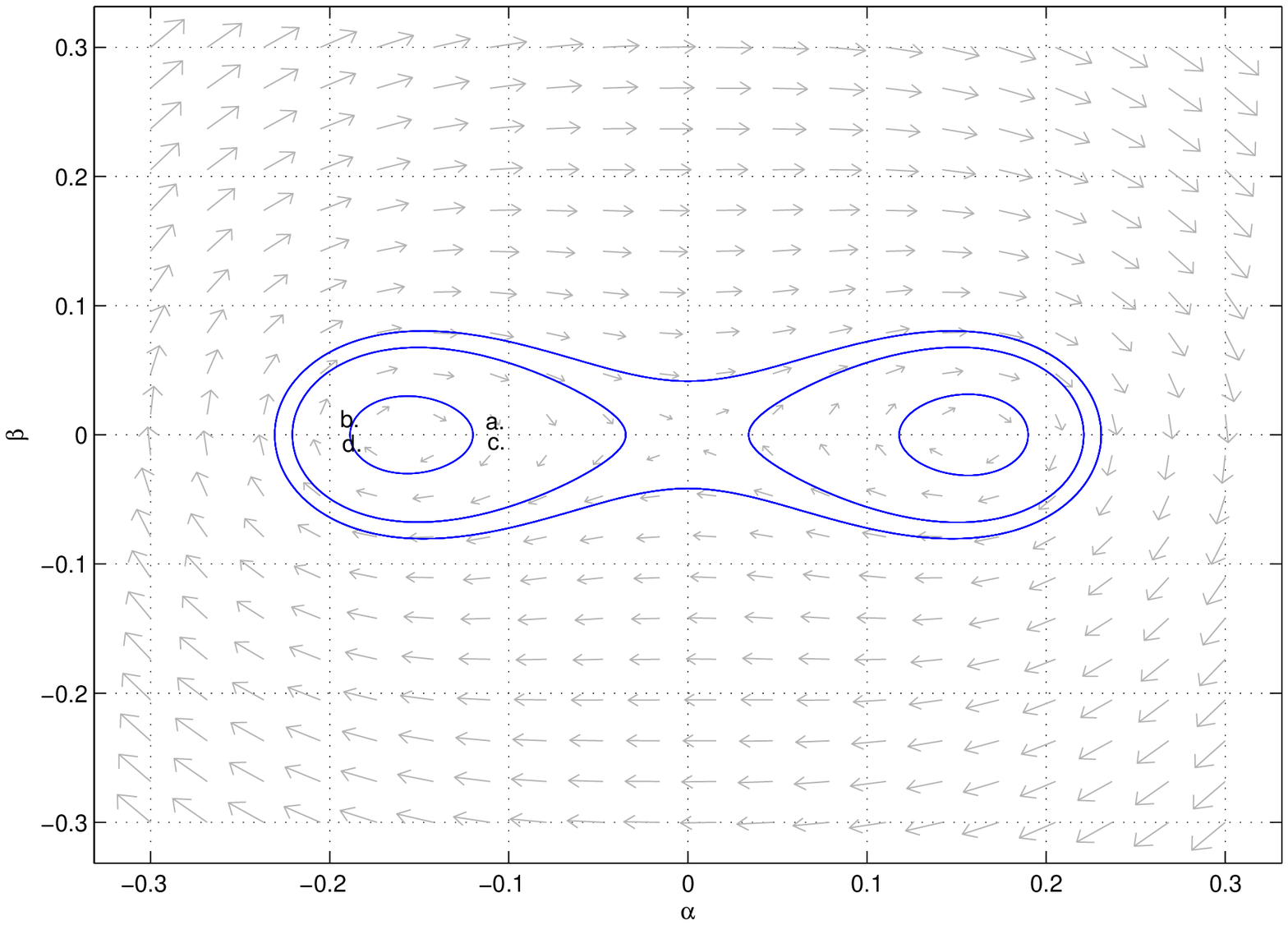}
\end{center}
\includegraphics[scale=0.2]{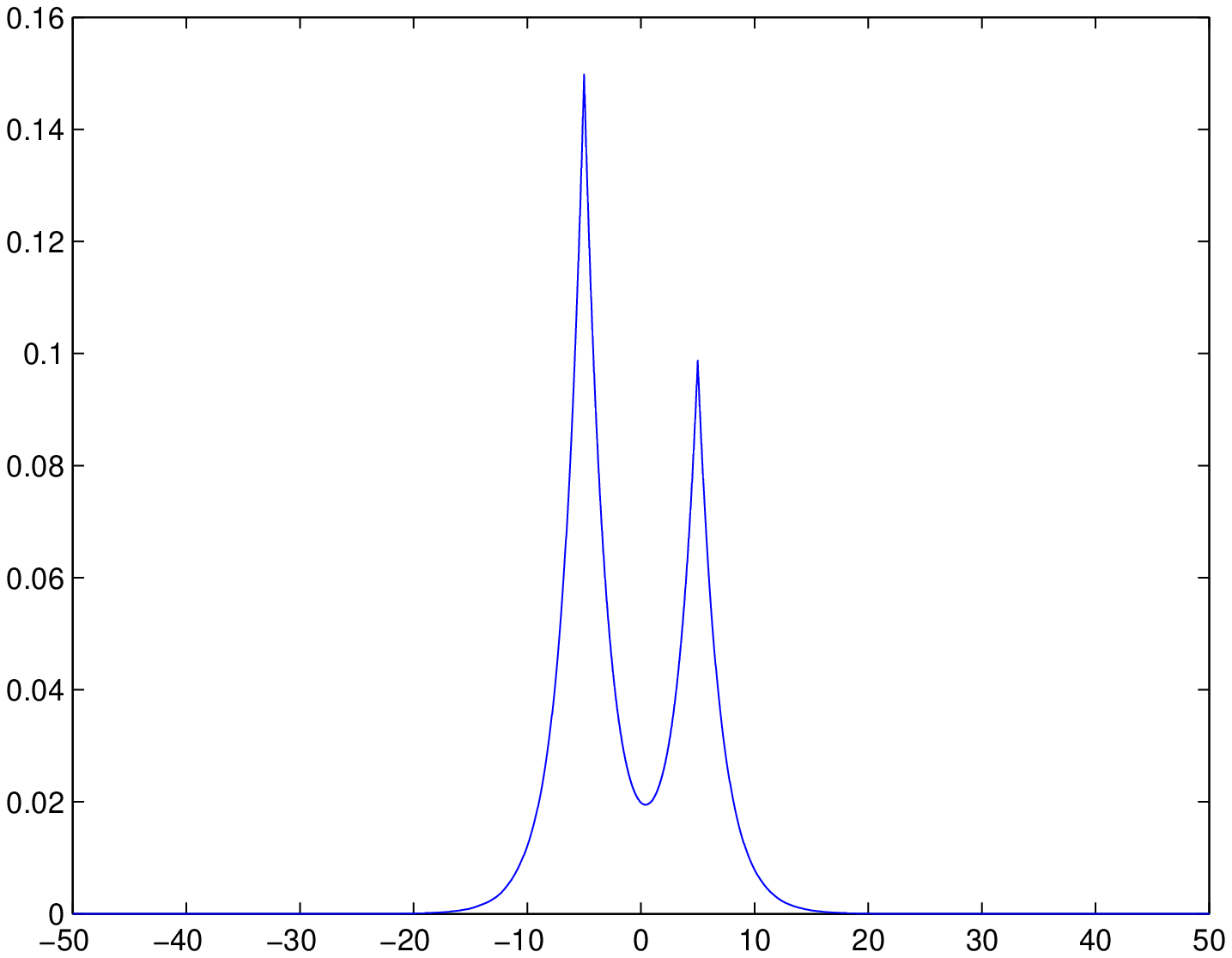} 
\includegraphics[scale=0.2]{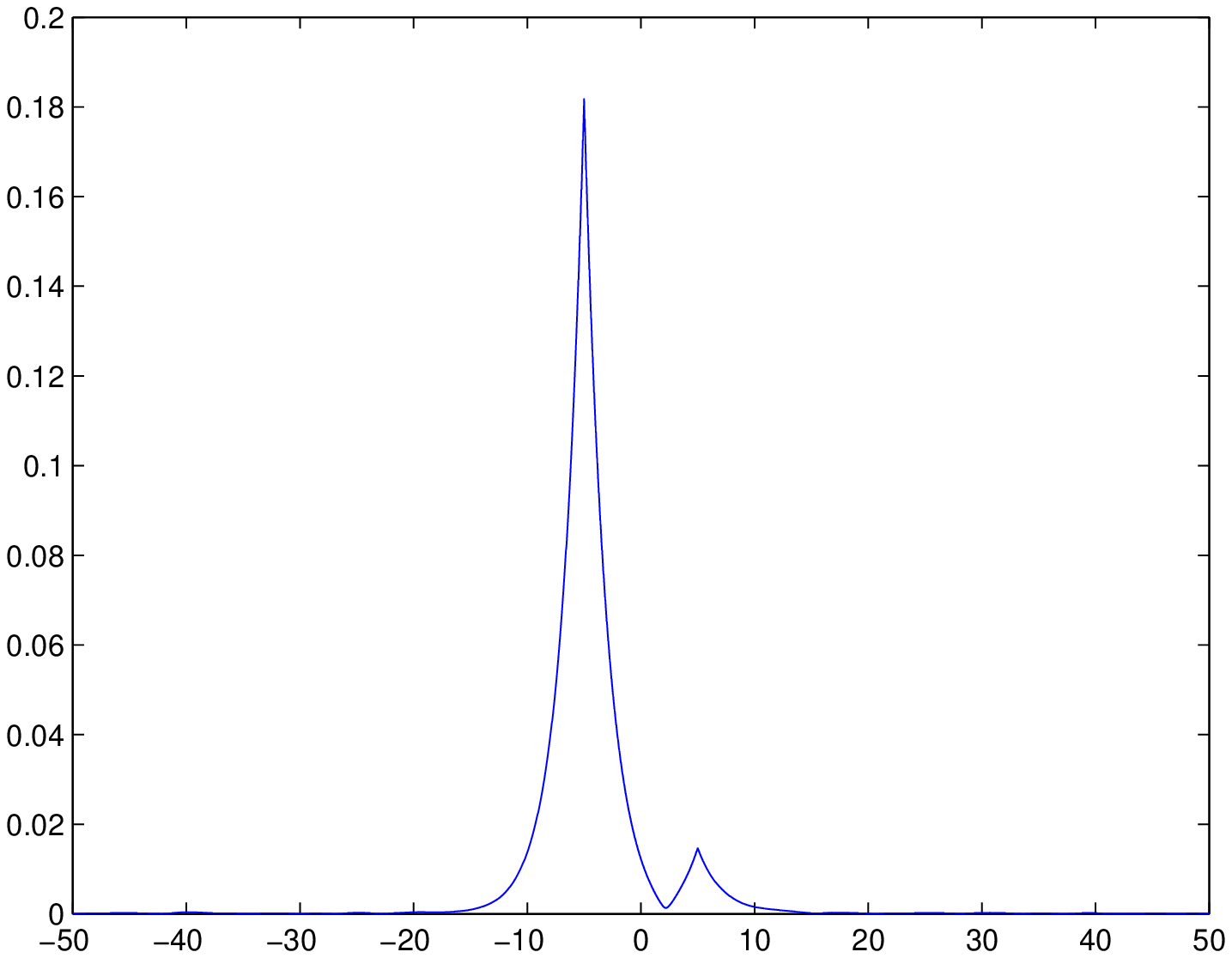} \\
\includegraphics[scale=0.2]{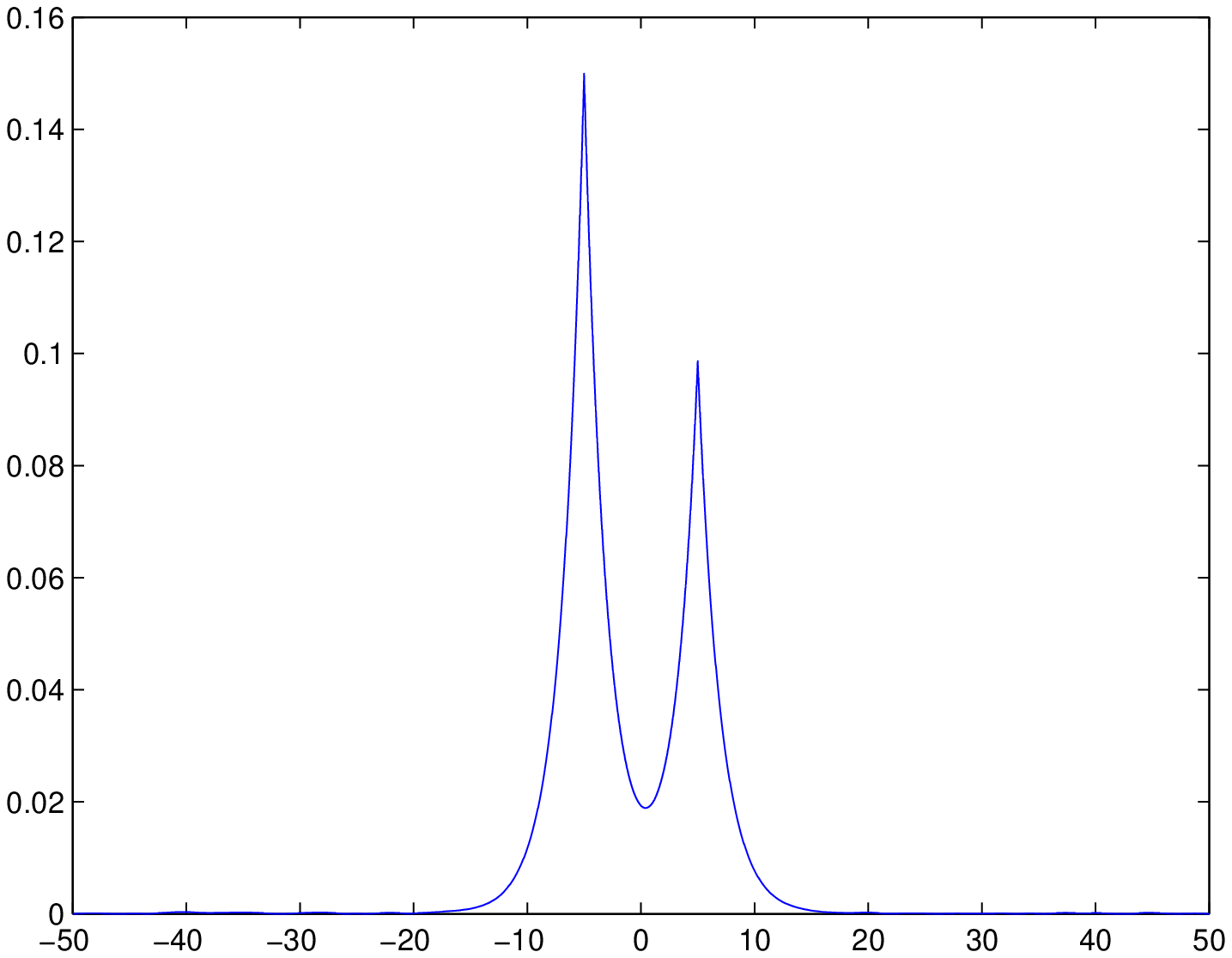} 
\includegraphics[scale=0.2]{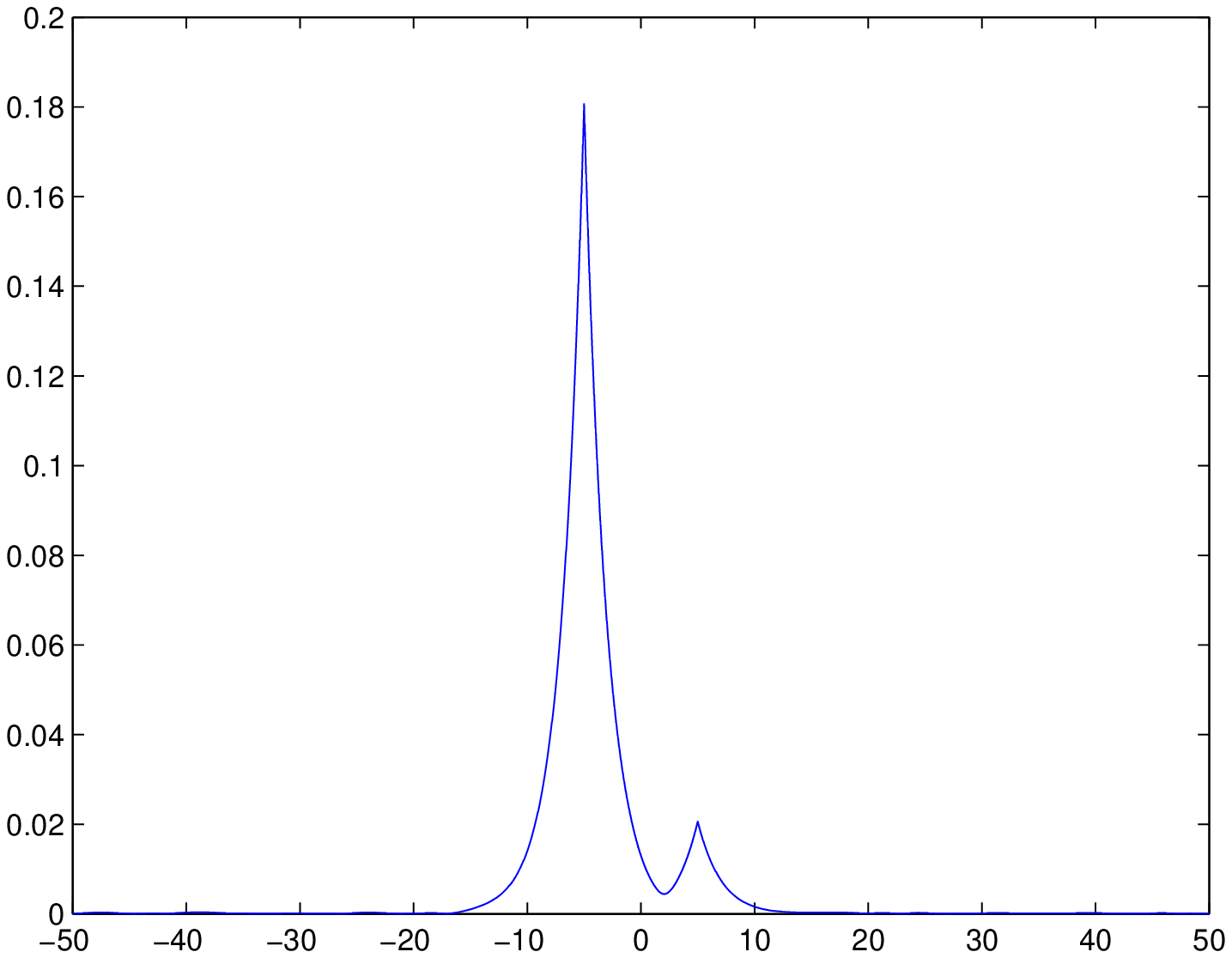}

\caption{At top, a numerical plot of the phase plane diagram for $n = N-N_{cr}^{FD}>0$ with specific points chosen along a closed orbit which shows localization of the mass on one side of the well. Below, numerical plots of the absolute value of the solution to Equation \eqref{eqn:nlsdwp} at various times with initial data such that $n = N-N_{cr}^{FD} > 0$ and $\Delta \theta (0) = 0$.  The plots correspond to points a, b, c, d respectively from the specified orbit. }
\label{fig4}
\end{figure}

\begin{figure}
\begin{center}
\includegraphics[scale=0.35]{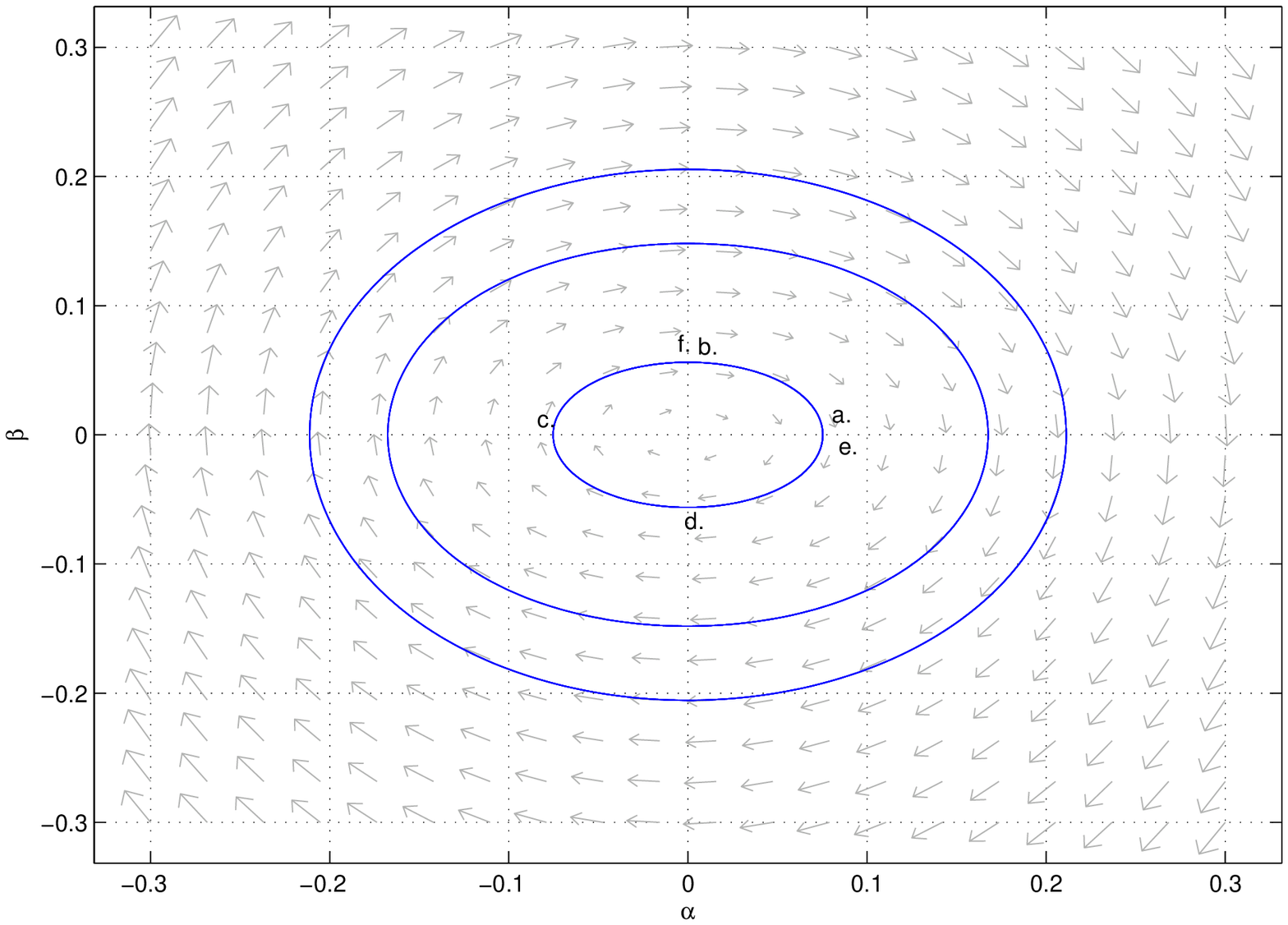}
\end{center}
\includegraphics[scale=0.2]{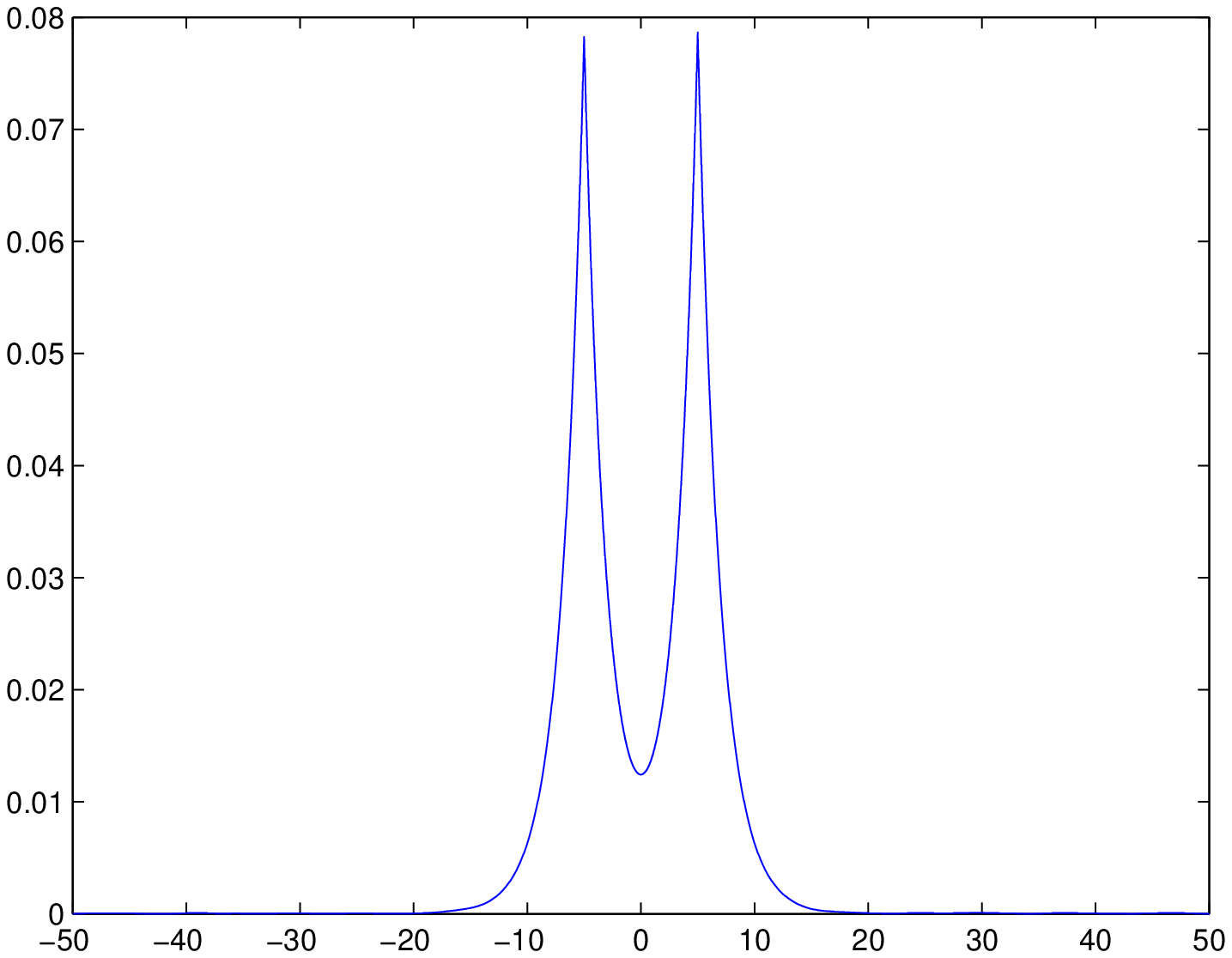} 
\includegraphics[scale=0.2]{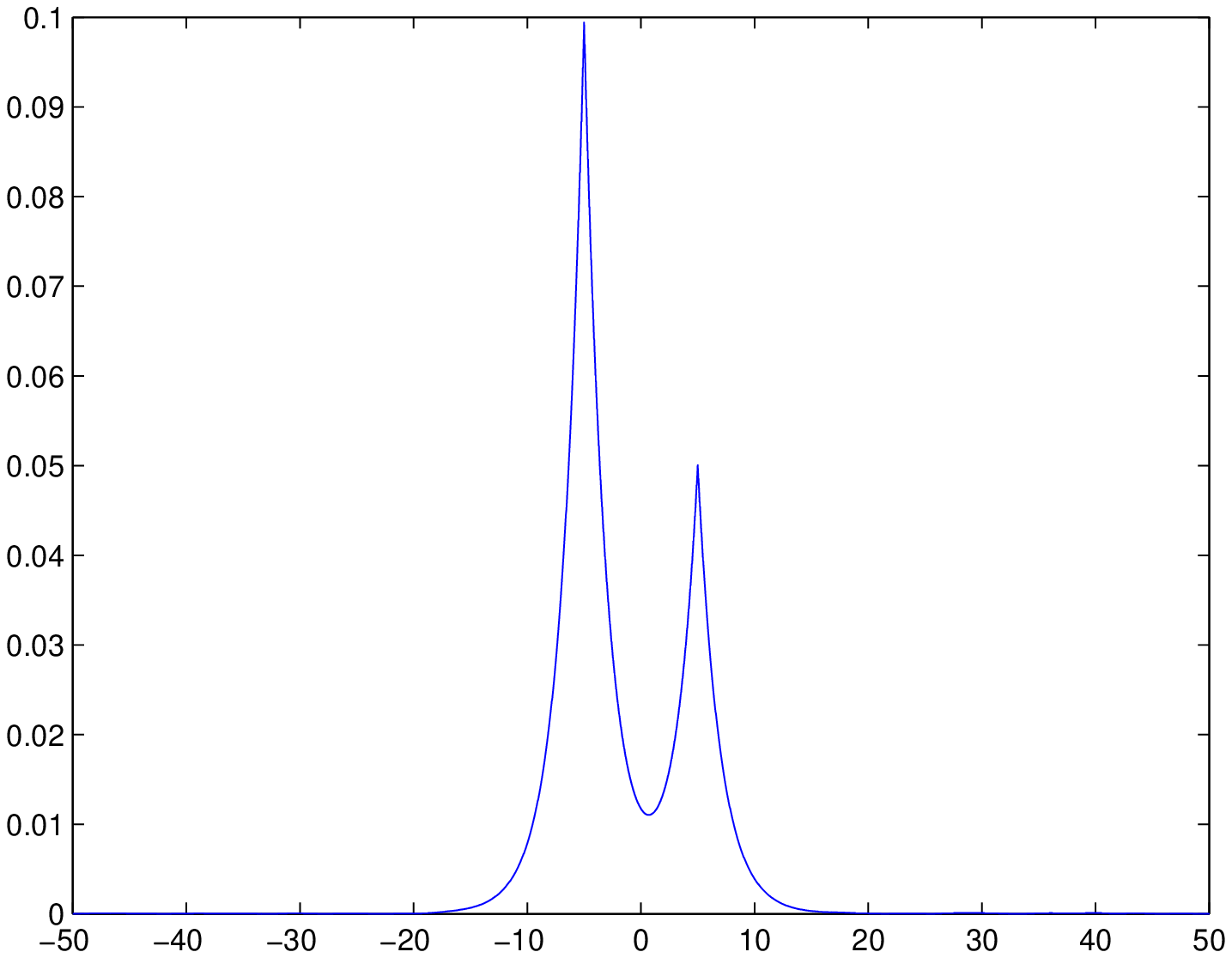} \\
\includegraphics[scale=0.2]{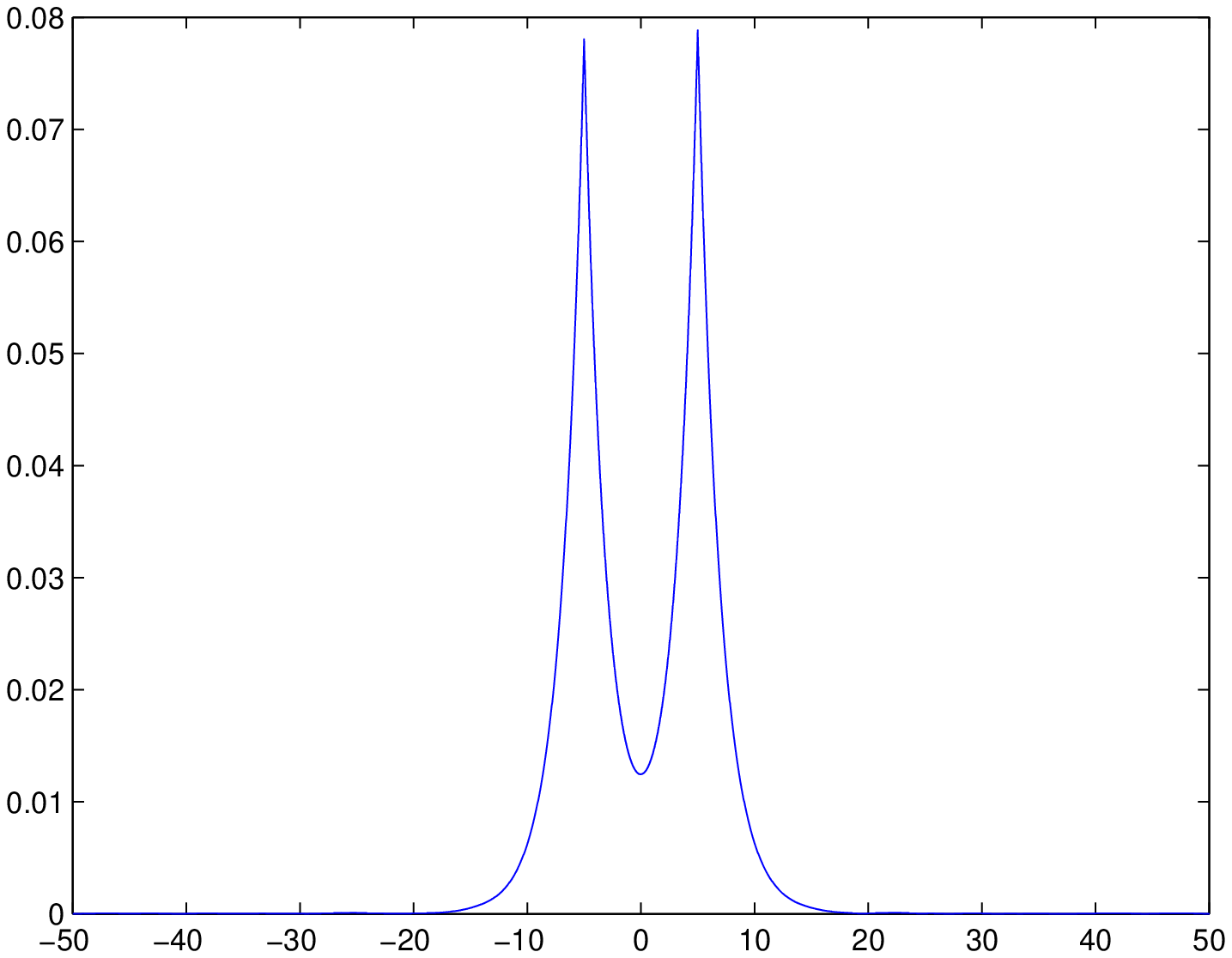} 
\includegraphics[scale=0.2]{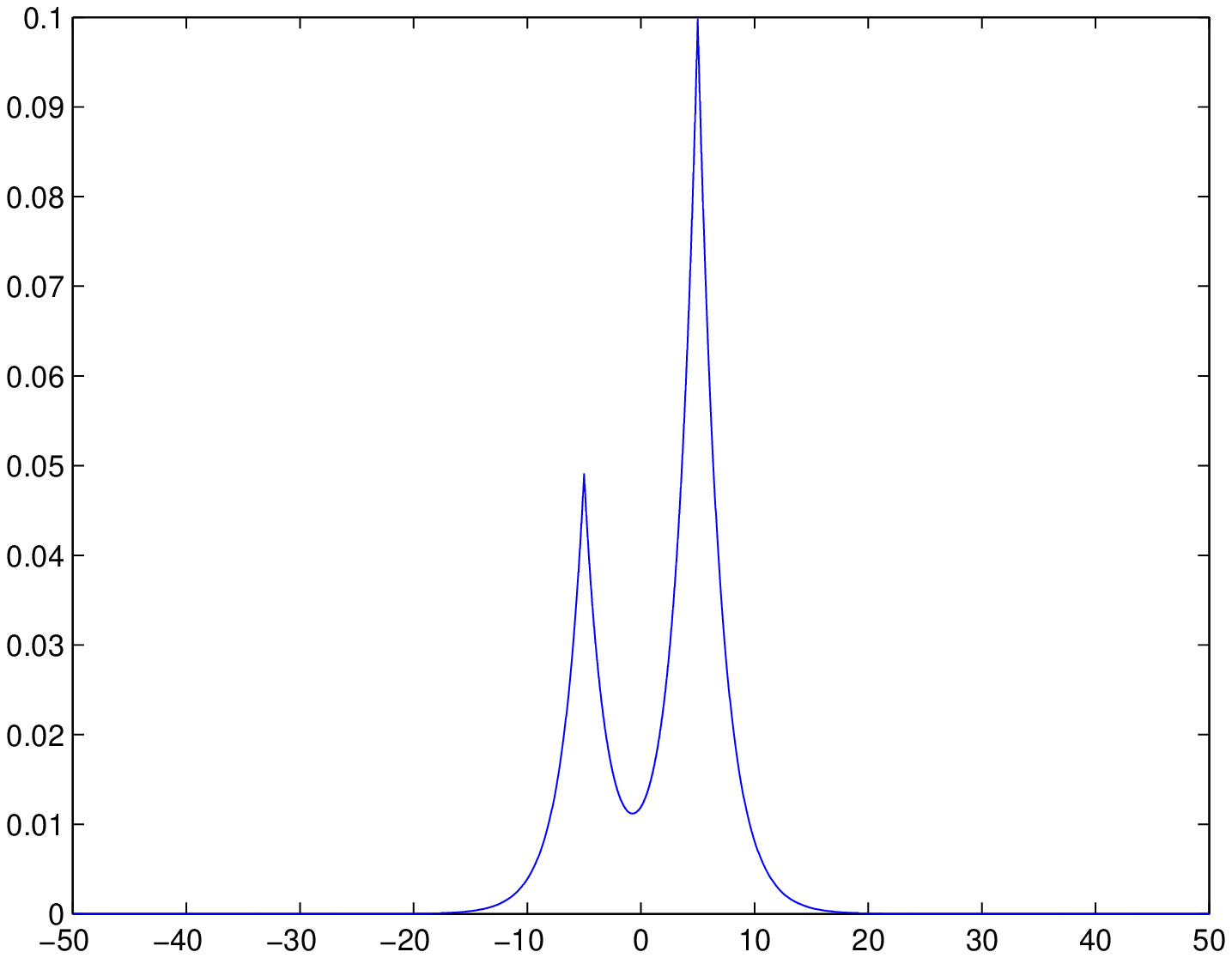} \\
\includegraphics[scale=0.2]{dwpnleq0b} 
\includegraphics[scale=0.2]{dwpnleq0c}

\caption{At top, a numerical plot of the phase plane diagram for $n =
  N-N_{cr}^{FD}<0$ with specific points chosen along a closed
  orbit. Below, numerical plots of  the absolute value of the solution to Equation \eqref{eqn:nlsdwp} at various times with initial data such that $n = N-N_{cr}^{FD} < 0$ and $\Delta \theta (0) = 0$.  The plots correspond to points a, b, c, d, e, and f respectively from the specified orbit.}
\label{fig5}
\end{figure}

Finally, by taking a system such that $|\Omega_0 - \Omega_1|$ is comparably large (or well-separation distance, $L$, comparably close but sufficiently large to guarantee the hypothesized discrete spectrum), we can observe for large perturbations coupling to the continuous spectrum and hence decay of a fully oscillatory solution to a ground state in a finite time.  In general, the mass dispersion is rather rapid, hence after a prescribed number of time steps determined by the computational domain, we cut off the solution near the origin and continue solving with the cut-off initial data.  See Figures \ref{fig9}, \ref{fig10}, \ref{fig11} and \ref{fig12} for various computed time evolutions of the (radiation) damped oscillations observed in such a system.

\begin{figure}
\includegraphics[scale=0.25]{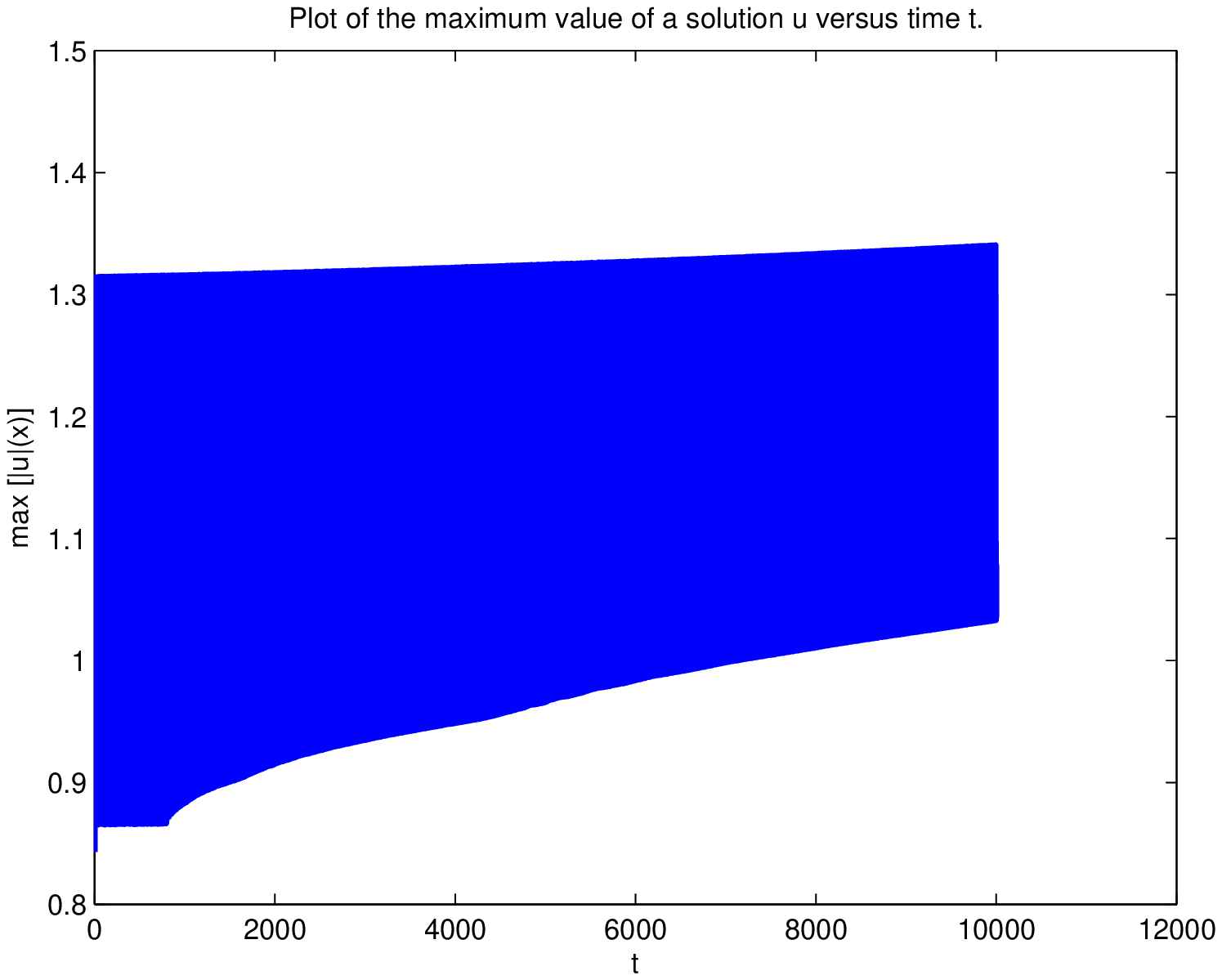} 
\includegraphics[scale=0.25]{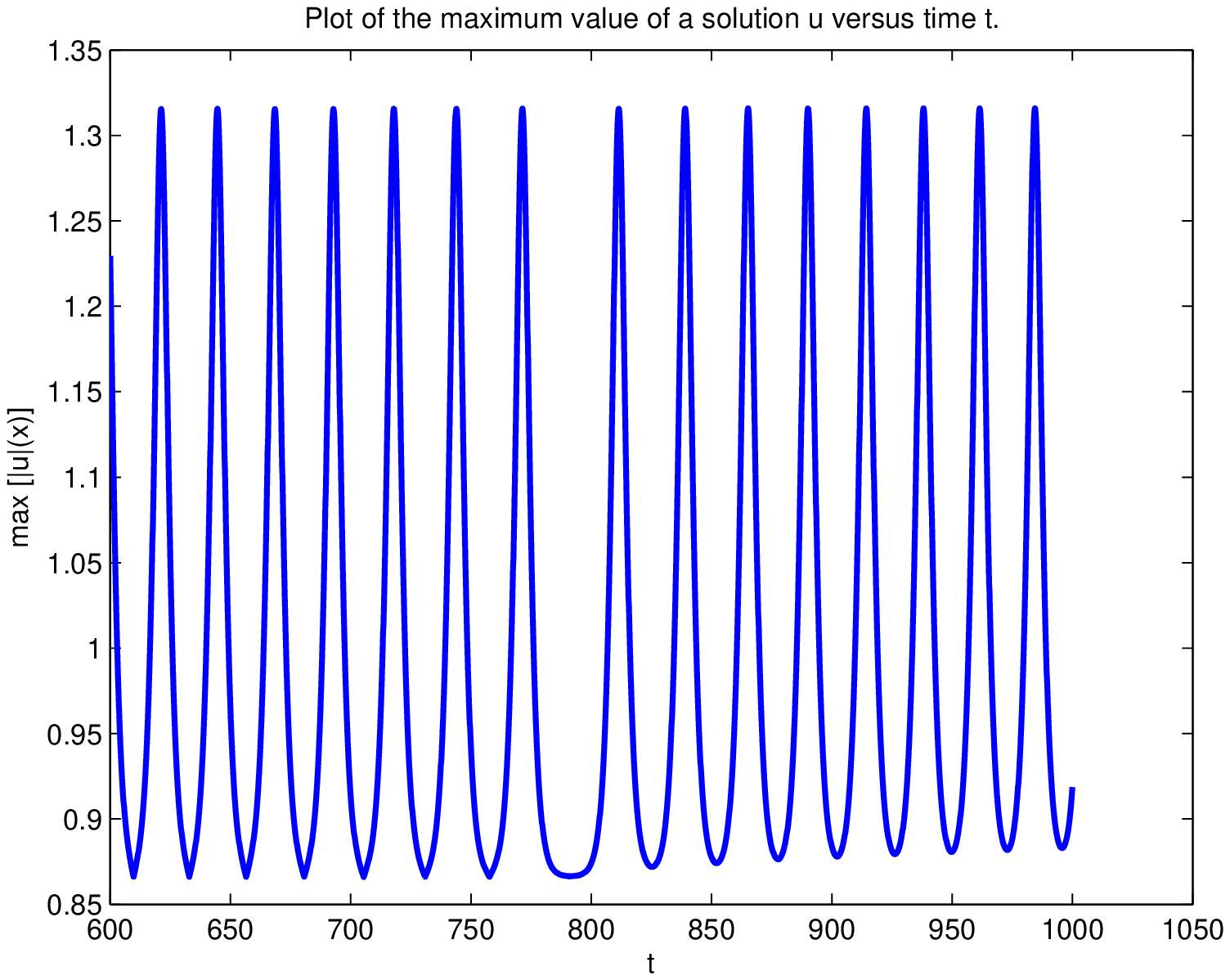}
\caption{A numerical plot of the maximum amplitude of an oscillatory solution decaying to a ground state and in particular the transition region blown up to capture the fast oscillations.}
\label{fig9}
\end{figure}

\begin{figure}
\includegraphics[scale=0.25]{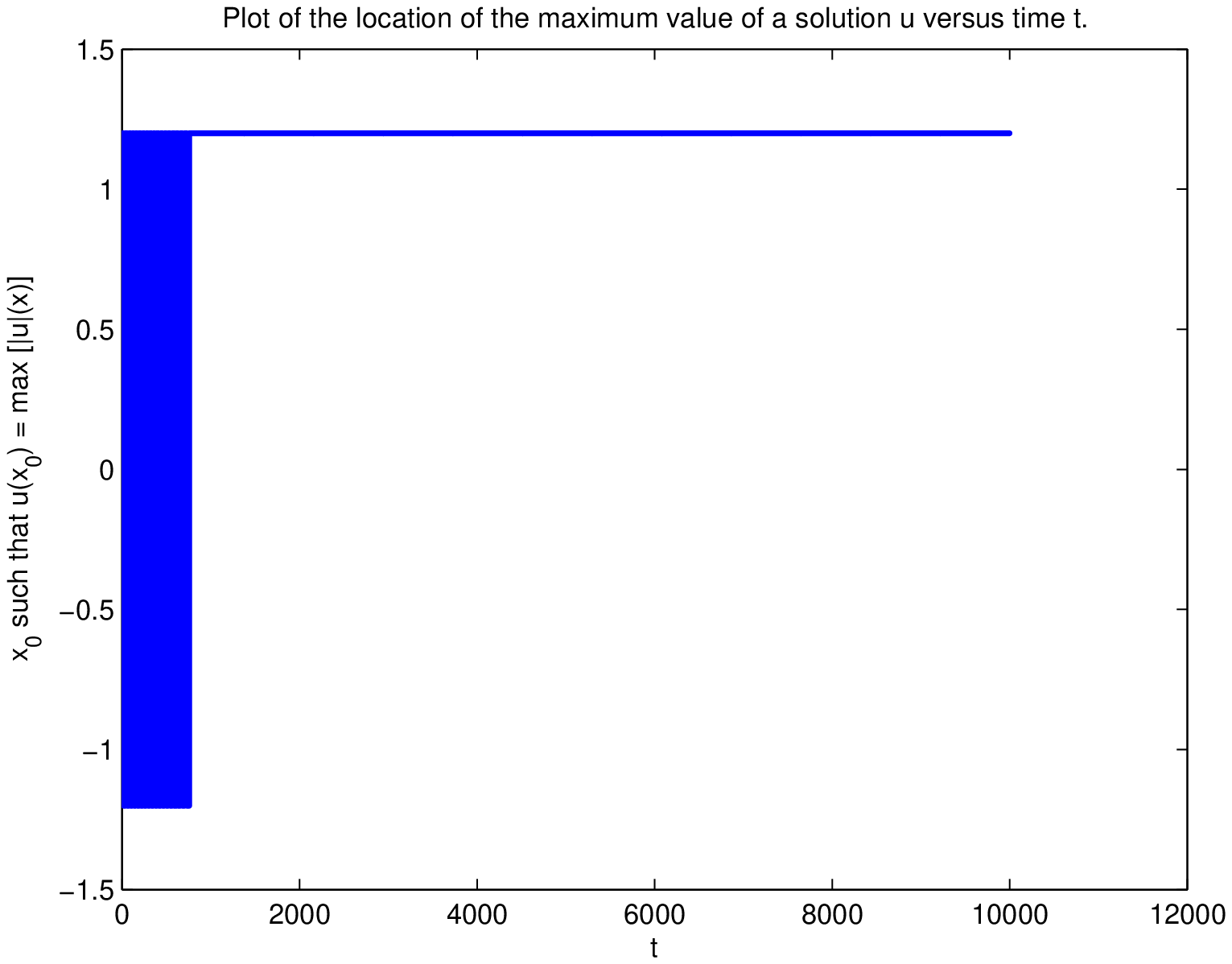} 
\includegraphics[scale=0.25]{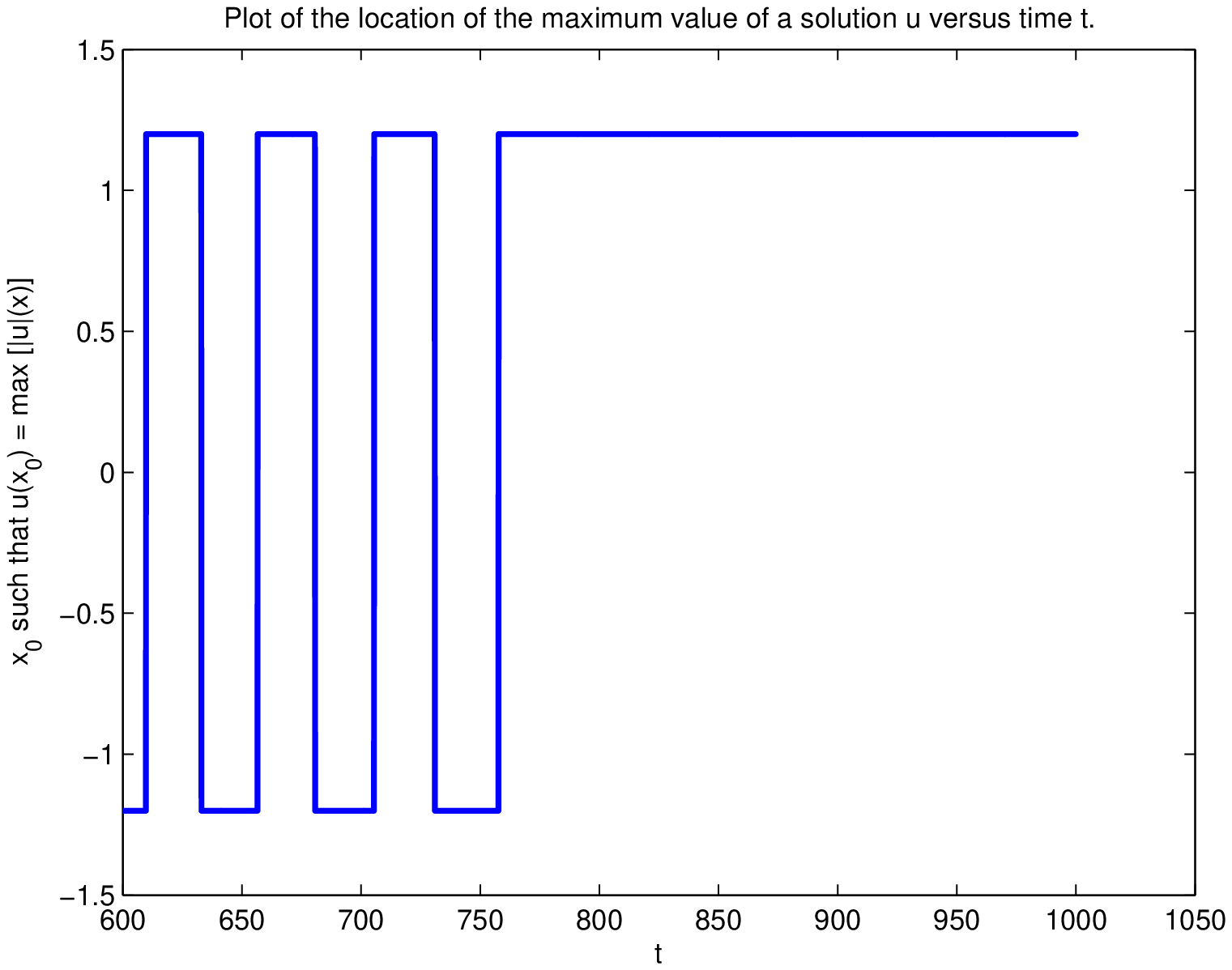}
\caption{A numerical plot of the location of maximum amplitude of an oscillatory solution decaying to a ground state and in particular the transition region blown up to capture the fast oscillations.}
\label{fig10}
\end{figure}

\begin{figure}
\includegraphics[scale=0.25]{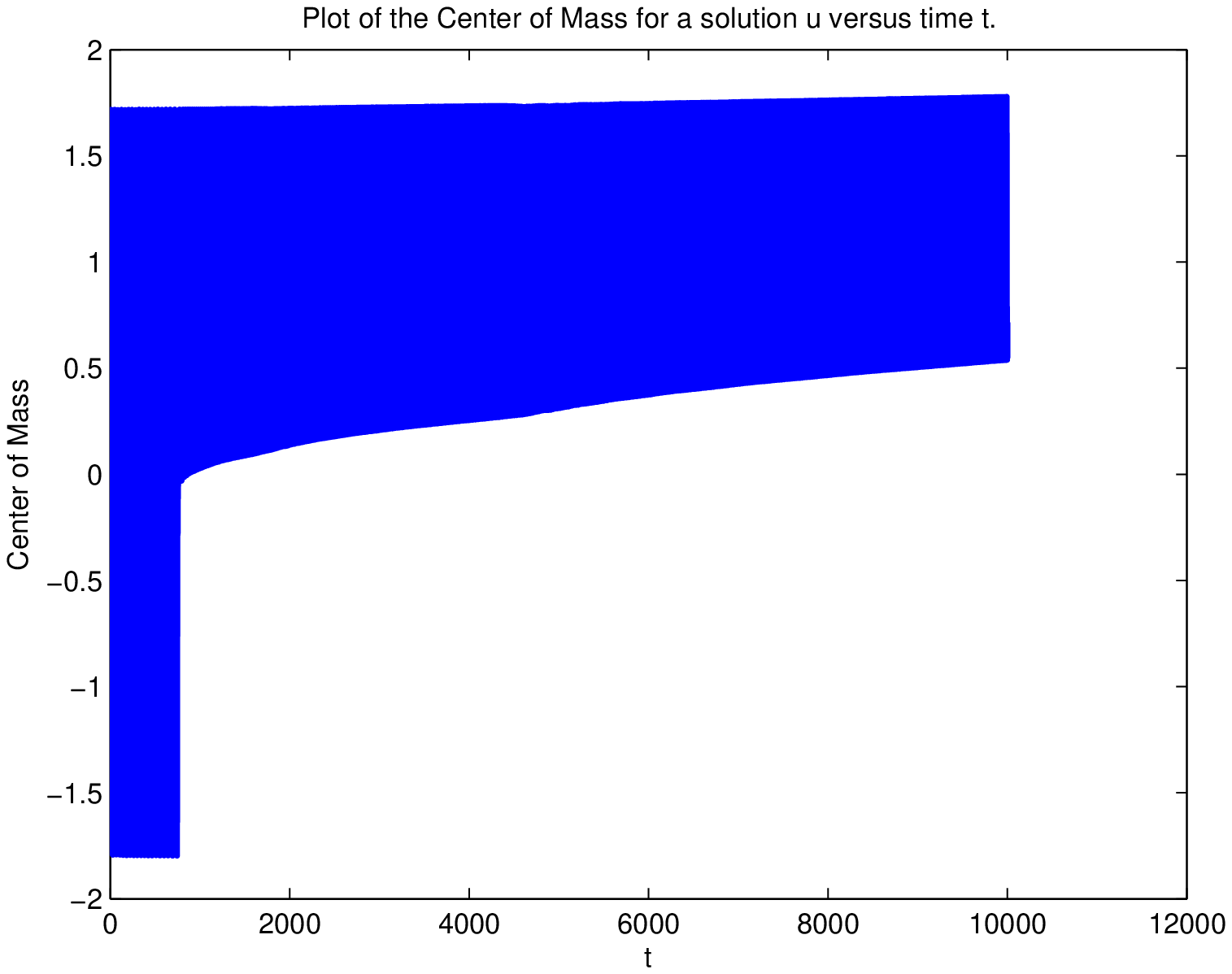} 
\includegraphics[scale=0.25]{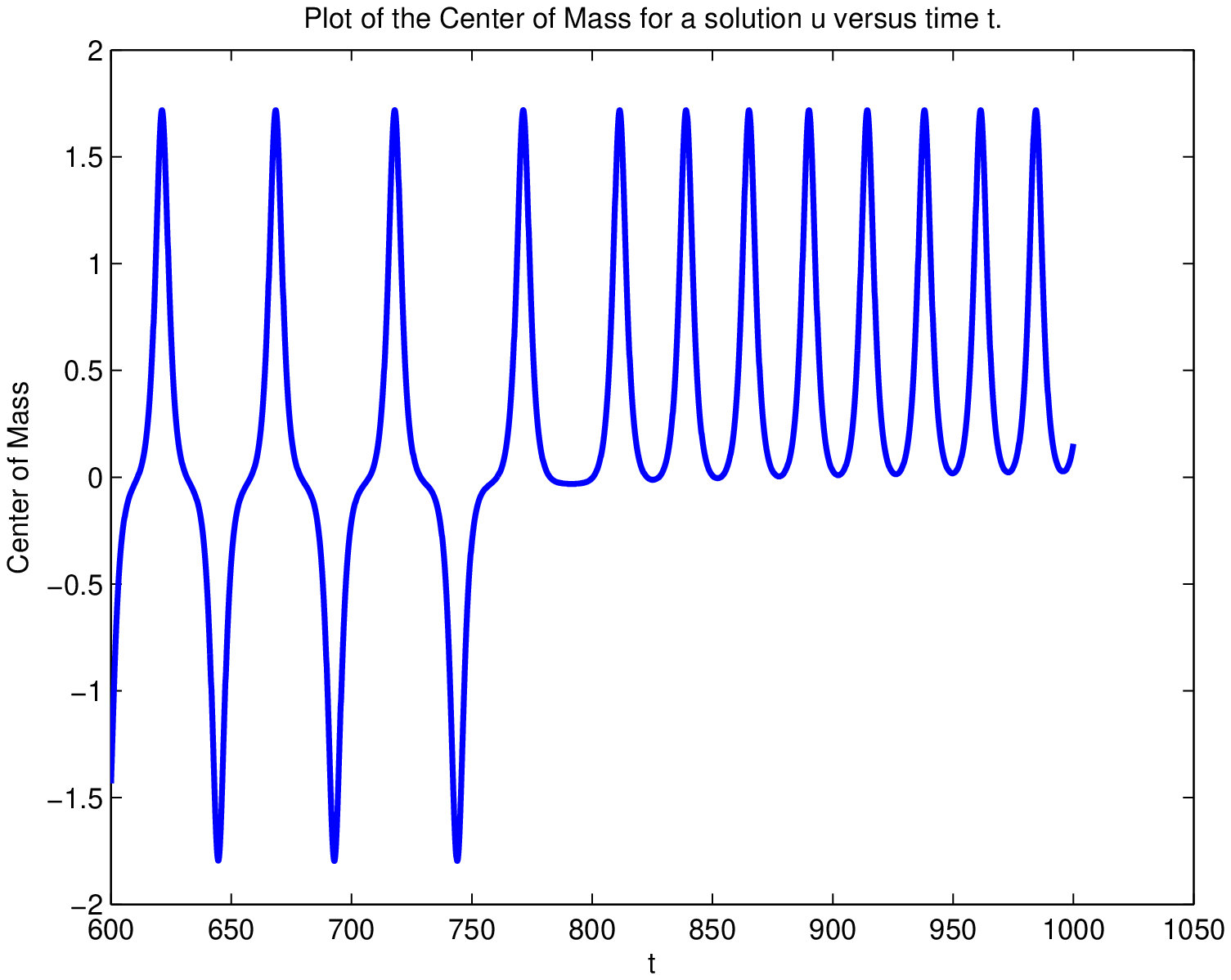} 
\caption{A numerical plot of the center of mass of an oscillatory solution decaying to a ground state and in particular the transition region blown up to capture the fast oscillations.}
\label{fig11}
\end{figure}

\begin{figure}
\includegraphics[scale=0.25]{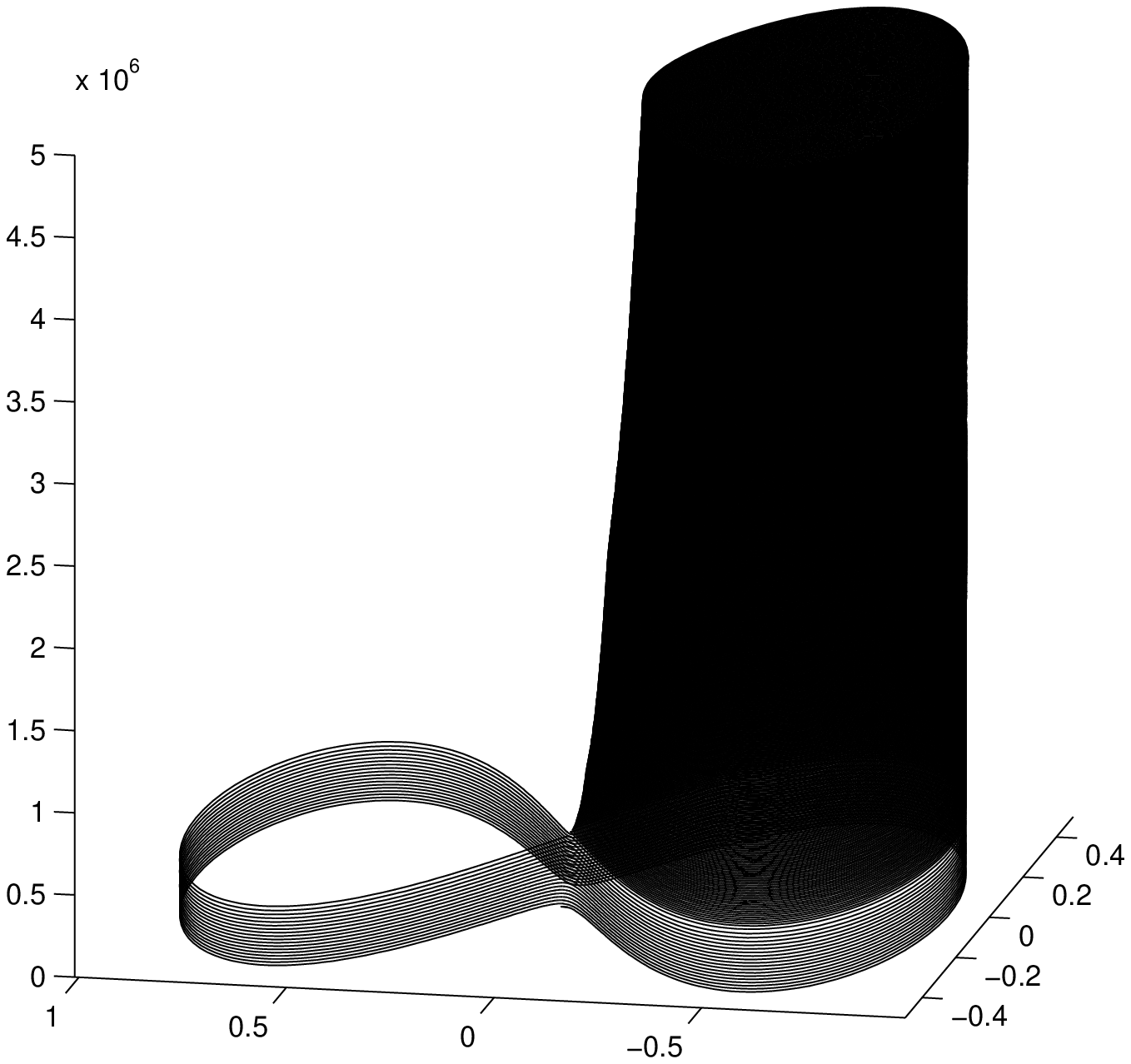} 
\includegraphics[scale=0.25]{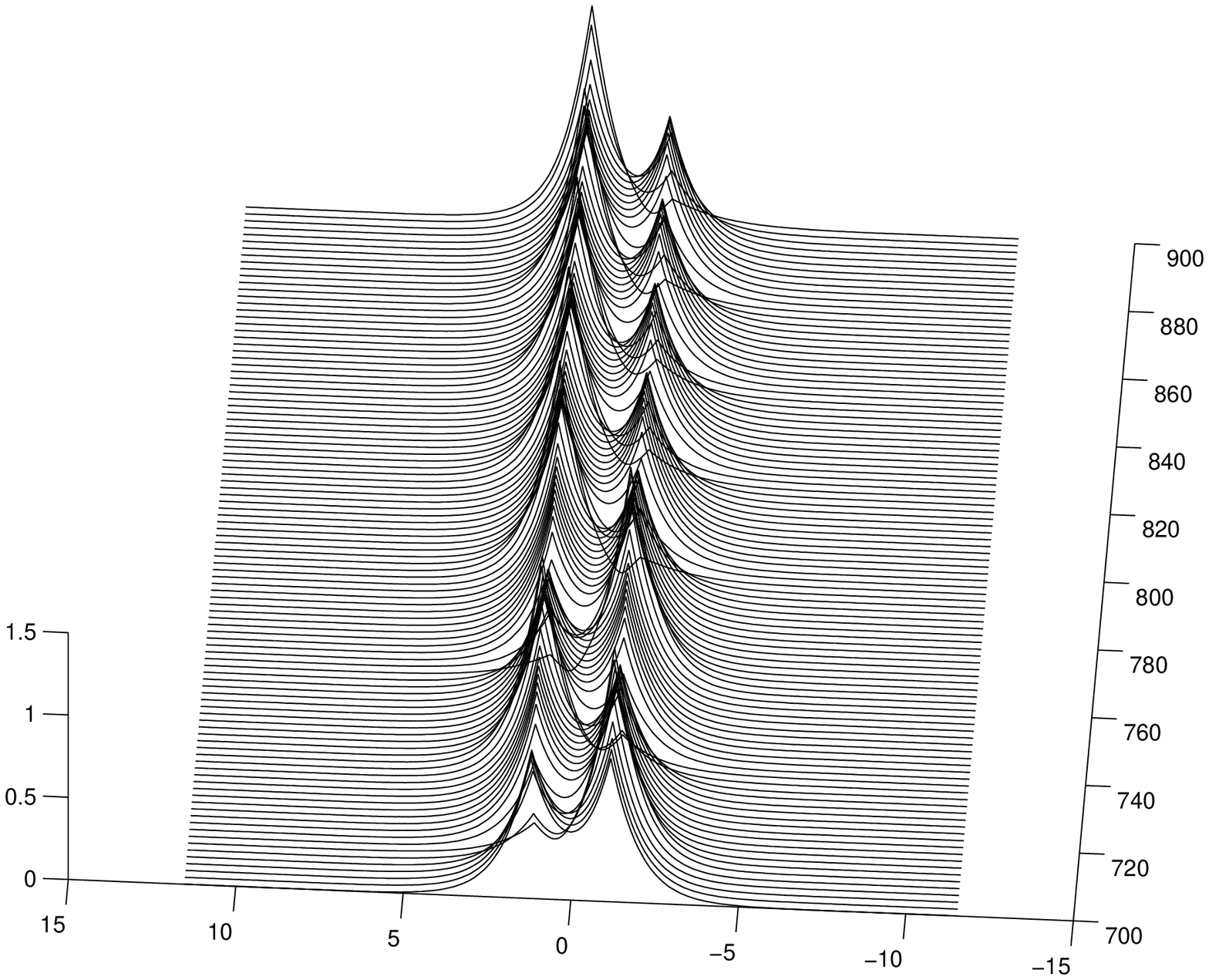}
\caption{A numerical plot of transition across the separatrix for the full infinite dimensional problem in both the projection onto the phase diagram and a plot of the amplitude with respect to time.}
\label{fig12}
\end{figure}

\section{Conclusion and discussion}
\label{sec:conclusion}

The fact that we may observe oscillations between potential wells in systems such as \eqref{eqn:nlsdwp} is not a new result, however we have been able to give a representation of the phenomenon in terms of a classical oscillatory system.  In future work on double well potentials, the authors hope to prove long time stability for oscillations far from the equilibrium point and optimize the time of existence proof by having better control of the damping caused by coupling to the continuous spectrum.  In addition, these pseudo-bound states represent possible solutions resulting from the problem of scattering of solitons across double well potential wells, which in the high velocity limit have been recently studied in \cite{ALS}.  

Finally, the authors would like to point out this question of
oscillation and the resulting dynamical systems becomes more
challenging and interesting as the number of wells is increased, see
\cite{KKC}.  In particular, as one increases the number of wells to
$\infty$, the phase shift required to see oscillation from one well to
the next might give insight into the celebrated Peireles-Nabbaro
barrier for discrete nonlinear Schr\"odinger systems, see
\cite{PeyKru}, \cite{KivCam}, \cite{Kev}.

\appendix

\section{Error Estimates for the Finite Dimensional Ansatz}
\label{sec:errorestfd}

We assume we are near the symmetry breaking equilibrium point, meaning we may take $\alpha(t), \beta(t) \ll A(t)$ and $A(t) > c > 0$.  
Then, we have

\begin{eqnarray*}
(i \dot{A}  - \dot{\theta} A - \Omega_0 A) e^{i \theta} \psi_0 + (i \dot{\alpha} - \dot{\beta} - \dot{\theta} (\alpha + i \beta) - (\alpha + i \beta) \Omega_1 )e^{i \theta} \psi_1 + i R_t - HR - \dot{\theta} R = \\
 - \left[ A^3 \psi_0^3 + (\alpha^2 + \beta^2) (\alpha + i \beta) \psi_1^3 + A (\alpha + i \beta)^2 \psi_1^2 \psi_0 + \right.\\
\left. 2 A (\alpha^2 + \beta^2) \psi_1^2 \psi_0 + A^2 (\alpha - i \beta) \psi_0^2 \psi^1 + 2 A^2 (\alpha + i \beta) \psi_0^2 \psi^1 \right] \\
- \left[ 2 A^2 \psi_0^2 + 4 A(\alpha) \psi_0 \psi_1 + 2 (\alpha^2 + \beta^2) \psi_1^2  \right] R + \left[ A^2 \psi_0^2 + (\alpha + i \beta)^2 \psi_1^2 + 2 A (\alpha + i \beta) \psi_0 \psi_1  \right] \bar{R}  \\
- \left[ A \psi_0 + (\alpha - i \beta) \psi_1 \right] R^2 - \left[ 2 A \psi_0 + 2 (\alpha + i \beta) \psi_1 \right] |R|^2 - |R|^2 R.
\end{eqnarray*}
Hence, 
\begin{eqnarray*}
( i \dot{A}  - \dot{\theta} A - \Omega_0 A + A^3 + A(\alpha + i \beta)^2 + 2 A (\alpha^2 + \beta^2)) = \\
- \left[ A \langle \psi_0^2, R^2 \rangle + (\alpha - i \beta) \langle \psi_0 \psi_1 , R^2 \rangle  \right] - \left[ 2 A \langle \psi_0^2, |R|^2 \rangle + 2 (\alpha + i \beta) \langle \psi_0 \psi_1, |R|^2 \rangle \right] \\
 - \left[ \langle \psi_0, |R|^2 R \rangle \right].
\end{eqnarray*}
As a result, we have
\begin{eqnarray*}
(i \dot{\alpha} - \dot{\beta} - \dot{\theta} (\alpha + i \beta) - \Omega_1 (\alpha + i \beta) + (\alpha^2 + \beta^2 )(\alpha + i \beta) + 2 A^2 (\alpha + i \beta) + A^2 (\alpha - i \beta)) = \\
- \left[ A \langle \psi_0 \psi_1, R^2 \rangle + (\alpha - i \beta) \langle \psi_1^2, R^2 \rangle \right] e^{i 2 \theta} - \left[ 2 A \langle \psi_0 \psi_1, |R|^2 \rangle + 2 (\alpha + i \beta) \langle \psi_1^2 , |R|^2 \rangle \right] \\
 - \left[ \langle \psi_0, |R|^2 R \rangle \right]
\end{eqnarray*}
and
\begin{eqnarray*}
i R_t - H R - \dot{\theta} R = - P_c \left[ A^3 \psi_0^3 + (\alpha^2 + \beta^2) (\alpha + i \beta) \psi_1^3 + A (\alpha + i \beta)^2 \psi_1^2 \psi_0 \right. \\
\left. + 2 A (\alpha^2 + \beta^2) \psi_1^2 \psi_0 + A^2 (\alpha - i \beta) \psi_0^2 \psi^1 + 2 A^2 (\alpha + i \beta) \psi_0^2 \psi^1 \right] \\
- P_c \left[ 2 A^2 \psi_0^2 + 4 A(\alpha) \psi_0 \psi_1 + 2 (\alpha^2 + \beta^2) \psi_1^2  \right] R - P_c \left[ A^2 \psi_0^2 + (\alpha + i \beta)^2 \psi_1^2 + 2 A (\alpha + i \beta) \psi_0 \psi_1  \right] \bar{R}  \\
- P_c \left[ A \psi_0 + (\alpha - i \beta) \psi_1 \right] R^2 - P_c \left[ 2 A \psi_0 + 2 (\alpha + i \beta) \psi_1 \right] |R|^2 - P_c |R|^2 R
\end{eqnarray*}
or
\begin{eqnarray*}
i R_t - H R - \dot{\theta} R = \left[ F_b (A, \alpha, \beta, \theta) + F_R (A, \alpha, \beta, \theta; R, \bar{R}) \right] ,
\end{eqnarray*}
where we have assumed
\begin{eqnarray*}
P_c F_b = F_b, \ P_c F_R = F_R.
\end{eqnarray*}

Let us take $A> 0$.  Then, we see
\begin{eqnarray*}
\dot{A} & = & - 2 \alpha \beta A + \text{Error}_A', \\
\dot{\alpha} & = & [\Omega_1 - (\alpha^2 + \beta^2) - A^2 + \dot{\theta}] \beta + \text{Error}_\alpha', \\
\dot{\beta} & = & - [ \Omega_1 - (\alpha^2 + \beta^2 + A^2) -2 A^2 + \dot{\theta}] \alpha + \text{Error}_\beta', \\
A \dot{\theta} & = & -\Omega_0 A+ A^3 + A (3 \alpha^2 + \beta^2) + \text{Error}_\theta' \\
i R_t - H R - \dot{\theta} R & = &  F_b (A, \alpha, \beta) + F_R (A, \alpha, \beta; R, \bar{R}).
\end{eqnarray*}

Specifically, we have \eqref{eqn:coupledfd1}-\eqref{eqn:coupledid} with
\begin{eqnarray*}
\text{Error}_A (R, \bar{R}, \vec{\alpha}) & = &  \Im  \left( - \langle \left[ 2 A^2 \psi_0^3 + 4 A(\alpha) \psi_0^2 \psi_1 + 2 (\alpha^2 + \beta^2) \psi_1^2 \psi_0   \right], R \rangle \right. \\
& - & \langle \left[ A^2 \psi_0^3 + (\alpha + i \beta)^2 \psi_1^2 \psi_0 + 2 A (\alpha + i \beta) \psi_0^2 \psi_1  \right],  \bar{R} \rangle \\
& - & \left[ A \langle \psi_0^2, R^2 \rangle + (\alpha - i \beta) \langle \psi_0 \psi_1 , R^2 \rangle  \right] \\
& - & \left. \left[ 2 A \langle \psi_0^2, |R|^2 \rangle + 2 (\alpha + i \beta) \langle \psi_0 \psi_1, |R|^2 \rangle \right] + \langle \psi_0 , |R|^2 R \rangle \right) , 
\end{eqnarray*}
\begin{eqnarray*}
\text{Error}_\alpha (R, \bar{R}, \vec{\alpha}) & = & \Im  \left( - \langle \left[ 2 A^2 \psi_0^2 \psi_1 + 4 A(\alpha) \psi_0 \psi_1^2 + 2 (\alpha^2 + \beta^2) \psi_1^3   \right], R \rangle \right. \\
& - & \langle \left[ A^2 \psi_0^2 \psi_1 + (\alpha + i \beta)^2 \psi_1^3 + 2 A (\alpha + i \beta) \psi_0 \psi_1^2  \right],  \bar{R} \rangle \\ 
&-& \left[ A \langle \psi_0 \psi_1, R^2 \rangle + (\alpha - i \beta) \langle \psi_1^2, R^2 \rangle \right]  \\
& - & \left. \left[ 2 A \langle \psi_0 \psi_1, |R|^2 \rangle + 2 (\alpha + i \beta) \langle \psi_1^2 , |R|^2 \rangle \right]  + \langle \psi_1 , |R|^2 R \rangle \right) \\
&-& \beta A^{-1} \Re \left( - \langle \left[ 2 A^2 \psi_0^3 + 4 A(\alpha) \psi_0^2 \psi_1 + 2 (\alpha^2 + \beta^2) \psi_1^2 \psi_0   \right], R \rangle \right. \\
& - & \langle \left[ A^2 \psi_0^3 + (\alpha + i \beta)^2 \psi_1^2 \psi_0 + 2 A (\alpha + i \beta) \psi_0^2 \psi_1  \right],  \bar{R} \rangle \\ 
&-& \left[ A \langle \psi_0^2, R^2 \rangle + (\alpha - i \beta) \langle \psi_0 \psi_1 , R^2 \rangle  \right]  \\
&-& \left. \left[ 2 A \langle \psi_0^2, |R|^2 \rangle + 2 (\alpha + i \beta) \langle \psi_0 \psi_1, |R|^2 \rangle \right] + \langle \psi_0 , |R|^2 R \rangle \right), 
\end{eqnarray*}
\begin{eqnarray*}
\text{Error}_\beta (R, \bar{R}, \vec{\alpha}) & = &  - \Re \left( - \langle \left[ 2 A^2 \psi_0^2 \psi_1 + 4 A(\alpha) \psi_0 \psi_1^2 + 2 (\alpha^2 + \beta^2) \psi_1^3   \right], R \rangle \right. \\
& - & \langle \left[ A^2 \psi_0^2 \psi_1 + (\alpha + i \beta)^2 \psi_1^3 + 2 A (\alpha + i \beta) \psi_0 \psi_1^2  \right],  \bar{R} \rangle \\ 
&-& \left[ A \langle \psi_0 \psi_1, R^2 \rangle + (\alpha - i \beta) \langle \psi_1^2, R^2 \rangle \right]  \\
& - & \left. \left[ 2 A \langle \psi_0 \psi_1, |R|^2 \rangle + 2 (\alpha + i \beta) \langle \psi_1^2 , |R|^2 \rangle \right]  + \langle \psi_1 , |R|^2 R \rangle \right) \\
& - & \alpha A^{-1} \Re \left( - \langle \left[ 2 A^2 \psi_0^3 + 4 A(\alpha) \psi_0^2 \psi_1 + 2 (\alpha^2 + \beta^2) \psi_1^2 \psi_0   \right], R \rangle \right. \\
& - & \langle \left[ A^2 \psi_0^3 + (\alpha + i \beta)^2 \psi_1^2 \psi_0 + 2 A (\alpha + i \beta) \psi_0^2 \psi_1  \right],  \bar{R} \rangle \\ 
&-& \left[ A \langle \psi_0^2, R^2 \rangle + (\alpha - i \beta) \langle \psi_0 \psi_1 , R^2 \rangle  \right]  \\
&-& \left. \left[ 2 A \langle \psi_0^2, |R|^2 \rangle + 2 (\alpha + i \beta) \langle \psi_0 \psi_1, |R|^2 \rangle \right] + \langle \psi_0 , |R|^2 R \rangle \right),
\end{eqnarray*}
and
\begin{eqnarray*}
\text{Error}_\theta (R, \bar{R}, \vec{\alpha}) & = & - A^{-1} \Re \left( - \langle \left[ 2 A^2 \psi_0^3 + 4 A(\alpha) \psi_0^2 \psi_1 + 2 (\alpha^2 + \beta^2) \psi_1^2 \psi_0   \right], R \rangle \right. \\
& - & \langle \left[ A^2 \psi_0^3 + (\alpha + i \beta)^2 \psi_1^2 \psi_0 + 2 A (\alpha + i \beta) \psi_0^2 \psi_1  \right],  \bar{R} \rangle \\
& - & \left[ A \langle \psi_0^2, R^2 \rangle + (\alpha - i \beta) \langle \psi_0 \psi_1 , R^2 \rangle  \right] \\
& - & \left. \left[ 2 A \langle \psi_0^2, |R|^2 \rangle + 2 (\alpha + i \beta) \langle \psi_0 \psi_1, |R|^2 \rangle \right] + \langle \psi_0 , |R|^2 R \rangle \right)  .
\end{eqnarray*}

\section{Dispersive Estimates}
\label{sec:est}

In this section, we follow closely the work \cite{W} on wave operators for Schr\"odigner operators defined on $\RR$.  For proofs and further exposition see \cite{W} and the references contained within.  First of all, let us define $H_0 = -\Delta$ and $H = -\Delta + V$ with the constraints on $V$ to be discussed in the sequel.  

The wave operators, $W_\pm$ are defined by
\begin{eqnarray}
W_\pm = \lim_{t \to \infty} e^{it H} e^{-it H_0}.
\end{eqnarray}
Similarly, their adjoints are defined by
\begin{eqnarray}
W_\pm^* = \lim_{t \to \infty} e^{i t H_0} e^{-it H} P_c,
\end{eqnarray}
where $P_c$ is the projection onto the continuous spectrum of $H$.  The notion of wave operators is intimately related to the idea of distorted Fourier bases, which are discussed in detail in \cite{Ag}, \cite{Ho2}, \cite{RSv4}.  In one dimension, this is directly related to the Jost solutions.  These objects are studied in general in \cite{RSv4} and generalized to even a certain class of non-self-adjoint operators in \cite{KS}.

We define a space $L^1_\gamma$ the space of all complex-valued measurable functions $\phi$ defined on $\RR$ such that
\begin{eqnarray}
\| \phi \|_{L^1_\gamma} = \int | \phi (x) | (1 + |x|)^\gamma dx < \infty .
\end{eqnarray}
Also, take the space $W^{k,p}$ to be the standard Sobolev space defined by having $k$ derivatives bounded in the $L^p$ norm.
Then, we have the following 
\begin{thm}[Weder]
\label{thm:w}
Suppose that $V \in L^1_\gamma$ for $\gamma > \frac{5}{2}$ and that for some $k = 1,2,\dots$, $V^{(l)} \in L^1$ for $l = 0,1,2,\dots,k-1$.  Then $W_\pm$ and $W^*_\pm$ originally defined on $W_{k,p} \cap L^2$, $1 \leq p \leq \infty$, have extensions to bounded operators on $W_{k,p}$, $1 < p < \infty$.  Moreover, there are constants $C_p$ such that:
\begin{eqnarray}
\| W_\pm f \|_{W^{k,p}} \leq C_p \| f \|_{W^{k,p}}, \ \| W^*_\pm f \|_{W^{k,p}} \leq C_p \| f \|_{W^{k,p}}, \ f  \in W^{k,p} \cap L^2, \ 1 < p < \infty.
\end{eqnarray}
\end{thm}

Note, there are specific requirements on the potential $V$ which allow Theorem \ref{thm:w} to be extended to the cases $p = 1$ and $p = \infty$, however we will not discuss them here.  

An important property of wave operators is that for any Borel function $f$, we have
\begin{eqnarray}
f(H) P_c = W_\pm f(H_0) W^*_\pm, \ f(H_0) = W^*_\pm f(H) P_c W_\pm.
\end{eqnarray}
Hence, we have
\begin{eqnarray}
\| e^{i H t} P_c f \|_{L^p} = \| W_\pm e^{i tH_0} W^*_\pm f \|_{L^p}
\end{eqnarray}
and using standard dispersive estimates for the linear Schr\"odinger operator (see for instance \cite{SS} for a concise overview) arrive at
\begin{eqnarray}
\| e^{i H t} P_c f \|_{L^p} \leq C_p t^{-(\frac{1}{2} - \frac{1}{p})} \| f \|_{W^{k,p}}.
\end{eqnarray}
Define a Strichartz pair $(q,r)$ to be admissible if
\begin{eqnarray}
\label{strnum}
\frac{2}{q} = \frac{1}{2} - \frac{1}{r}
\end{eqnarray}
with $2 \leq r < \infty$.  Then, we arrive at the celebrated Strichartz estimates
\begin{eqnarray}
\label{eqn:strich1}
\| e^{iHt} P_c u_0 \|_{L^q W^{k,r}} \lesssim \| u_0 \|_{ W^{k,2}}
\end{eqnarray}
and, using duality techniques and once again the boundedness of the
wave operators, we have
\begin{eqnarray}
\label{eqn:strich2}
\| \int_0^t e^{iH(t-s)} P_c f ds \|_{L^q W^{k,r}} \lesssim \| f(x,t) \|_{L^{\tilde{q}'}_t W^{k,\tilde{r}'}_x},
\end{eqnarray}
where $(q,r)$ and $(\tilde q, \tilde r)$ satisfy \eqref{strnum}.
.

As a side note, using positive commutators and well crafted local smoothing spaces, from \cite{MMT}, we have the full Strichartz estimate
\begin{eqnarray}
\label{eqn:str1}
\| \int_0^t e^{iH(t-s)} P_c f ds \|_{L^\infty L^2} \lesssim \| f(x,t) \|_{L^{\tilde{p}'}_t L^{\tilde{q}'}_x},
\end{eqnarray}
where $(\tilde{p},\tilde{q})$ is any allowable pair as in \eqref{strnum}.
Now, implementing the boundedness of wave operators on $W^{k,p}$ spaces from \cite{W}, we have the following useful relation
\begin{eqnarray}
\| \int_0^t e^{iH(t-s)} P_c f ds \|_{L^\infty H^1} \lesssim \| f(x,t) \|_{L^{\tilde{p}'}_t W^{1,\tilde{q}'}_x},
\end{eqnarray}
where again $(\tilde{p},\tilde{q})$ is a Strichartz pair as in
\eqref{strnum} without first going through the dispersive estimates. 

Note, as mentioned in the introduction the discussion above may be
extended to the case of $V$ having delta function type singularities using formalism discussed in \cite{DMW}.

\end{document}